\documentclass{amsart}
\usepackage{dsfont}
\usepackage[usenames,dvipsnames,svgnames,table]{xcolor}
\usepackage[utf8]{inputenc}
\usepackage[OT2, T1]{fontenc}

\usepackage{ tipa }
\usepackage{helvet}
\usepackage{color}
\usepackage{mathrsfs}  
\usepackage{stmaryrd}  
\usepackage{amsmath,amsxtra,amsthm,amssymb,xr,enumerate,fullpage,comment, graphicx, mathtools}
\usepackage[all]{xy}
\usepackage[bbgreekl]{mathbbol}
\usepackage{bbm}

\DeclareSymbolFont{cyrletters}{OT2}{wncyr}{m}{n}
\DeclareMathSymbol{\Sha}{\mathalpha}{cyrletters}{"58}

\DeclareMathSymbol{\invques}{\mathord}{operators}{`>}
\DeclareUnicodeCharacter{00BF}{\tmquestiondown}
\DeclareRobustCommand{\tmquestiondown}{%
  \ifmmode\invques\else\textquestiondown\fi
}

\usepackage{verbatim}
\numberwithin{equation}{section}

\makeatletter
\newcommand{\mylabel}[2]{#2\def\@currentlabel{#2}\label{#1}}
\makeatother

\newtheorem{theorem}{Theorem}[section]
\newtheorem{lemma}[theorem]{Lemma}
\newtheorem{conj}[theorem]{Conjecture}
\newtheorem{proposition}[theorem]{Proposition}
\newtheorem{corollary}[theorem]{Corollary}
\newtheorem{defn}[theorem]{Definition}
\newtheorem{example}[theorem]{Example}

\newtheorem{remark}[theorem]{Remark}

\newtheorem{assumption}[theorem]{Assumption}

\newcommand{\Gal}{\operatorname{Gal}}

\newcommand{\QQ}{\mathbb{Q}}
\newcommand{\Qp}{{\mathbb{Q}_p}}
\newcommand{\Zp}{\mathbb{Z}_p}
\newcommand{\ZZ}{\mathbb{Z}}

\newcommand{\g}{\mathbf{g}}

\newcommand{\ord}{\mathrm{ord}}

\newcommand{\fp}{\mathfrak{p}}

\newcommand{\cL}{\mathcal{L}}

\newcommand{\cO}{\mathcal{O}}

\newcommand{\GL}{\mathrm{GL}}

\newcommand{\cyc}{\textup{cyc}}

\newcommand{\Hom}{\mathrm{Hom}}

\newcommand{\LL}{\Lambda}

\newcommand{\f}{\textup{\bf f}}

\newcommand{\h}{\textup{\bf h}}

\newcommand{\lra}{\longrightarrow}

\newcommand{\res}{\textup{res}}

\newcommand{\cP}{\mathcal{P}}
\newcommand{\cF}{\mathcal{F}}

\newcommand{\p}{\mathfrak{p}}

\newcommand{\cW}{\mathcal{W}}

\newcommand{\cR}{\mathcal{R}}

\newcommand{\Cp}{\mathbb{C}_p}

\newcommand{\Spf}{{\rm Spf}}
\newcommand{\cl}{{\rm cl}}
\newcommand{\bal}{{\rm bal}}
\newcommand{\wt}{{\rm w}}
\newcommand{\Fr}{{\rm Fr}}

\newcommand{\Ad}{{\rm ad}}
\newcommand{\crit}{{\rm crit}}

\newcommand{\rmw}{{\rm w}}
\newcommand{\cris}{{\rm cris}}

\newcommand{\hatotimes}{{\,\widehat\otimes\,}}

\newcommand{\hf}{\f}
\newcommand{\hg}{\g}
\newcommand{\hh}{\h}
\newcommand{\Q}{\QQ}

\newcommand{\CC}{{\mathbb C}}

\usepackage{hyperref}

\definecolor{pinegreen}{rgb}{0.0, 0.47, 0.44}

 \definecolor{pAlgae}{RGB}{87,115,135}
\definecolor{airforceblue}{rgb}{0.36, 0.54, 0.66}
	\definecolor{bondiblue}{rgb}{0.0, 0.58, 0.71}
\definecolor{britishracinggreen}{rgb}{0.0, 0.26, 0.15}
\definecolor{camouflagegreen}{rgb}{0.47, 0.53, 0.42}
\definecolor{darkcyan}{rgb}{0.0, 0.55, 0.55}

\hypersetup{
    colorlinks = true,
    linkcolor=blue,
     filecolor=blue,
     citecolor = darkcyan,      
     urlcolor=cyan,
    linkbordercolor = {white},
}


\begin{document}

\title{O\lowercase{n the} A\lowercase{rtin formalism for triple product $p$-adic} $L$-\lowercase{functions}}

\author{K\^az\i m B\"uy\"ukboduk}
\address{K\^az\i m B\"uy\"ukboduk\newline UCD School of Mathematics and Statistics\\ University College Dublin\\ Ireland}
\email{kazim.buyukboduk@ucd.ie}

\author{Ryotaro Sakamoto}
\address{Ryotaro Sakamoto \newline Department of Mathematics\\University of Tsukuba\\1-1-1 Tennodai\\Tsukuba\\Ibaraki 305-8571\\Japan}
\email{rsakamoto@math.tsukuba.ac.jp}

\begin{abstract}
Our main objective in the present article is to study the factorisation problem for triple-product $p$-adic $L$-functions, particularly in the scenarios when the defining properties of the $p$-adic $L$-functions involved have no bearing on this problem, although Artin formalism would suggest such a factorisation. Our analysis, which is guided by the ETNC philosophy, recasts this problem as a comparison of diagonal cycles, Beilinson--Kato elements, and Heegner cycles.
\end{abstract}

\maketitle

\tableofcontents

\section{Introduction}

The principal objective of this paper (which is the first in a series of 3 papers) is to study the \emph{irregular} factorisation problem for certain triple product $p$-adic $L$-functions as well as their algebraic counterparts. Here, the adjective \emph{irregular} is used in reference to Perrin-Riou's theory of $p$-adic $L$-functions, where she predicts that $p$-adic $L$-functions should come attached with what she calls regular submodules. The meaning of the adjective \emph{irregular} is explained in \S\ref{subsubsec_intro_comparison_palvannan_dasgupta}, where we discuss how the problem at hand compares with those considered by Dasgupta in \cite{Dasgupta2016} (which is an example of a \emph{regular} factorisation problem).

Let $\hf$ and $\hg$ be cuspidal primitive Hida families of elliptic modular forms and let $\hg^c$ denote the conjugate family of $\hg$. The main tasks we undertake can be summarized as follows.
\begin{itemize}
     \item[\mylabel{item_G1}{\bf Task 1}.] We formulate a precise conjecture to factor the restriction $\cL^{\g}_p(\f\otimes \g \otimes \g^c)_{\vert_{\cW_2}}$ of the $\g$-unbalanced $p$-adic  triple product $L$-function $\cL^{\g}_p(\f\otimes \g \otimes \g^c)$ to a $2$-dimensional weight space $\cW_2$ (that we introduce in \S\ref{subsubsec_30042021_1123} below). This factorisation conjecture reflects the Artin formalism, and addresses the scenario when the interpolation range for $\cL^{\g}_p(\f\otimes \g \otimes \g^c)_{\vert_{\cW_2}}$ is empty.
     \item[\mylabel{item_G2}{\bf Task 2}.] We study the algebraic counterpart of this factorisation problem. This involves an extensive study of relevant Selmer complexes and what we call \emph{modules of leading terms}.
\end{itemize}

Our Conjecture~\ref{conj_main_6_plus_2}, which settles \ref{item_G1}, predicts that the $p$-adic $L$-function $\cL^{\g}_p(\f\otimes \g \otimes \g)^2_{\vert_{\cW_2}}$ factors as the product of a degree-$6$ $p$-adic $L$-function and the image ${\rm Log}_{\omega_\f}({\rm BK}_{\f}^\dagger)$ of Ochiai's big Beilinson--Kato element ${\rm BK}_{\f}^\dagger$ associated to $\f$ under a suitably defined large logarithm map ${\rm Log}_{\omega_\f}$. The latter factor ${\rm Log}_{\omega_\f}({\rm BK}_{\f}^\dagger)$ is closely related to the derivatives of $p$-adic $L$-functions; see \S\ref{BK-PR} where we elaborate on this point. The appearance of the derivatives of $p$-adic $L$-functions (albeit in disguise) in the factorisation of the $p$-adic $L$-function $\cL^{\g}_p(\f\otimes \g \otimes \g)^2_{\vert_{\cW_2}}$ is rather subtle and it is explained by what we call the ``BDP-principle'' (see \S\ref{eqn_2022_05_17_1725}, especially \eqref{item_bdp1} below).

We describe a line of attack to prove Conjecture~\ref{conj_main_6_plus_2} in full generality in Appendix~\ref{subsec_strategy_to_prove_conj_2_2}, through a comparison of the conjectural Gross--Kudla formulae\footnote{In the level of generality required in the present work, this was recently announced by Liu--Yuan--Zhang--Zhang. Since we need this result only in a degenerate case, one may also rely on~\cite{Xue_Hang_2019} (see especially the discussion in the final paragraph of \S1.1 in op. cit.).} for diagonal cycles and Gross--Zagier formulae for Heegner cycles, and record in \S\ref{subsubsec_obstacles} the key missing ingredients to complete this argument.

The algebraic counterpart of the Factorisation Conjecture~\ref{conj_main_6_plus_2}, which is the focus of \ref{item_G2}, is best phrased in terms of what we call the \emph{modules of leading terms}. We introduce and study these objects in \S\ref{sec_Koly_Sys_Dec_4_18_05_2021} in great generality. We remark that the Selmer groups relevant to our factorisation problem will not be all torsion, and as a result,  the algebraic counterpart of Conjecture~\ref{conj_main_6_plus_2} will not involve only the characteristic ideals of Selmer groups. This is unlike the scenario considered in \cite{palvannan_factorization} (see \S\ref{subsubsec_2022_05_21_1103} for a relevant discussion), and it is another reflection of the \emph{irregularity} of the factorisation problem we consider in this paper.  

We establish the algebraic analogue of our Factorisation Conjecture~\ref{conj_main_6_plus_2} in great generality. Our result in this vein is recorded as Theorem~\ref{thm_main_8_4_4_factorisation_intro} (see also Theorem~\ref{thm_main_8_4_4_factorisation} in the main body of our article). It is worth noting that ${\rm Log}_{\omega_\f}({\rm BK}_{\f}^\dagger)$ (which encodes, as we remarked earlier, derivatives of $p$-adic $L$-functions) makes an appearance also in this statement, and it underscores the relevance of the modules of leading terms.

In our companion article~\cite{BC_Part2}, we establish further evidence for the validity of the Factorisation Conjecture~\ref{conj_main_6_plus_2}, proving it fully when $\g$ has CM by an imaginary quadratic field $K$. In future work, we will also treat a similar problem in the context of automorphic form on ${\rm GSp}_4\times {\rm GL}_2 \times {\rm GL}_2$.  

\subsection{Layout}

We provide a brief overview of the contents of our article. We introduce the notation that we frequently use in \S\ref{subsec_the_set_up_intro}. 
After the motivational background, we formulate our actorisation Conjecture \ref{conj_main_6_plus_2} in \S\ref{subsec_Artin_formalism_intro},  and we record our main results in \S\ref{subsec_results_2022_05_22}. We conclude \S\ref{section_intro_results} with a comparison (in \S\ref{subsubsec_intro_comparison_palvannan_dasgupta}) of the main problem that we tackle with those in the previous works \cite{Gross1980Factorization, Dasgupta2016, palvannan_factorization} on related themes.

In \S\ref{section_2_L_functions}, we review various $L$-series and $p$-adic $L$-functions that are relevant to our main problem. We introduce Greenberg-type Selmer complexes in \S\ref{sec_selmer_complexes}, and study Tamagawa factors in \S\ref{subsubsec_2022_09_08_1236} (this is crucial as per our goal that our Selmer complexes are perfect). 

In \S\ref{sec_Koly_Sys_Dec_4_18_05_2021}, we introduce a canonical cyclic submodule of an extended Selmer module, that we call the module of leading terms (of algebraic $p$-adic $L$-functions). The relationship between the modules of leading terms and Kolyvagin systems is explained in Remark~\ref{rmk:comparisonkolyvagin}. Based on these constructions, we formulate and prove the algebraic counterpart of the actorisation Conjecture \ref{conj_main_6_plus_2} in \S\ref{subsec_KS} assuming that $\g$ does not have CM. 

We conclude our article with Appendix~\ref{subsec_strategy_to_prove_conj_2_2}, where we explain that the validity of Conjecture~\ref{conj_main_6_plus_2} can be reduced to the verification of a Generalized Gross--Kudla formula (which is the subject of the work of Liu--Yuan--Zhang--Zhang in progress) and the construction of the $p$-adic $L$-function $\cL_p^\Ad(\hf\otimes \Ad^0\hg)$ (see \S\ref{subsubsec_conj_2022_06_02_0940} for its conjectural description).

\section{Background and results}
\label{section_intro_results}
Before stating our results in detail, we introduce in \S\ref{subsec_the_set_up_intro} the notation that we will rely on throughout this work. After the preliminary discussion in \S\ref{subsubsec221_2022_05_21_1116}--\S\ref{subsubsec_2017_05_17_1800} to serve as a motivation for our actorisation Conjecture~\ref{conj_main_6_plus_2}, we summarize our results in \S\ref{subsec_results_2022_05_22}. We conclude this section with an extensive discussion in \S\ref{subsubsec_intro_comparison_palvannan_dasgupta} to compare the present work with earlier similar results in different settings, underlining the key technical differences and difficulties that arise in our case of interest.

\subsection*{Acknowledgments} Both authors thank the anonymous referee for numerous helpful comments and suggestions, which helped us improve our paper. KB thanks Daniel Disegni, Hang Xue and Wei Zhang for enlightening discussions and helpful exchanges on the status of the Gross--Kudla conjecture. He also thanks Henri Darmon for his encouragement. KB’s research in this publication was conducted with the financial support of Taighde \'{E}ireann -- Research Ireland under Grant number IRCLA/2023/849 (HighCritical). 
RS was supported by JSPS KAKENHI Grant Number 24K16886.

\subsection{The set up}
\label{subsec_the_set_up_intro}
\subsubsection{Hida families}
\label{subsubsec_2022_05_16_1506}
Let $p$ be an odd prime and let $\cO$ be the ring of integers of a finite extension $E$ of $\Qp$. Let us put $[\,\cdot\,]: \ZZ_p^\times \hookrightarrow \LL_{\rm wt}^\times$ for the natural injection, where  $\LL_{\rm wt}=\Zp[[\Zp^\times]]$. The universal weight character $\bbchi$ is defined as the composite map
\[
    \bbchi: G_\QQ\stackrel{\chi_\cyc}{\lra}\ZZ_p^\times\hookrightarrow \LL_{\rm wt}^\times,
\]
where $\chi_\cyc$ is the $p$-adic cyclotomic character. A ring homomorphism $\LL_{\rm wt} \xrightarrow{\nu} \cO$ is called an arithmetic specialisation of weight $k\in \ZZ$ if the compositum 
\[
    G_\QQ\stackrel{\bbchi}{\lra} \LL_{\rm wt}^\times\stackrel{\nu}{\lra}\cO
\]
agrees with $\chi_\cyc^{k}$ on an open subgroup of $G_\QQ$. We regard integers as elements of the weight space via
\begin{align*}
   \ZZ \lra {\rm Hom}_{{\rm cont}}(\Lambda_{\rm wt}, \cO)\,,\qquad\qquad 
    n \longmapsto (\nu_n:[x] \mapsto x^n)
\end{align*}
For any integer $k$, let us define $\LL_{\rm wt}^{(k)}\cong \LL(1+p\ZZ_p) := \mathbb{Z}_p[[1+p\mathbb{Z}_p]]$ as the ``weight $k$''-eigen-subspace of $\LL_{\rm wt}^{(k)}$, that is, 
\[
\LL_{\rm wt}^{(k)} := \{x \in \LL_{\rm wt} \mid [\xi] \cdot x = \xi^k x \,\, \textrm{ for any $\xi \in \mu_{p-1}(\mathbb{Z}_p^{\times})$} \}. 
\]
Note that the map $\LL_{\rm wt} \stackrel{\nu_{k} }{\lra} \ZZ_p$ factors through $\LL_{\rm wt}^{(k)}$. 
For a positive integer $N$ coprime to $p$, let $\mathbf{T}^{\rm ord}(N)$ denote Hida's universal ordinary Hecke algebra of tame level $N$. We let $\f=\sum_{n=1}^{\infty} \mathbb{a}_{n}(\f)q^n \in \cR_\f[[q]]$ denote the primitive Hida family of tame conductor $N_\f$ and tame nebentype character $\varepsilon_\f$\,, which admits a crystalline specialisation $f_{\circ}$ of weight $k$. We recall that a primitive Hida family uniquely corresponds to an irreducible component ${\rm Spec}(\cR_\f)$ of ${\rm Spec}(\mathbf{T}^{\rm ord}(N_\f))$, whose arithmetic specialisations are new away from $p$ (cf. \cite{Hsieh}, \S3.1). Here, $\cR_\f$ is the branch (i.e. the irreducible component) of Hida's universal ordinary Hecke algebra determined by $f_{\circ}$. It is finite flat over $\LL_{\rm wt}^{(k)}$ and one has $a_p(\hf) \in \mathcal \cR_\hf^\times$. 
The universal weight character $\bbchi$ gives rise to the character
\[
    \bbchi_{\f}:  G_\QQ\xrightarrow{\bbchi} \LL_{\rm wt}^\times \twoheadrightarrow \LL_{\rm wt}^{(k),\times}\lra \cR_\f^\times\,.
\]
In general, for any $\kappa\in \cW_\hf :=\Spf(\cR_\f)(\CC_p)$, let us write ${\rm wt}(\kappa)\in \Spf(\LL_{\rm wt})(\CC_p)$ for the point that $\kappa$ lies over and call it the weight character of $\kappa$, so that ${\rm wt}(\kappa)\circ \bbchi =\kappa\circ \bbchi_{\f}$. We call $\cW_\hf$ the weight space for the Hida family $\hf$.

We say that $\kappa \in \cW_\hf$ is an arithmetic (or classical)  specialisation of $\hf$ whenever the restriction ${\rm wt}(\kappa)$ of $\kappa$ to $\LL_{\rm wt}^{(k)}$ agrees with $\nu_{{\rm w}(\kappa)}$ on an open subgroup of $\ZZ_p^\times$ for some integer ${\rm w}(\kappa)\in \ZZ_{\geq 2}$. We denote by $\cW_\hf^{\rm cl}\subset \cW_\hf$ the set of arithmetic specialisations. We say that an arithmetic specialisation $\kappa$ is crystalline if ${\rm w}(\kappa)\equiv k \pmod{p-1}$. 

By the fundamental work of Hida, we have a big Galois representation
$G_{\Q,\Sigma} \xrightarrow{\rho_\hf} \GL_2(\mathrm{Frac}(\cR_\hf))$ attached to $\f$, where $\Sigma$ is a finite set of primes containing all those dividing $pN_\f\infty$. The Galois representation $\rho_\hf$ is characterized by the property that
\[
    \mathrm{Tr}\, \rho_\hf(\Fr_\ell) = a_\ell(\hf), \qquad \forall \ell\not\in\Sigma.
\]
Here, $\Fr_\ell$ denotes the arithmetic Frobenius. Let us denote by $T_\hf\subset \mathrm{Frac}(\cR_\hf)^{\oplus 2}$ the Ohta lattice (cf. \cite{ohta99,ohta00}, see also \cite{KLZ2} where our $T_\f$ corresponds to $M(\f)^*$ in op. cit.) that realizes the Galois representation $\rho_\hf$ in \'etale cohomology groups of a tower of modular curves. If we assume in addition that 
    \begin{itemize}
       \item[\mylabel{item_Irr}{\bf (Irr)}]  the residual representation $\bar{\rho}_\hf$ is absolutely irreducible
    \end{itemize}
then any $G_\QQ$-stable lattice in $ \mathrm{Frac}(\cR_\hf)^{\oplus 2}$ is homothetic to $T_\f$. We henceforth assume that \ref{item_Irr} holds true for all Hida families that appear in this work. By a theorem of Wiles, we have
\[
    \rho_\hf|_{G_{\Qp}} \simeq \begin{pmatrix} \delta_\f & * \\ 0 & \alpha_\hf\end{pmatrix} 
\]
where $\alpha_\hf: G_{\QQ_p}\to \cR_\f^\times$ is the unramified character for which we have $\alpha_\hf(\Fr_p)=a_p(\f)$ and 
$\delta_\f:=\bbchi_{\f}\, \chi_\cyc^{-1}\, \alpha_\hf^{-1}\,\varepsilon_\f$. If we assume that
    \begin{itemize}
         \item[\mylabel{item_Dist}{\bf (Dist)}] $\delta_\f \not\equiv \alpha_\hf \pmod{\mathfrak{m}_{\cR_\f}}$
    \end{itemize}
then the lattice $T_\f$ fits in an exact sequence
\begin{equation} \label{eqn:filtrationf}
    0\lra T_\hf^+ \lra T_\hf \lra T_\hf^- \lra 0 
\end{equation} 
of $\cR_\f[[G_{\Qp}]]$-modules, where the action on $ T_\hf^+$ (resp. $ T_\hf^-$) is given by $\delta_\f$ (resp. $\alpha_\f$). 

We henceforth assume that \ref{item_Irr} and \ref{item_Dist} hold true for all Hida families that appear in this work.

\subsubsection{Self-dual triple products}
\label{subsubsec_211_2022_06_01_1635}
Let $\f$, $\g$, and $\h$ be three primitive Hida families of ordinary $p$-stabilized newforms of tame levels $N_\f, N_\g, N_\h$ (as in \S\ref{subsubsec_2022_05_16_1506}), whose tame nebentype characters verify the following self-duality condition:
\begin{equation}
\label{eqn_self_duality_condition}
    \varepsilon_\f\varepsilon_\g\varepsilon_\h = 1.
\end{equation}
Put $T_3 := T_\f\,\widehat\otimes\,_{\ZZ_p} T_\g\,\widehat\otimes\,_{\ZZ_p} T_\h$. Then $T_3$ is a Galois representation of rank $8$ over the complete local Noetherian ring $\cR_3 := \cR_\f \,\widehat\otimes\,_{\ZZ_p} \cR_\g \,\widehat\otimes\,_{\ZZ_p} \cR_\h$ of Krull-dimension $4$. We consider the associated weight space 
$$\cW_3 := {\rm Spf}(\cR_3)(\Cp):=\cW_\hf\times \cW_\hg \times \cW_\hh.$$ 
The self-duality condition \eqref{eqn_self_duality_condition} ensures that we have a perfect $G_\QQ$-equivariant pairing
\[
   T_3\otimes T_3\lra \bbchi_{\f\g\h}\chi_\cyc^{-{3}},
\]
where $\bbchi_{\f\g\h}:=\bbchi_{\f}\otimes \bbchi_{\g} \otimes \bbchi_{\h}: G_\Q\to \cR_3^{\times}$. 
Since $p>2$ by assumption, there exists a unique character $\bbchi_{\f\g\h}^{\frac{1}{2}}:G_\Q\to \cR_3^{\times}$ with $(\bbchi_{\f\g\h}^{\frac{1}{2}})^2=\bbchi_{\f\g\h}$. Then the Galois representation
$T_3^\dagger := T_3\otimes \bbchi_{\f\g\h}^{-\frac{1}{2}}\chi_\cyc^2$ is self-dual, in the sense that $T_3^\dagger \cong \Hom_{\cR_3}(T_3^\dagger, \cR_3)(1)$. 


\subsubsection{L-functions and periods}
\label{subsubsec_213_2022_05_17_1455}
For a crystalline specialisation
$$x=(\kappa,\lambda,\mu)\in \cW_3^{\rm cl}:=\cW_\hf^\cl\times\cW_\hg^\cl\times\cW_\hh^\cl$$ 
of weight $({\rm w}(\kappa), {\rm w}(\lambda), {\rm w}(\mu))$, consider the triple 
\[
	f:=\hf_\kappa^\circ \in S_{{\rm w}(\kappa)}(N_f,\chi_f),\qquad g:=\hg_\lambda^\circ \in S_{{\rm w}(\lambda)}(N_g,\chi_g),\qquad h:=\hh_\mu^\circ \in S_{{\rm w}(\mu)}(N_h,\chi_h)\,,
\]
whose ordinary $p$-stabilisations are $(\hf_\kappa,\hg_\lambda,\hh_\mu)$. Let us write 
\[
\mathscr{M}(x):=\mathscr{M}(f\otimes g\otimes h):=\mathscr{M}(f) \otimes \mathscr{M}(g)\otimes \mathscr{M}(h)
\]
for the motive associated to $f\otimes g\otimes h$, where $\mathscr{M}(?)$ (for $?=f,g,h$) stands for Scholl's motive, whose self-dual twist is $\mathscr{M}^\dagger(x):= \mathscr{M}(x)(c(x))$, with $$c(x):=\frac{{\rm w}(\kappa)+{\rm w}(\lambda)+{\rm w}(\mu)}{2}-1\,.$$
The $L$-function $L(f\otimes g\otimes h,s)$ associated with the motive $\mathscr{M}(x)$ is known to have an analytic continuation and under the self-duality assumption \eqref{eqn_self_duality_condition},  it satisfies a functional equation of the form
\begin{equation} \label{eqn:functionalequationtriple}
    \Lambda(f\otimes g \otimes h, s) \doteq \varepsilon(f\otimes g \otimes h) \cdot N(f\otimes g\otimes h)^{-s} \cdot \Lambda(f\otimes g\otimes h, 2c(x)-s),
\end{equation}
where $\Lambda(f\otimes g\otimes h,s)$ denotes the completed $L$-function, with the $\Gamma$-factors $L_\infty(f\otimes g\otimes h,s)$, and where $N(f\otimes g\otimes h)\in \ZZ_{\geq 1}$, and $\varepsilon(f\otimes g \otimes h)\in \{\pm 1\}$.

The study of the critical values of $L(f\otimes g\otimes h,s)$  (in the sense of Deligne) shows that the set of classical points can be subdivided into the following four regions, according to the choices of periods:
\begin{align*}
	\cW_3^\bal &:= \{ x\in \cW_\hf^\cl\times\cW_\hg^\cl\times\cW_\hh^\cl \mid {\rm w}(\kappa)+{\rm w}(\lambda)+{\rm w}(\mu) > 2\max\{{\rm w}(\kappa), {\rm w}(\lambda), {\rm w}(\mu)\} \}  \rightsquigarrow  \, \Omega^\bal \sim \langle f,f\rangle \langle g,g\rangle \langle h,h\rangle, 
 \\
	\cW_3^\hf &:= \{ x\in \cW_\hf^\cl\times\cW_\hg^\cl\times\cW_\hh^\cl \mid {\rm w}(\kappa)+{\rm w}(\lambda)+{\rm w}(\mu) \leq 2{\rm w}(\kappa)\} \hspace{2.9cm} \rightsquigarrow \, \Omega^f \sim \langle f,f\rangle^2, 
 \end{align*}
 and similarly define $\cW_3^\hg $ and $\cW_3^\hh$.
\subsubsection{$p$-adic Garrett--Rankin $L$-functions} Under certain technical assumptions (cf. \S\ref{subsec:tripleproducts}), Hsieh has constructed four $p$-adic $L$-functions
\[
    \cL_p^\bullet(\hf\otimes\hg\otimes \hh) : \cW_3\lra \Cp, \qquad  \bullet\in \{\bal,\hf,\hg,\hh\}.
\]
These are uniquely determined by an interpolation property of the following form: For all $(\kappa,\lambda,\mu)\in\cW_3^\bullet\cap \cW_3^{\cl}$, we have
\begin{equation}
\label{eqn_2022_05_16_1610}
    \cL_p^\bullet(\hf\otimes\hg\otimes \hh)(\kappa,\lambda,\mu) \doteq \frac{\Lambda(f\otimes g\otimes h, c(x))}{\Omega^\bullet}, \qquad    \bullet\in \{\bal, \hf,\hg,\hh\}\,.
\end{equation}
 where ``$\dot{=}$'' means equality up to interpolation factors which we do not specify here; see Theorems \ref{thm:unbalancedinterpolation} and \ref{thm:balancedinterpolation} for the precise formulae.

\subsubsection{Root numbers and the identical vanishing of $p$-adic $L$-functions} 
Let us define $N:={\rm LCM}(N_\f, N_\g, N_\h)$. According to Deligne, the root number $\varepsilon(f\otimes g \otimes h)$ determines the parity of $\ord_{s=c}L(f\otimes g\otimes h,s)$ at the central critical value. The root number $\varepsilon(f\otimes g \otimes h)$ can be expressed as a product of local root numbers: 
$$\varepsilon(f\otimes g \otimes h)=\prod_{v\mid N\infty} \varepsilon_v(f\otimes g \otimes h), \qquad \varepsilon_v(f\otimes g \otimes h)\in \{\pm 1\}.$$ 
The non-archimedean local root numbers are constant in families. In more precise terms, $\prod_{v\mid N} \varepsilon_v(f\otimes g \otimes h)$ is constant as $f$, $g$ and $h$ vary in families thanks to the \emph{rigidity of automorphic types}, cf. \cite[Lemma 2.14]{FouquetOchiai} and \cite[Corollary 5.3.3]{Disegni_local_constants}, see also the discussion on Page 414, Remark 3.1 and Definition 3.9 of \cite{Hsieh}. On the other hand, the root number $\varepsilon_\infty(f\otimes g \otimes h)$ at infinity is determined as follows:
\[
    \varepsilon_\infty(f\otimes g\otimes h) = 
    \begin{cases}
        -1, & \text{if } (\kappa,\lambda,\mu)\in \cW_3^\bal, \\
        +1, & \text{otherwise}.
    \end{cases}
\]
As a result, 
\begin{equation}
\label{eqn_2022_05_16_1612}
\begin{aligned}
 \varepsilon(f\otimes g \otimes h)= 
 \begin{cases}
    +1&  \hbox{if } (\kappa,\lambda,\mu)\in \cW_3^\bal \hbox{ and }  \prod_{v\mid N} \varepsilon_v(f\otimes g \otimes h)=-1\,\\
   & \hbox{or if } (\kappa,\lambda,\mu)\notin \cW_3^\bal \hbox{ and }  \prod_{v\mid N} \varepsilon_v(f\otimes g \otimes h)=+1\,,\\\\
    -1&  \hbox{if } (\kappa,\lambda,\mu)\in \cW_3^\bal \hbox{ and }  \prod_{v\mid N} \varepsilon_v(f\otimes g \otimes h)=+1\,\\
    & \hbox{or if } (\kappa,\lambda,\mu)\notin \cW_3^\bal \hbox{ and }  \prod_{v\mid N} \varepsilon_v(f\otimes g \otimes h)=-1\,.\\
\end{cases}
\end{aligned}
\end{equation}
Combining the interpolation formula \eqref{eqn_2022_05_16_1610} with \eqref{eqn_2022_05_16_1612}, we infer that
\begin{align}
    \label{eqn_2022_05_16_1613}
    \begin{aligned}
     \cL_p^\bal(\hf\otimes\hg\otimes \hh)=0& \qquad \hbox{ if } \prod_{v\mid N} \varepsilon_v(f\otimes g \otimes h)=+1\,,\\
      \cL_p^\bullet(\hf\otimes\hg\otimes \hh)=0, \quad \bullet\in \{\hf,\hg,\hh\}& \qquad \hbox{ if } \prod_{v\mid N} \varepsilon_v(f\otimes g \otimes h)=-1\,.
    \end{aligned}
\end{align}
In particular, the four $p$-adic $L$-functions never simultaneously assume non-zero values at crystalline specialisations.


\subsubsection{}
\label{subsubsec_216_2022_06_01_1635} In the present paper, we are solely interested in the setting when 
\begin{equation}
\label{eqn_2022_05_16_1626}
\varepsilon(f\otimes g \otimes h)=-1 \quad  \hbox{ for } (\kappa,\lambda,\mu)\in \cW_3^\bal\,.
\end{equation}
We remark that this is precisely the scenario of \cite{BSV, DarmonRotger}, where the authors have constructed cycles to account for the systematic vanishing \eqref{eqn_2022_05_16_1613} at the central critical point.


\subsubsection{}
\label{subsubsec_30042021_1123}
Our main results concern the case when $\hh = \hg^c$, where $\hg^c=\hg\otimes \varepsilon_\hg^{-1}$ is the conjugate Hida family. In this case, note that our running self-duality assumption \eqref{eqn_self_duality_condition} forces $\varepsilon_\f$ to be trivial. 

Since $p>2$, and as $\det(\rho_\f)=\bbchi_{\f}\chi_\cyc^{-1}$, we define the central critical twist of $T_\f$ on setting $T_\f^\dagger:=T_\f\otimes \bbchi_\f^{-\frac{1}{2}}\chi_\cyc$. Whenever $\hh = \hg^c$, we will identify the ring $\cR_\hh$ with $\cR_\hg$, through which we shall treat a specialisation $\lambda \in \cW_\hg$ of $\cR_{\hg}$ also as a specialisation of  $\cR_\hh$. Note in this case that the specialisation $\hg^c_\lambda$ of the Hida family $\hh=\hg^c$ is the conjugate of the overconvergent eigenform $\g_\lambda$: Namely, $\hg^c_\lambda=\g_\lambda\otimes \varepsilon_{\hg}^{-1}\,.$

The Galois representation $M_2^\dagger := T_\hf^\dagger\hatotimes\,_{\ZZ_p} \Ad^0(T_\hg)$, where $\Ad^0$ denotes the trace-zero endomorphisms of $T_\hg$, is then self-dual as well. We note that $M_2^\dagger$ is a Galois representation of rank-$6$ over the complete local Noetherian ring 
$\cR_2 := \cR_\hf\hatotimes_{\ZZ_p} \cR_\hg$ of Krull-dimension $3$.
Let us consider the associated weight space $\cW_2 := \Spf(\cR_2)(\Cp)$. We decompose the set of classical points in the weight space $\cW_2^{\rm cl}\subset \cW_2$ as follows:
\begin{equation}
\label{eqn_2022_05_17_1552}
     \cW_2^{\rm cl} = \underbrace{\left\lbrace (\kappa,\lambda) \in \cW_2^{\cl}: 2{\rm w}(\lambda) > {\rm w}(\kappa) \right\rbrace}_{\cW_2^{\Ad}}\quad \bigsqcup\quad\underbrace{\left\lbrace (\kappa,\lambda) \in \cW_2^{\cl}: 2{\rm w}(\lambda) \leq {\rm w}(\kappa) \right\rbrace}_{\cW_2^{\hf}}\,.
\end{equation}

\subsubsection{} 
Let us consider the natural homomorphisms
\begin{align}
\begin{aligned}
     \label{eqn_2022_05_17_1410}
     \cR_3&\xrightarrow{\iota_{2,3}^*} \cR_2& \qquad\qquad \cR_\hf&\xhookrightarrow{\varpi_{2,1}^*} \cR_2\\
    a\otimes b\otimes c &\longmapsto a\otimes bc &\qquad\qquad a&\longmapsto a\otimes 1
\end{aligned}
\end{align}
of complete local Noetherian $\ZZ_p$-algebras, which induce the morphisms
\begin{align}
\begin{aligned}
     \label{eqn_2022_05_17_1421}
     \cW_2&\xrightarrow{\iota_{2,3}} \cW_3 &\qquad\qquad \cW_2&\stackrel{\varpi_{2,1}}{\twoheadrightarrow} \cW_{\hf}\\
    (\kappa,\lambda) &\longmapsto (\kappa,\lambda,\lambda)&\qquad\qquad  (\kappa,\lambda)&\longmapsto \kappa\,.\\
\end{aligned}
\end{align}
Let us put $T_2^\dagger:=\iota_{2,3}^* \bigl( T_3^\dagger\bigr)$. We then have a split exact sequence
\begin{equation} \label{eqn_factorisationrepresentations_tr}
    0\lra M_2^\dagger\xrightarrow{\iota_{\rm tr}} T_2^\dagger\xrightarrow{{\rm id}\otimes {\rm tr}} \varpi^*_{2,1}(T_\hf^\dagger)\lra 0\,,
\end{equation}
where ${\rm tr}: \Ad(T_\g)\to \cR_\g$ is the trace map (which is $G_\Q$-equivariant when we endow the target with trivial Galois action). The self-duality of $M_2^\dagger$, $T_2^\dagger$ and $T_\f^\dagger$ gives rise to the exact sequence 
\begin{equation} \label{eqn_factorisationrepresentations_dual_trace}
    0\lra \varpi^*_{2,1}(T_\hf^\dagger) \xrightarrow{{\rm id}\otimes {\rm tr}^*} T_2^\dagger\xrightarrow{\pi_{{\rm tr}^*}} M_2^\dagger\lra 0\,, 
\end{equation}
where ${\rm tr}^*$ is given by transposing the trace map and using the self-duality of $\Ad(T_\g)$:
$${\rm tr}^*: \cR_\g \xrightarrow{{}^{\rm t}{\rm tr}} \Ad(T_\g)^*\xrightarrow{\sim}\Ad(T_\g)\,. $$
Both exact sequences \eqref{eqn_factorisationrepresentations_tr} and \eqref{eqn_factorisationrepresentations_dual_trace} are split:
\begin{equation} \label{eqn_factorisationrepresentations_dual_trace_splitting}
\xymatrix{
    0\ar[r]& \varpi^*_{2,1}(T_\hf^\dagger) \ar[r]^(.63){{\rm id}\otimes {\rm tr}^*} & T_2^\dagger \ar[r]^{\pi_{{\rm tr}^*}} \ar@/_-1.0pc/[l]^(.43){{\rm id}\otimes {\rm tr}}& M_2^\dagger\ar[r] \ar@/_-1.0pc/[l]^(.46){\iota_{\rm tr}}& 0\,.}
\end{equation}

\subsection{The Artin formalism} 
\label{subsec_Artin_formalism_intro}
Until otherwise stated, we shall work in the setting of \S\ref{subsubsec_30042021_1123}, so that $\hh=\hg^c$ is the conjugate Hida family of $\hg$. Recall in this scenario that we identify $\cW_{\hh}$ with $\cW_{\hg}$.
\subsubsection{} 
\label{subsubsec221_2022_05_21_1116}
Let $x=(\kappa,\lambda,\lambda)\in \cW_3^{\cl}$ be a classical point and let $f=\hf_\kappa$, $g=\hg_\lambda$ and $g^c=\hg^c_\lambda$ denote the corresponding specialisations. The self-dual motive $\mathscr{M}^\dagger(x)$ attached to $f\otimes g \otimes h$ (cf. \S\ref{subsubsec_213_2022_05_17_1455}) then decomposes:
\begin{equation}
\label{eqn_2022_05_17_1544}
    \mathscr{M}^\dagger(x)=(\mathscr{M}^\dagger(f)\otimes {\rm ad}^0\mathscr{M}(g))\oplus \mathscr{M}^\dagger(f)\,.
\end{equation}
The Artin formalism applied to this decomposition shows that the complex analytic Garrett--Rankin $L$-function naturally factors as
\begin{equation} \label{eqn:complexfactorisation}
    L(f\otimes g \otimes g^c,s) = L(f\otimes \Ad^0(g),s-{\rm w}(\lambda)+1) \cdot L(f,s-{\rm w}(\lambda)+1).
\end{equation}
At the central critical point ${\rm w}(\kappa)/2+{\rm w}(\lambda)-1$, the factorisation \eqref{eqn:complexfactorisation} reads
\begin{equation}
    \label{eqn_2022_05_16_1424}
    L(f\otimes g \otimes g^c,{\rm w}(\kappa)/2+{\rm w}(\lambda)-1) = L(f\otimes \Ad^0(g),{\rm w}(\kappa)/2) \cdot L(f,{\rm w}(\kappa)/2)\,.
\end{equation}
 \subsubsection{$p$-adic Artin formalism: $p$-adic $L$-functions of families of degree 6 and degree 2 motives.}
 \label{subsubsec_intro_padic_artin_formalism_1}
Our main goal in the present work is to establish $p$-adic analogues of the factorisation~\eqref{eqn_2022_05_16_1424}.  Recall that there are four $p$-adic $L$-functions one may consider, given by an interpolation formula in the range $\cW_3^\bullet$, where $\bullet\in\{\f,\g,\h,{\rm bal}\}$. Therefore, the $p$-adic Artin formalism in our set-up amounts to the factorisation each one of these four $p$-adic $L$-functions, into the product of two $p$-adic $L$-functions.

We briefly discuss the nature of these two $p$-adic $L$-functions to motivate the reader for the main problem at hand and for our results, even though our main results do not dwell on the constructions of these $p$-adic $L$-functions. The $p$-adic $L$-function that is relevant to the second summand in \eqref{eqn_2022_05_17_1544} (as $f$ varies in the Hida family $\hf$) is the Mazur--Kitagawa $p$-adic $L$-function $\cL_p^{\rm Kit}(\hf)=\cL_p^{\rm Kit}(\hf)(\kappa,1+\pmb{\sigma})\,,$ where $\pmb{\sigma}$ denotes the cyclotomic variable. The putative $p$-adic $L$-functions that correspond to the first summand in \eqref{eqn_2022_05_17_1544} (as both $f$ and $g$ vary in the respective Hida families) are the $p$-adic $L$-functions $\cL_p^?(\hf\otimes\Ad^0\hg)$ that interpolates the algebraic parts of when $(\kappa,\lambda)\in \cW_2^?$ for $?=\Ad,\f$ (cf.  \eqref{eqn_2022_05_17_1552}). We refer the reader to \S\ref{sec:Adjoint} for details. The $p$-adic $L$-functions $\cL_p^?(\hf\otimes\Ad^0\hg)$ are presently unavailable in full generality; but as we have noted above, we do not rely on the existence of these $p$-adic $L$-functions in our main results. Since we work in the situation of \eqref{eqn_2022_05_16_1626}, we have
$\cL_p^{\bal}(\f\otimes\g\otimes\g^c)=0$ owing to its interpolative properties (cf. \eqref{eqn_2022_05_16_1610}) and the functional equation \eqref{eqn:functionalequationtriple}. 

\subsubsection{Root numbers}
 It follows from \eqref{eqn_2022_05_16_1612} and \eqref{eqn_2022_05_16_1626} that $\varepsilon(\hf_\kappa\otimes \Ad^0(\hg_\lambda))=\varepsilon(\hf)$ (resp. $-\varepsilon(\hf)$) if resp. $(\kappa,\lambda)\in \cW_2^\hf$ (resp. $(\kappa,\lambda)\in \cW_2^\Ad$), where $\varepsilon(\hf)$ is the constant value of the root number $\varepsilon(\hf_\kappa)$ as $\kappa\in \cW_\hf^{\cl}$ varies. 
 
\subsubsection{} In view of the interpolation formulae for the Mazur--Kitagawa $p$-adic $L$-function and the same for $\cL_p^?(\hf\otimes\Ad^0\hg)$, we infer that either $\cL_p^\Ad(\hf\otimes\Ad^0\hg)=0$ (if $\varepsilon(\hf)=+1$)\,, or else $\cL_p^{\rm Kit}(\hf)(\kappa,{\rm w}(\kappa)/2)=0$ identically (if $\varepsilon(\hf)=-1$). In all cases, we have the rather uninteresting factorisation 
$$ \cL_p^{\bal}(\f\otimes\g\otimes\g^c)^2(\kappa,\lambda,\lambda)=0=\cL_p^\Ad(\hf\otimes\Ad^0\hg)\cdot L_p^{\rm Kit}(\hf)(\kappa,{\rm w}(\kappa)/2)$$
of the balanced $p$-adic $L$-function.
   
\subsubsection{} 
\label{subsubsec_2017_05_17_1800}

The factorisation of the $\f$-dominant $p$-adic $L$-function 
\begin{equation}
\label{eqn_2022_05_17_1725}
\cL_p^{\hf}(\f\otimes\g\otimes\g^c)^2(\kappa,\lambda,\lambda)\,\dot{=}\,\cL_p^\hf(\hf\otimes\Ad^0\hg)\cdot \cL_p^{\rm Kit}(\hf)(\kappa,{\rm w}(\kappa)/2)
\end{equation}
reduces, once the $p$-adic $L$-function $\cL_p^\hf(\hf\otimes\Ad^0\hg)$ with the expected interpolative properties is constructed, to a comparison of periods at the specialisations of both sides at $(\kappa,\lambda)\in \cW_2^{\hf}$. Here, ``$\dot{=}$'' means an equality up to a non-zero correction factor that interpolates the ratios of periods that are involved in the interpolation formulae of the $p$-adic $L$-functions that appear on two sides of \eqref{eqn_2022_05_17_1725}.

We are left to consider the factorisation problem for $\cL_p^{\hg}(\f\otimes\g\otimes\g^c)(\kappa,\lambda,\lambda)$, which is the most challenging case. To indicate the source of the difficulty, we note that
 $$\iota_{2,3}(\cW_2)\cap \cW_3^\hg =\emptyset$$ 
 and therefore, the interpolation range for $\iota_{2,3}^*\circ \cL_p^\g(\f\otimes\g\otimes\g^c)$ is empty. We will address this problem in the scenario when $\varepsilon(\hf)=-1$, so that 
 \begin{equation}
 \label{eqn_2017_05_17_1803}
     \varepsilon(\hf_\kappa\otimes \Ad^0(\hg_\lambda))= \begin{cases}
         +1& \hbox{ if } (\kappa,\lambda)\in \cW_2^\Ad\,,\\
         -1& \hbox{ if } (\kappa,\lambda)\in \cW_2^\hf\,.
     \end{cases} 
 \end{equation}
Differentiating both sides of \eqref{eqn:complexfactorisation}, we obtain
\begin{equation} \label{eqn:derivativefactorisation}
    \frac{d}{ds}L(\hf_\kappa\otimes \hg_\lambda \otimes \hg_\lambda^c,s)\big{\vert}_{s= c(x)} =  L(\hf_\kappa\otimes \Ad^0(\hg_\lambda),s)\big{\vert}_{s=\frac{{\rm w}(\kappa)}{2}} \cdot \frac{d}{ds} L(\hf_\kappa,s)\big{\vert}_{s=\frac{{\rm w}(\kappa)}{2}}\,.
\end{equation} 
Inspired by the fundamental results of \cite{bertolinidarmonprasanna13, BSV, DarmonRotger} (and the $p$-adic Beilinson philosophy based on Perrin-Riou's conjectures), we adopt the guiding principle that 
\begin{itemize}
    \item[\mylabel{item_bdp1}{\bf BDP})] the $p$-adic $L$-function $\cL_p^{\hg}(\f\otimes\g\otimes\g^c)(\kappa,\lambda,\lambda)$ should be thought of as a $p$-adic avatar of the family of derivatives $$\left\{\frac{d}{ds}L(\hf_\kappa\otimes \hg_\lambda\ \otimes \hg_\lambda^c,s)\big{\vert}_{s=c(x)}\,:(\kappa,\lambda)\in \cW_2^{\f}\right\}\,.$$ 
\end{itemize}
Based on this principle, we propose the following as a $p$-adic variant of \eqref{eqn:derivativefactorisation} in families. 

\begin{conj}
\label{conj_main_6_plus_2}
Suppose that $\varepsilon(\hf)=-1$ and \eqref{eqn_2022_05_16_1626} holds true. Let ${\rm Log}_{\omega_\f}({\rm BK}_\f^\dagger)$ denote the logarithm of the big Beilinson--Kato class (see \S\ref{subsubsec_leading_terms_bis_3} for relevant definitions). We then have the following factorisation of $p$-adic L-functions:
\[
    \cL_p^\hg(\hf\otimes\hg\otimes\hg^c)^2(\kappa,\lambda,\lambda) = \mathscr C(\kappa)\cdot \cL_p^\Ad(\hf\otimes \Ad^0\hg)(\kappa,\lambda) \cdot {\rm Log}_{\omega_\f}({\rm BK}_{\f}^\dagger)\,,
\]
where $\mathscr C \in \cR_\f[\frac{1}{p}]$ is as in \S\ref{subsubsec_2024_11_09_1543}.
\end{conj}

Conjecture~\ref{conj_main_6_plus_2} is a natural variant of the factorisation problem Dasgupta studied in \cite{Dasgupta2016} for the $p$-adic Rankin--Selberg $p$-adic $L$-functions, generalising Gross' method that he has developed in \cite{Gross1980Factorization} to study the same problem for Katz $p$-adic $L$-functions. There are, however, key technical differences, and these are detailed in \S\ref{subsubsec_intro_comparison_palvannan_dasgupta}. 

 See \S\ref{BK-PR} for a brief explanation as to how ${\rm Log}_{\omega_\f}({\rm BK}_\f^\dagger)$ encodes information about the derivatives of $p$-adic $L$-functions. This tells us that Conjecture~\ref{conj_main_6_plus_2} is in concurrence with the principle \eqref{item_bdp1}.

\subsubsection{}
\label{subsubsec_2024_11_09_1543}
The fudge factor $\mathscr C(\kappa)$ (where $\kappa\in \cW_\f$ is an arithmetic specialisation) that appears in Conjecture~\ref{conj_main_6_plus_2} is explicitly defined  as follows:  
\begin{align}
    \label{eqn_2024_11_9_1238}
    \begin{aligned}
        \mathscr C(\kappa):= \dfrac{w_\f \,\cdot  \mathfrak C_{\rm exc}(\hf\otimes\hg_K)\cdot \langle v_{g_{\rm Eis}}^-, \eta_{g_{\rm Eis}}\rangle}{2\cdot (-2\sqrt{-1})^{\wt(\kappa)-1} \cdot \langle \omega_{g_{\rm Eis}}, v_{g_{\rm Eis}}^+\rangle}\times \dfrac{\omega_K h_K  (1-p^{-1})  \log_p(u)}{|D_K|\cdot C_{\hf_\kappa}^+C_{\hf_\kappa}^-\cdot \mathcal E(\hf_\kappa^\circ,\Ad)}\,.
    \end{aligned}
\end{align}
Here, $K$ is an imaginary quadratic field satisfying the Heegner hypothesis relative to $\f$, and, 
\begin{itemize}
    \item $w_\f$ is the Atkin--Lehner pseudo-eigenvalue for the family $\f$, so that $w_\hf(\kappa)^2 = (-N_\hf)^{2-\wt(\kappa)}$;
      \item $\mathfrak C_{\rm exc}(\hf\otimes\hg_K)=\prod_
{q\in \in \Sigma_{\rm exc}}(1+q^{-1})$ where $\Sigma_{\rm exc}=\Sigma_{\rm exc}(\hf\otimes \hg)$ is the set of primes determined by the local properties of $\f$ and $\g$ (cf. \cite{ChenHsieh2020}, \S 1.5);
    \item $\mathbf{g}_K$ is the canonical CM Hida family and $g_{\rm Eis}$ is its unique weight-one $p$-stabilisation as in \cite[\S4.2]{BDV} (where these are denoted by $\g$ and $g$, respectively)\,;
    \item $\eta_{g_{\rm Eis}}$ and $\omega_{g_{\rm Eis}}$ is as in \cite[\S2.3.3.1]{BDV}, and $v_{g_{\rm Eis}}^\pm$  is defined in \cite[\S4.2]{BDV} (on Page 33); 
    \item $h_K$ is the class number of $K$, $\omega_K=|\cO_K^\times|$, and $u\in \mathcal O_K[\frac{1}{p}]^\times$ is such that $(u) = \p^{h_K}$; 
    \item $C_{\f_\kappa}^{\pm}$ is a $p$-adic period (cf. \cite{Ochiai2006}, Theorem 6.7 and Remark 6.8).
\end{itemize}

\subsection{Results}
\label{subsec_results_2022_05_22}
While Conjecture~\ref{conj_main_6_plus_2} in its most general form appears to be out of reach (for reasons we explain in \S\ref{subsubsec_obstacles}), we may prove the validity of its algebraic counterpart (concerning the relevant Selmer complexes). This amounts to a factorisation of the modules of leading terms (under mild hypotheses) that we introduce in \S\ref{sec_Koly_Sys_Dec_4_18_05_2021}.

We denote by $\delta(T,\Delta)$ the module of algebraic $p$-adic $L$-functions for $(T,\Delta)\in \{(M_2^\dagger,{\rm tr}^*\Delta_\g), (T_2^\dagger,\Delta_\g)\}$, given as in \S\ref{subsubsec_leading_terms_bis_2} and \S\ref{subsubsec_leading_terms_bis_1}, respectively. We also let ${\rm BK}_\f^\dagger\in H^1(G_{\QQ},T_\f^\dagger)$ the big Beilinson--Kato class constructed by Ochiai in \cite{Ochiai2006}.

\begin{theorem}[Corollary~\ref{cor_2026_05_11} below]
\label{thm_main_8_4_4_factorisation_intro}
Under the hypotheses recorded in \S\ref{subsubsec_hypo_section_6}, we have
\begin{equation}
\label{eqn_2022_09_13_1733_intro}
\delta(T_2^\dagger,\Delta_\g)\,{=}\,  {\rm Exp}_{F^-T_\f^\dagger}^*(\delta(M_2^\dagger,{\rm tr}^*\Delta_{\g}))\cdot \varpi_{2,1}^*{\rm Log}_{\omega_\f}({\rm BK}_\f^\dagger).
\end{equation}
\end{theorem}

\subsection{Comparison with earlier work}
\label{subsubsec_intro_comparison_palvannan_dasgupta}

This subsection is dedicated to a comparison of the setting of the present work with that in \cite{Dasgupta2016}\footnote{As well as its variant for algebraic $p$-adic $L$-functions in \cite{palvannan_factorization}}), highlighting the key technical differences.

\subsubsection{} Dasgupta in \cite{Dasgupta2016} relies on the factorisation 
\begin{equation}
    \label{eqn_2022_05_18_1037}
    L'(f\otimes f\otimes\psi\beta^{-1},1)=L({\rm Sym}^2f\otimes\psi\beta^{-1},1)L'(\eta,0)
\end{equation}
where we tacitly follow the notation of \S1.4 of opt. cit. (in particular, $f$ denotes here an eigenform of weight $2$). The main technical ingredients in \cite{Dasgupta2016} are the following: 
\begin{itemize}
  \item[\mylabel{item_Das1}{\bf D1})] Beilinson regulator formula expressing $L'(f\otimes f\otimes\psi\beta^{-1},1)$ in terms of a \emph{Beilinson--Flach} unit $b_{f,\psi,\beta}$,
   \item[\mylabel{item_Das2}{\bf D2})]  Class number formula expressing $L'(\eta,0)$ in terms of a circular unit $u_\eta$.
\end{itemize}
Both $b_{f,\psi,\beta}$ and $u_\eta$ live in the same one-dimensional vector space, and the input above combined with \eqref{eqn_2022_05_18_1037} yields an explicit comparison of $b_{f,\psi,\beta}$ and $u_\eta$, which involves the critical value $L({\rm Sym}^2f\otimes\psi\beta^{-1},1)$. 

Note that $L(f\otimes f\otimes\psi\beta^{-1},s)$ has no critical values. In particular, it is not critical at $s=1$. The $p$-adic Beilinson conjectures of Perrin-Riou predicts in this case that the $p$-adic $L$-value $L_p(f,f,\psi,\nu_{2,\beta})$ (where we still rely on the notation of \cite[\S1.4]{Dasgupta2016}) is related to the derivative $L'(f\otimes f\otimes\psi\beta^{-1},1)$ of the Rankin--Selberg $L$-series:
\begin{equation}
    \label{eqn_2022_05_20_1035}
    L_p(f,f,\psi,\nu_{2,\beta}) \xleftrightarrow{\hbox{ $p$-adic Beilinson philosophy }} L'(f\otimes f\otimes\psi\beta^{-1},1)\,.
\end{equation}
The relationship \eqref{eqn_2022_05_20_1035} takes a more concrete form via \cite[Equation 18]{Dasgupta2016} (which crucially relies on the work of \cite{KLZ2})
\begin{equation}
    \label{eqn_2022_05_20_1207}
    L_p(f,f,\psi,\nu_{2,\beta})=\log_p(b_{f,\psi,\beta})
\end{equation}
combined with Step \eqref{item_Das1}.

Our departure point is the factorisation \eqref{eqn:derivativefactorisation} in the scenario when $(\kappa,\lambda)\in \cW_2^{\rm bal}$. Let us fix crystalline specialisations $f$ and $g$ of $\hf$ and $\hg$, respectively, of respective weights $k$ and $l$ with $2l < k$, so that $(\kappa,\lambda)\in \cW_2^{\f}$. Note that $c(x) =\frac{k}{2}+l-1$ is the central critical point of the Garrett--Rankin $L$-series $L(f\otimes g\otimes g^c,s)$ and in fact, 
$L(f\otimes g\otimes g^c,s)=0$ in the scenario we have placed ourselves.

In very rough terms, our guiding principle is the expected relation
\begin{align}
\begin{aligned}
 \label{eqn_2022_05_20_1056}
 \lim_{\substack{\hf_\kappa\to f\\ 
 \hg_\lambda \to g\\ (\kappa,\lambda)\in \cW_2^{\hg}}}L(\hf_\kappa\otimes \hg_\lambda\otimes \hg_\lambda^c,{\rm w}(\kappa)/2+{\rm w}(\lambda)-1)^{\rm alg}\,\dot{=}\,\cL_p^{\hg}(f\otimes g \otimes g^c)&\xleftrightarrow{\ref{item_bdp1}}L'(f\otimes g\otimes g^c, c(x))\,,
\end{aligned}
\end{align}
inspired by the main results of \cite{bertolinidarmonprasanna13, BSV, DarmonRotger} (and also paralleling \eqref{eqn_2022_05_20_1035}), where the superscript ``alg'' denotes an appropriately defined algebraic part of the indicated $L$-value.    

Note that $s=c(x)$ is the central critical point, unlike the value $s=1$ in \eqref{eqn_2022_05_18_1037}, which is a non-critical point. The expected \eqref{item_bdp1} in \eqref{eqn_2022_05_20_1056} amounts to a combination of a) the reciprocity laws of \cite{BSV, DarmonRotger} relating diagonal cycles to $\cL_p^{\hg}(f\otimes g \otimes g^c)$, and b) the Gross--Kudla conjecture expressing $L'(f\otimes g\otimes g^c,c(x))$ in terms of the Bloch--Beilinson heights of the same cycles. Moreover, in view of the factorisation \eqref{eqn:derivativefactorisation}, $L'(f\otimes g\otimes g^c,c(x))$ can be also computed in terms of the heights of Heegner cycles. In other words, the Artin formalism for the $p$-adic $L$-function $\cL_p^{\hg}(f\otimes g \otimes g^c)$ with an empty range of interpolation should amount to a comparison of a pair of Arithmetic GGP conjectures (rather than a comparison of Beilinson's conjectures, as in the case of \cite{Gross1980Factorization, Dasgupta2016}): One for ${\rm GL}_2\times {\rm GL}_2 \times {\rm GL}_2$ in a rather degenerate case (that boils down to Gross--Kudla conjecture), and another for ${\rm GL}_2$ (which boils down to Gross--Zagier--Zhang formulae). We refer the reader to \S\ref{subsec_strategy_to_prove_conj_2_2} and \cite{BCPV} for an elaboration on this point of view.


\subsubsection{}
\label{subsubsec_2022_05_21_1103}
In the present subsection, we discuss the comparison of the factorisation we consider to that studied in \cite{Dasgupta2016} with the perspective offered by Perrin-Riou's theory of $p$-adic $L$-functions. 

Given a pure motive of motivic weight $w(M)\leq -2$, let us denote by $V$ its $p$-adic realisation. Let $H^1_{\rm f}(G_{\QQ},V)$ denote the Bloch--Kato Selmer group attached to $V$ and assume that the natural map 
$$\res_p:\, H^1_{\rm f}(G_{\QQ},V)\lra H^1_{\rm f}(G_{\QQ_p},V)$$
is injective. In the setting we have placed ourselves, this assumption conjecturally always holds. We say that a submodule  $D\subset {\bf D}_{\rm cris}(V)$ is regular if 
\begin{itemize}
    \item $D\cap {\rm Fil}^0{\bf D}_{\rm cris}(V)=\{0\}$\,,
    \item the natural map
    $$D\,\,\oplus\,\, \log_V\,\circ \,\res_p(H^1_{\rm f}(G_{\QQ},V))\lra {\bf D}_{\rm cris}(V)/ {\rm Fil}^0{\bf D}_{\rm cris}(V)=:t_V(\QQ_p)$$
    is an isomorphism.
\end{itemize}
The vector space $t_V(\QQ_p)$ is called the Bloch--Kato tangent space of $V$.

In particular, if  $D\subset {\bf D}_{\rm cris}(V)$ is regular, then we have 
\begin{equation}
\label{eqn_2022_05_21_1043}
    \dim D+\dim H^1_{\rm f}(G_{\QQ},V)=\dim t_V(\QQ_p)\,.
\end{equation} 
Perrin-Riou conjectures that to each pair $(V,D)$ where $D\subset {\bf D}_{\rm cris}(V)$ is a regular submodule, there exists a $p$-adic $L$-function $L_p(V,D,\sigma)$, where $\sigma$ is as usual the cyclotomic variable (cf. \cite[\S2 and \S4]{PR95}; see also \cite{Benois2}, \S2.2.5).

We first review Perrin-Riou's theory applied to the setting of \cite{Dasgupta2016} (where we take $\psi$ to be the trivial Dirichlet character for simplicity). Let $\hf$ be a non-CM Hida family as before and let $f=\hf_{\kappa_0}$ denote a $p$-stabilized eigenform which arises as a crystalline specialisation of $\hf$ of weight $k=w(\kappa_0)\geq 2$. In what follows, let us put $V_{\bf\cdot}:=T_{\bf\cdot} \otimes_{\ZZ_p}\QQ_p$ to ease notation. Let us consider
\begin{equation}
\label{eqn_2022_05_21_1047}
    D:={\bf D}_{\rm cris}(F^+V_f\otimes V_f^*(1)) \subset {\bf D}_{\rm cris}(\Ad(V_f)(1)),
\end{equation}
where $F^+V_f\subset V_f$ is the $p$-ordinary filtration. Note that the motivic weight of $\Ad(V_f)(1)$ equals $-2$. Even though $D$ is a regular submodule, the construction of $L_p(\Ad(V_f)(1),D,\sigma)$ is highly non-trivial, owing to the fact that $L(\Ad(f),s)$ admits no critical points.

We remark that $D$ can be naturally thought of as the ``limit'' of the sequence of reqular submodules
$$D_{\kappa,\lambda}:={\bf D}_{\rm cris}(F^+V_{\f_\kappa}\otimes V_{\hf_\lambda}^*(1))\subset {\bf D}_{\rm cris}(V_{\f_\kappa}\otimes V_{\hf_\lambda}^*(1))\,,\qquad {\rm w}(\kappa)>{\rm w}(\lambda)\,.$$ 
In this scenario, Hida has constructed a family of $p$-adic $L$-functions $L_p^{\rm Hida}({\rm ad}(\hf_\kappa)(1),\sigma)$. We remark that the $p$-adic $L$-function $L_p^{\rm Hida}({\rm ad}(\hf_\kappa)(1),\sigma)$ coincides with the one we denoted by $L_p(f,f,\sigma)$ in \S\ref{eqn_2022_05_18_1037} up to twisting, cf. \cite[\S1.3]{Dasgupta2016}.  The work of Kings--Loeffler--Zerbes in \cite{KLZ2} recovers Hida's $p$-adic $L$-function as a Perrin-Riou-style $p$-adic $L$-function: 
$$L_p^{\rm Hida}({\rm ad}(\hf_\kappa)(1),\sigma)=L_p({\rm ad}(V_{\hf_\kappa})(1), D_{\kappa,\lambda},\sigma).$$  One then defines 
$$L_p(\Ad(V_f)(1),D,\sigma):=\lim_{\substack{\kappa,\lambda\to \kappa_0\\ {\rm w}(\kappa)>{\rm w}(\lambda)}} L_p(V_{\f_\kappa}\otimes V_{\hf_\lambda}^*(1), D_{\kappa,\lambda},\sigma)\,.$$

Let us consider the natural exact sequence
\begin{equation}
    \label{eqn_dual_trace_palvannan__2022_05_21}
    0\lra E(1) \xrightarrow{{\rm tr}^*} \Ad(V_f)(1) \xrightarrow{\pi_{{\rm tr}^*}} \Ad^0(V_f)(1)   \lra 0\,,
\end{equation}
induced from the transpose of the trace map and the self-duality of $\Ad(V_f)$, and where $E/\QQ_p$ is an extension that contains the Hecke field of $f$. We may \emph{propagate} the regular subspace $D$ given as in \eqref{eqn_2022_05_21_1047} to obtain the regular submodules  \begin{align}
\begin{aligned}
\label{eqn_2022_05_21_1048}
    \{0\}={{\rm tr}^{*,-1}}(D)&\subset {\bf D}_{\rm cris}(E(1))\,, 
    \\
    \pi_{{\rm tr}^*}(D) &\subset {\bf D}_{\rm cris}(\Ad^0(V_f)(1))\,.
\end{aligned}
\end{align} 
Thanks to the fact that $D$ propagates to regular submodules of ${\bf D}_{\rm cris}(E(1))$ and ${\bf D}_{\rm cris}(\Ad^0(V_f)(1))$, one expects a factorisation
\begin{equation}
    \label{eqn_2022_05_21_1053}
    L_p(\Ad(V_f)(1),D,\sigma)= L_p(\Ad^0(V_f)(1),\pi_{{\rm tr}^*}(D),\sigma)\cdot  L_p(E(1),\{0\},\sigma)
\end{equation}
of Perrin-Riou-style $p$-adic $L$-functions. The factorisation \eqref{eqn_2022_05_21_1053} has been established by Dasgupta~\cite{Dasgupta2016}.

The setting that we place ourselves is significantly different: The regular submodule relevant to the $g$-dominant triple product $p$-adic $L$-function associated to $f\otimes g\otimes g^c$ does \emph{not} propagate to regular submodules for sub-quotients. As a result, one does not expect a factorisation of Perrin-Riou-style $p$-adic $L$-functions, akin to \eqref{eqn_2022_05_21_1053}, pointing towards a new phenomenon.


\section{\texorpdfstring{$L$}{}-functions}
\label{section_2_L_functions}

In \S\ref{section_2_L_functions}, we consider three Hida families that we shall denote by $\hf, \g$ and $\h$, of tame levels $N_\hf,N_\g,N_\h$, and tame characters $\varepsilon_\f,\varepsilon_\g,\varepsilon_\h$, respectively. For any weight space $\cW_{?}$ that will appear, we will denote by $\cW_{?}^{\rm cris}$ the subset of the classical points 
in $\cW_{?}^\cl$ which are crystalline at $p$, in the sense that the $p$-adic realisation of the motive associated to $\kappa\in \cW_{?}^{\rm cris}$  is required to be crystalline at $p$. In particular, given a crystalline point $\kappa \in \cW_\hf^\cl$ of weight ${\rm w}(\kappa)\geq 2$, the specialisation $\f_\kappa$ is $p$-old and arises as the unique $p$-ordinary stabilisation of a newform $\hf_\kappa^\circ \in S_{{\rm w}(\kappa)}(\Gamma_1(N_\hf),\varepsilon_\f)$ (and similarly for $\g$ and $\h$).

\subsection{Garrett--Rankin (degree-8) \texorpdfstring{$L$}{}-series and interpolation of central critical values}
\label{subsec:tripleproducts}
In this subsection, we retain the notation and conventions of \S\ref{subsubsec_211_2022_06_01_1635}--\S\ref{subsubsec_216_2022_06_01_1635}.

\subsubsection{} Let $\f$, $\g$, and $\h$ be three primitive Hida families of ordinary $p$-stabilized newforms of respective tame levels $N_\f, N_\g, N_\h$ and nebentypes $\varepsilon_\f,\varepsilon_\g,\varepsilon_\h$, all verifying the hypotheses \ref{item_Irr} and \ref{item_Dist}. We assume that $\varepsilon_\f\varepsilon_\hg\varepsilon_\hh$ is the trivial Dirichlet character modulo $N:={\rm LCM}(N_\f,N_\g,N_\h)$.

Given a $p$-crystalline classical point $x=(\kappa,\lambda,\mu) \in \cW_3^{\cris} := \cW_\hf^\cris \times \cW_{\hg}^{\cris}\times \cW_{\hh}^{\cris}$, we will often write $\hf_\kappa^\circ,\hg_\lambda^\circ,\hh_\mu^\circ$ for the associated newforms of level prime to $p$. We will also write $c$ in place of $c(x)$ whenever the choice of the point $x=(\kappa,\lambda,\mu)$ is clear from the context.

\subsubsection{Modified Hida periods}
\label{subsubsec_321_2022_05_26_1036}
For any primitive Hida family $\f$ as before, consider the generator $\mathfrak c_\f$ of the congruence ideal of $\f$ (cf. \cite{AIF_Hida88}). For a crystalline specialisation $\kappa \in \cW_\f^\cris$ we define the modified canonical period on setting
\begin{equation} \label{eqn:modifiedHidaperiods}
    \Omega_{\f_\kappa} := \mathcal E_p(\f_\kappa^\circ,\Ad) \cdot \frac{(-2\sqrt{-1})^{k+1}}{\mathfrak c_\f(\kappa)} \cdot \langle \f_\kappa^\circ, \f_\kappa^\circ \rangle, \qquad \text{where } \quad \mathcal E(\f^\circ_\kappa,\Ad) := \left( 1 - \frac{\beta_{\f^\circ_\kappa}}{\alpha_{\f^\circ_\kappa}} \right) \left( 1-\frac{\beta_{\f^\circ_\kappa}}{p\alpha_{\f^\circ_\kappa}}\right).
\end{equation}
We remark that $\mathfrak c_\f(\kappa) := \kappa(\mathfrak c_\hf)$ is the congruence number of $\f_\kappa$ (cf. \cite{Hida81}, (0.3)). By \cite[Corollary 6.24 \& Theorem 6.28]{hida_2016}, the period $\Omega_{\f_\kappa}$ is equal to the product $\Omega_{\f_\kappa}^+\Omega_{\f_\kappa}^-$ of the plus/minus canonical periods (cf. \cite[\S1.3]{Vatsal_Periods_1999}, see also \cite[Page 488]{hida_1994}) up to a $p$-adic unit.

\subsubsection{}  Let us put
\[
L_{\infty}(\f_\kappa^\circ\otimes\g_\lambda^\circ\otimes\h_\mu^\circ,s) := \begin{cases}
\Gamma_{\CC}(s)\ \Gamma_{\CC}(s+\wt(\kappa)-2c)\ \Gamma_{\CC}(s+1-\wt(\lambda))\ \Gamma_{\CC}(s+1-\wt(\mu)), & \quad x\in \cW_3^\hf, \\
\Gamma_{\CC}(s)\ \Gamma_{\CC}(s+1-\wt(\kappa))\ \Gamma_{\CC}(s+1-\wt(\lambda))\ \Gamma_{\CC}(s+1-\wt(\mu)), & \quad x\in \cW_3^\bal
\end{cases}
\]
(where we recall that $\cW_3^\f$ and $\cW_3^\bal$ are defined in \S\ref{subsubsec_213_2022_05_17_1455}) and we consider then the completed motivic $L$-series 
\[
    \Lambda(\f_\kappa^\circ\otimes\g_\lambda^\circ\otimes\h_\mu^\circ, s) := L_{\infty}(\f_\kappa^\circ\otimes\g_\lambda^\circ\otimes\h_\mu^\circ,s)\, L(\f_\kappa^\circ\otimes\g_\lambda^\circ\otimes\h_\mu^\circ,s),
\]
where $L(\f_\kappa^\circ\otimes\g_\lambda^\circ\otimes\h_\mu^\circ,s)$ stands for the usual motivic $L$-series given by a product of Euler factors at non-archimedean primes. The completed $L$-series $\Lambda(\f_\kappa^\circ\otimes\g_\lambda^\circ\otimes\h_\mu^\circ,s+c)$ coincides with Garrett's automorphic $L$-function $L(s+1/2,\Pi)$ associated with the irreducible unitary automorphic triple product representation $\Pi := \pi_f\times\pi_g\times \pi_h$ of $\GL_2(\mathbb A) \times\GL_2(\mathbb A) \times\GL_2(\mathbb A)$. In particular, by the work of Piatetski-Shapiro and Rallis~\cite{PSR87}, there exists $\varepsilon(\f_\kappa^\circ\otimes\g_\lambda^\circ\otimes\h_\mu^\circ)\in \{\pm 1\}$ and a positive integer $N(\f_\kappa^\circ\otimes\g_\lambda^\circ\otimes\h_\mu^\circ)$ such that we have a functional equation
\begin{equation}\label{eqn:functionaltriple}
    \Lambda(\f_\kappa^\circ\otimes\g_\lambda^\circ\otimes\h_\mu^\circ,s) = \varepsilon(\f_\kappa^\circ\otimes\g_\lambda^\circ\otimes\h_\mu^\circ) \cdot N(\f_\kappa^\circ\otimes\g_\lambda^\circ\otimes\h_\mu^\circ)^{c(x)-s} \cdot \Lambda(\f_\kappa^\circ\otimes\g_\lambda^\circ\otimes\h_\mu^\circ, 2c(x)-s)\,.
\end{equation}

\subsubsection{}\label{subsubsec_2022_06_02_0902}  We are especially interested in the central critical value $s=c$ and its $p$-adic avatar. If we let $x$ vary in the space $\cW_3^\cris$, the algebraic parts of the central critical values are interpolated by a $p$-adic analytic $L$-function, and we note that Hsieh's $p$-optimal construction in \cite{Hsieh} guarantees the integrality of his $p$-adic $L$-functions. 

As in \cite{Hsieh}, we define for $x\in \cW_3^?$ the modified Euler factors $\mathcal E_p^?(x)=\mathcal E_p^?(\f_\kappa^\circ\otimes \g_\lambda^\circ\otimes \h_\mu^\circ,c(x))$
\begin{align} \label{eqn:Eulertriple}
\begin{aligned}
    \mathcal E_p^?(x)= \begin{cases}
 \left(1-\frac{\beta_{\f_\kappa^\circ}\alpha_{\g_\lambda^\circ}\alpha_{\h_\mu^\circ}}{p^c}\right)^2
\left(1-\frac{\beta_{\f_\kappa^\circ}\alpha_{\g_\lambda^\circ}\beta_{\h_\mu^\circ}}{p^c}\right)^2
\left(1-\frac{\beta_{\f_\kappa^\circ}\beta_{\g_\lambda^\circ}\alpha_{\h_\mu^\circ}}{p^c}\right)^2
\left(1-\frac{\beta_{\f_\kappa^\circ}\beta_{\g_\lambda^\circ}\beta_{\h_\mu^\circ}}{p^c}\right)^2 & \text{ if } ?=\f\\
\left(1-\frac{\alpha_{\f_\kappa^\circ}\beta_{\g_\lambda^\circ}\beta_{\h_\mu^\circ}}{p^c}\right)^2
\left(1-\frac{\beta_{\f_\kappa^\circ}\alpha_{\g_\lambda^\circ}\beta_{\h_\mu^\circ}}{p^c}\right)^2
\left(1-\frac{\beta_{\f_\kappa^\circ}\beta_{\g_\lambda^\circ}\alpha_{\h_\mu^\circ}}{p^c}\right)^2
\left(1-\frac{\beta_{\f_\kappa^\circ}\beta_{\g_\lambda^\circ}\beta_{\h_\mu^\circ}}{p^c}\right)^2 & \text{ if } ?=\bal
\end{cases}
\end{aligned}
\end{align}
where $\alpha_{\f_\kappa^\circ}$ and $\beta_{\f_\kappa^\circ}$ denote the roots of the Hecke polynomial of $\f_\kappa^\circ$ at $p$ with $v_p(\alpha_{\f_\kappa^\circ})=0$ (and we similarly define $\{\alpha_{\g_\lambda^\circ},\beta_{\g_\lambda^\circ}\}$ and $\{\alpha_{\h_\mu^\circ},\beta_{\h_\mu^\circ}\}$).

\subsubsection{} The set
$\Sigma^- := \{ v \mid N: \quad \varepsilon_v(\f_\kappa^\circ\otimes \g_\lambda^\circ\otimes \h_\mu^\circ)=-1, \quad \forall x\in \cW_3^\cl \}$
is well-defined. Let us denote by $\Sigma_\mathrm{exc}$ the set of primes introduced in \cite[Equation (1.5)]{Hsieh} that only depends on the local type of the associated automorphic representations at primes dividing $N$, which is independent of $x$.

\begin{theorem}[Hsieh]
\label{thm:unbalancedinterpolation}
    Assume that $N={\rm LCM}(N_\hf, N_\hg, N_\hh)$ is squarefree and that $\Sigma^- = \emptyset$. Then there exists a unique element $\cL_p^\hf(\hf\otimes\hg\otimes\hh)\in \mathcal R_3$ with the following interpolative property for all $x = (\kappa, \lambda, \mu) \in \cW_3^\hf$:
    \[
        \cL_p^{\hf}(\hf\otimes\hg\otimes\hh)(x)^2 = \mathcal E_p^\hf(x)^2  \cdot (\sqrt{-1})^{-2{\rm w}(\kappa)} \cdot \mathfrak C_{\rm exc}(\hf\otimes\hg\otimes \hh)\cdot  \frac{\Lambda(\f_\kappa^\circ\otimes \g_\lambda^\circ\otimes \h_\mu^\circ, c(x))}{  \Omega_{\f_\kappa}^2 }\,,
    \]
    where $\mathfrak C_{\rm exc}(\hf\otimes\hg\otimes \hh) =  \prod_{q\in\Sigma_\mathrm{exc}} (1+q^{-1})^2$.
\end{theorem}
Exchanging the roles of $\f$, $\g$, and $\h$, we analogously have the $p$-adic $L$-functions with interpolation range in the other two unbalanced regions in the weight space: $\cL_p^\hg(\hf\otimes\hg\otimes\hh)$ and $\cL_p^\hh(\hf\otimes\hg\otimes\hh) \in \mathcal R_3$.

\begin{theorem}[Hsieh]
\label{thm:balancedinterpolation}
    Assume that $N={\rm LCM}(N_\hf, N_\hg, N_\hh)$ is squarefree and suppose that it factors as $N=N^+N^-$, with $\gcd(N^+, N^-)=1$, where $N^- = \prod_{\ell\in \Sigma^-}\ell$. Assume further that $\Sigma^-$ has odd cardinality and that all three residual representation $\overline{\rho}_\hf, \overline{\rho}_\hg, \overline{\rho}_\hh$ are ramified at all primes $\ell\in \Sigma^-$ with $\ell\equiv 1\pmod{p}$. Then there exists a unique element $\cL_p^\bal(\hf\otimes\hg\otimes\hh)\in \mathcal R_3$ with the following interpolative property for all $x = (\kappa,\lambda,\mu) \in \cW_3^\bal$:
    \[
        \cL_p^{\bal}(\hf\otimes\hg\otimes\hh)(x)^2 = \mathcal E_p^\bal(x)^2 \cdot  (\sqrt{-1})^{1-\rmw(\kappa)-\rmw(\lambda)-\rmw(\mu)} \cdot \mathfrak C_{\rm exc}(\hf\otimes\hg\otimes \hh) \cdot \frac{\Lambda(\f_\kappa^\circ\otimes\h_\lambda^\circ\otimes\h_\mu^\circ, c(x))}{\Omega_{\f_\kappa} \Omega_{\g_\lambda} \Omega_{\h_\mu} }\,.
    \]
\end{theorem}

Notice that the period in the balanced range differs from the Gross period in \cite[Def. 4.12]{Hsieh} by an element $u\in \mathcal R_3$, which is nonzero at all classical specialisations (this difference will not make any difference for our purposes).

\subsubsection{} Summarizing, \cite[Theorem B]{Hsieh} supplies us with 4 $p$-adic $L$-functions $$\cL_p^{\hf}(\hf\otimes\hg\otimes\hh)\, , \quad
\cL_p^{\hg}(\hf\otimes\hg\otimes\hh)\, ,  \quad
\cL_p^{\hh}(\hf\otimes\hg\otimes\hh)\, , \quad
\cL_p^{\mathrm{bal}}(\hf\otimes\hg\otimes\hh)\quad \in \quad \mathcal R_{3}\,,$$ 
interpolating the special values of the motivic $L$-function $L(\f_\kappa^\circ\otimes\h_\lambda^\circ\otimes\h_\mu^\circ,c)$ over $4$ mutually disjoint ranges of interpolation in $\cW_3$.

\subsection{\texorpdfstring{$L$}{}-series for \texorpdfstring{$f\otimes \Ad^0(g)$}{} and interpolation} \label{sec:Adjoint}
We assume in \S\ref{sec:Adjoint} that $\varepsilon_\f=1$ and retain the notation and conventions of \S\ref{subsec:tripleproducts}, concentrating in the particular scenario when $\h=\g^c:=\g\otimes \varepsilon_\g^{-1}$ is the conjugate Hida family (as in \S\ref{subsubsec_30042021_1123}).

\subsubsection{} The point $ y = (\kappa,\lambda)\in \cW_2^{\cl}$ is called $\hf$-dominant if $\frac{\rmw(\kappa)}{2}\geq \rmw(\lambda)$ and balanced otherwise. We define $\cW_2^{\hf}$,  $\cW_2^{\Ad} \subset \cW_2^{\cl}$ as in \eqref{eqn_2022_05_17_1552}.
Following Serre's general recipe, we put
\[
    L_\infty(\hf_\kappa^\circ\otimes\Ad^0\hg_\lambda^\circ,s) = 
    \begin{cases}
        \Gamma_{\CC}(s+{\rm w}(\lambda)-1)\cdot \Gamma_{\CC}(s)\cdot \Gamma_{\CC}(s+1-{\rm w}(\lambda)), & \qquad y\in \cW_2^\hf\,,\\
        \Gamma_{\CC}(s+{\rm w}(\lambda)-1)\cdot \Gamma_{\CC}(s)\cdot \Gamma_{\CC}(s+{\rm w}(\lambda) - {\rm w}(\kappa) ), & \qquad y\in \cW_2^\Ad 
    \end{cases}
\]
and we consider the completed motivic $L$-series
\[
    \Lambda(\hf_\kappa^\circ\otimes\Ad^0\hg_\lambda^\circ,s) = L_\infty(\hf_\kappa^\circ\otimes\Ad^0\hg_\lambda^\circ,s) \cdot L(\hf_\kappa^\circ\otimes\Ad^0\hg_\lambda^\circ,s),
\]
where we define $L(\hf_\kappa^\circ\otimes\Ad^0\hg_\lambda^\circ,s)$ as an Euler product of non-archimedean factors. 

\subsubsection{} 

When $y\in \cW_2^?\subset \cW_2^\cris$ ($?=\f,{\rm ad}$), we define the modified Euler factors
\begin{align*}
\begin{aligned} 
  \mathcal E^?_p(\kappa,\lambda,\Ad):=
\begin{cases}
  (1-\beta_{\hf_\kappa^\circ}\alpha_{\hg_\lambda^\circ}\beta_{\hg_\lambda^\circ}^{-1}p^{-\frac{\rmw(\kappa)}{2}})
  (1-{\beta_{\hf_\kappa^\circ}}{p^{-\frac{\rmw(\kappa)}{2}}}) 
  (1-{\beta_{\hf_\kappa^\circ}\beta_{\hg_\lambda^\circ}}{\alpha_{\hg_\lambda^\circ}^{-1}p^{-\frac{\rmw(\kappa)}{2}}}),& \text{ if } ?=\hf \\
 (1-{\alpha_{\hf_\kappa^\circ}\beta_{\hg_\lambda^\circ}}{\alpha_{\hg_\lambda^\circ}^{-1}p^{-\frac{\rmw(\kappa)}{2}}})
 (1-{\beta_{\hf_\kappa^\circ}}{p^{-\frac{\rmw(\kappa)}{2}}}) 
 (1-{\beta_{\hf_\kappa^\circ}\beta_{\hg_\lambda^\circ}}{\alpha_{\hg_\lambda^\circ}^{-1}}p^{\frac{-\rmw(\kappa)}{2}}),
 & \text{ if } ?=\Ad
\end{cases}
\end{aligned}
\end{align*}
where $\alpha_{\bullet}$ and $\beta_{\bullet}$ are as in \S\ref{subsubsec_2022_06_02_0902}. When $\kappa$ is non-crystalline, we put $\mathcal E^?_p(\kappa,\lambda,\Ad)=0$.

\subsubsection{} 
\label{subsubsec_conj_2022_06_02_0940}
The following conjecture is in the spirit of Hsieh's $p$-optimal construction of triple product $p$-adic $L$-functions (which in turn dwells on the Coates--Perrin-Riou formalism):

\begin{conj}  
\label{conj_2022_06_02_0940}
    There exist $p$-adic $L$-functions $\cL_p^{?}(\hf\otimes \Ad^0\hg )\in \mathcal R_2$ (where $?=\f,\,\Ad$) with the following interpolative properties:
    \begin{align}
    \label{eqn_conj_2022_06_02_0940}
        \cL_p^{?}(\hf\otimes \Ad^0\hg )(y) =
        \begin{cases}
         C_{\f_\kappa}^- \cdot \mathfrak{g}(\psi_{\kappa}^{\frac{1}{2}})\cdot  \mathcal E^?_p(\kappa,\lambda,\Ad) \cdot   \dfrac{\Lambda(\hf_\kappa^\circ\otimes\Ad^0\hg_\lambda^\circ,\psi_{\kappa}^{-\frac{1}{2}},\frac{\wt(\kappa)}{2})}{\Omega_{\hf_\kappa}^- \cdot \Omega_{\hf_\kappa}^2 } & \text{ if } ?=\hf\,,\\
    C_{\f_\kappa}^- \cdot \mathfrak{g}(\psi_{\kappa}^{\frac{1}{2}})\cdot  \mathcal E^?_p(\kappa,\lambda,\Ad) \cdot  \dfrac{\Lambda(\hf_\kappa^\circ\otimes\Ad^0\hg_\lambda^\circ,\wt(\kappa)/2)}{\Omega_{\hf_\kappa}^- \cdot \Omega_{\hg_\lambda}^2}  & \text{ if } ?=\Ad\,.
      \end{cases}
    \end{align}
    for all arithmetic $($not necessarily crystalline$)$ arithmetic specialisations $y=(\kappa,\lambda)$. Here:
    \begin{itemize}
    \item $C_{\f_\kappa}^-$ is a $p$-adic period $($cf. \cite{Ochiai2006}, Theorem 6.7 and Remark 6.8$)$.
        \item $\psi_{\kappa}$ is the wild-nebentype of $\f_\kappa$ $($which is trivial if $\kappa$ is crystalline$)$ and $\mathfrak{g}(-)$ is the Gauss sum. 
    \end{itemize}
\end{conj}

 We note that the importance of $p$-adic $L$-function $\cL_p^\Ad(\hf\otimes \Ad^0\hg )$  (whose existence is conjectural at present) in this article is due to its appearance in Conjecture~\ref{conj_main_6_plus_2}. We refer the reader to \cite{JanuszewskiHida, PdVP19, CdVP23, ACR25} for results towards Conjecture~\ref{conj_2022_06_02_0940} (as well as the related earlier works \cite{Schmidt, JanusReal}).


\section{Selmer complexes}
\label{sec_selmer_complexes}
We introduce in the present section the algebraic counterparts of the ($p$-adic) analytic objects we have discussed in \S\ref{section_2_L_functions}. These will be used to formulate (and resolve) the factorisation problem for algebraic $p$-adic $L$-functions. We note that our approach (that incorporates the ETNC\footnote{abbrv. Equivariant Tamagawa Number Conjecture.} philosophy) also sheds light on the analytic counterparts.

\subsection{Greenberg conditions}
\label{subsec_greenberg_conditions_17_05_2021}
For each $T \in \{T_\hf^\dagger, T_2^\dagger, T_3^\dagger, M_2^\dagger \}$ as in \S\ref{subsec_the_set_up_intro} (with $\varepsilon_\f = 1$ and $\h=\g^c$, the conjugate Hida family), the fact that $\f$ and $\g$ are slope-zero families, together with \ref{item_Irr} and \ref{item_Dist}, gives rise to short exact sequences of $G_{\Qp}$-modules of the form
\begin{equation}
    \label{eqn_11_05_2021_filtration}
        0\lra F^+T \lra T \lra T/F^+T \lra 0
\end{equation}
which we shall use to define Selmer complexes. We note that even for a given $T$, there exists multiple one-step filtrations as above, which one should think of as a reflection of the fact that there are multiple choices of $p$-adic $L$-functions, depending on the chosen domain of interpolation. 

Given a short exact sequence \eqref{eqn_11_05_2021_filtration}, we consider the following local conditions on $T$ on the level of continuous cochains, in the sense of Nekov\'a\v{r}~\cite[\S6]{nekovar06}:
\[
    U_v^+(T) =
    \begin{cases}
        C^\bullet(G_p, F^+T),   & v=p \\
        C^\bullet(G_v/I_v, T^{I_v}), &v\in \Sigma\setminus\{p\}\,.
    \end{cases}
\]
Here $G_v := G_{\QQ_v}$ and $\Sigma := \{v \mid N\} \cup \{p, \infty\}$. 
These come equipped with a morphism of complexes
\[
    \iota^+_v : U_v^+(T) \longrightarrow C^\bullet(G_v,T)\,, \qquad\forall v\in \Sigma\,.
\]
For each $v\in \Sigma\setminus\{p\}$, we note that $U_v^+(T)$ is quasi-isomorphic to the complex 
\[
    U^+_v(T) = \left[ T^{I_v} \xrightarrow{ \, \Fr_v-1 \, } T^{I_v} \right],
\]
concentrated in degrees $0$ and $1$. Here, $I_v\subset G_v$ is the inertia subgroup and $\Fr_v \in G_v/I_v$ is the (geometric) Frobenius element. 

\subsubsection{}
Let $v\neq p$ be any prime and suppose that 
$$0\lra M_1 \lra M_2 \lra M_3\lra 0$$
is a short exact sequence of continuous $G_v$-representations, which are free of finite rank over a complete local Noetherian ring $R$. Then the sequence
$$ 0\lra U^+_v(M_1) \lra U^+_v(M_2)\lra U^+_v(M_3) $$
is also exact. Moreover, if $M_2=M_1\oplus M_3$, then
\begin{equation}
    \label{eqn_local_conditions_at_v_for_direct_sum}
     U^+_v(M_2)= U^+_v(M_1) \oplus  U^+_v(M_3)\,.  
\end{equation}


\subsubsection{}
\label{subsubsec_311_17_05_2021}
The data
$$\Delta(F^+T):=\{U_v^+(T)\xrightarrow{\iota_v^+} C^\bullet(G_v,T): v\in \Sigma\}$$
is called a Greenberg-local condition on $T$. In this paragraph, we shall record examples of  Greenberg local conditions, especially those crucial for our purposes.

\begin{example}
\label{example_local_conditions}
In this example, $X=T_?^\dagger$ with $?=2,3$.
\item[i)]  The $\f$-dominant Greenberg-local condition $\Delta_{\hf}=\Delta(F^+_\f X)$ on $X$ is given by the choices
\begin{align*}
 F_\hf^+T_3^\dagger &\,:= F^+T_{\hf}^\dagger\,\widehat\otimes\, T_\hg \,\widehat\otimes\, T_\hg^* \hookrightarrow T_3^\dagger\,,
 \\
 F_\hf^+T_2^\dagger &\,:= F_\hf^+T_3^\dagger\otimes_{\iota_{2,3}^*}\cR_2 \,=\,F^+T_{\hf}^\dagger\,\widehat\otimes\, \Ad(T_\hg) \hookrightarrow T_2^\dagger\,.
\end{align*}
\item[ii)]  The $\g$-dominant Greenberg-local condition $\Delta_{\hg}:=\Delta(F^+_\g T_3^\dagger)$ on $T_3^\dagger$ is given with the choice
$$F_\hg^+T_3^\dagger := T_{\hf}^\dagger \,\widehat\otimes\,  F^+T_\hg \,\widehat\otimes\,  T_\hg^* \hookrightarrow T_3^\dagger$$
whereas the $\g$-dominant Greenberg-local condition $\Delta_{\hg}:=\Delta(F^+_\g T_2^\dagger)$ on $T_2^\dagger$ is given with the choice
$$F_\hg^+T_2^\dagger :=  F_{\g}^+T_3^\dagger\otimes_{\iota_{2,3}^*}\cR_2= T_{\hf}^\dagger \,\widehat\otimes\,  \underbrace{F^+T_\g \,\otimes_{\cR_\g}\,  T_\hg^*}_{=:F^+_\g \Ad(T_\g)} \hookrightarrow T_2^\dagger\,. 
$$

\item[iii)]  The balanced Greenberg-local condition $\Delta_{\rm bal}:=\Delta(F^+_{\rm bal} X)$ on $X$ is given with the choice
\begin{align*}
F_{\rm bal}^+T_3^\dagger &\,:= F^+T_\hf^\dagger \,\widehat\otimes\, F^+T_\hg \,\widehat\otimes\, T_\hg^* + F^+T_\hf^\dagger \,\widehat\otimes\, T_\hg \,\widehat\otimes\, F^+T_\hg^* + T_\hf \,\widehat\otimes\, F^+T_\hg \,\widehat\otimes\, F^+T_\hg^*  \hookrightarrow T_3^\dagger\,,
\\
   F_{\rm bal}^+T_2^\dagger &\,:= F_{\rm bal}^+T_3^\dagger\otimes_{\iota_{2,3}^*}\cR_2 \, =\,F^+T_\hf^\dagger \,\widehat\otimes\, F^+T_\hg \otimes_{\cR_\g} T_\hg^* + F^+T_\hf^\dagger \,\widehat\otimes\, T_\hg \otimes_{\cR_\g} F^+T_\hg^* + T_\hf \,\widehat\otimes\, F^+T_\hg \otimes_{\cR_\g} F^+T_\hg^*  \hookrightarrow T_2^\dagger\,.
\end{align*}

\item[iv)] We will also work (especially in the proof of Proposition~\ref{prop_suport_range_Selmer_complexes_8_6_2_8}) with the Greenberg local conditions $\Delta_-:=\Delta(F^+_{{\rm bal-}}X)$ and $\Delta_+:=\Delta(F^+_{{\rm bal+}},X)$ given by the following rank-3 and rank-5 direct summands of $X$, respectively:
\begin{align*}
    F^+_{{\rm bal-}}X \,:= F_\g^+X \cap F_{\rm bal}^+\,,
\qquad \qquad 
    F^+_{{\rm bal+}}X \,:= F_\g^+X + F_{\rm bal}^+X\,.
\end{align*}
In explicit terms,
\begin{align*}
F^+_{{\rm bal-}}T_2^\dagger &\,:= F_\g^+T_2^\dagger \cap F_{\rm bal}^+T_2^\dagger\,\,=\,\,F^+T_\hf^\dagger \,\widehat\otimes\, F^+T_\hg \otimes_{\cR_\g} T_\hg^* + T_\hf \,\widehat\otimes\, F^+T_\hg \otimes_{\cR_\g} F^+T_\hg^*\,,
\\
F^+_{{\rm bal+}}T_2^\dagger &\,:= F_\g^+T_2^\dagger + F_{\rm bal}^+T_2^\dagger\,\,=\,\,F_\g^+T_2^\dagger  + F^+T_\hf \,\widehat\otimes\, T_\hg \otimes_{\cR_\g} F^+T_\hg^*\,.
\end{align*}
The $\cR_3$-direct-summands $F^+_{{\rm bal\pm}}T_3^\dagger$ also have a similar explicit description, which we omit here to save ink.
\end{example}

\begin{remark}
\label{remark_self_duality_orthogonal_complements}
We remark that under the self-duality pairing $T_2^\dagger\otimes T_2^\dagger\to \cR_2(1)$, the local conditions $\Delta_{\g}$ and $\Delta_{\rm bal}$ are self-orthogonal-complements: 
$\Delta_{\g}^\perp =\Delta_{\g}\,,\,\, \Delta_{\rm bal}^\perp=\Delta_{\rm bal}\,.$
Moreover, 
$\Delta_{+}^\perp =\Delta_{-},\,$ and $\Delta_{-}^\perp=\Delta_{+}\,.$
\end{remark}

We shall propagate these local conditions to $M_2^\dagger$ and $T_\f^\dagger$ via the exact sequence \eqref{eqn_factorisationrepresentations_dual_trace}.
\begin{defn}
\label{defn_propagate_local_conditions_via_dual_trace}
For $?\in\{\f,\g,{\rm bal}\}$, we define the Greenberg-local condition ${\rm tr}^*\Delta_?=\Delta(F_?^+M_2^\dagger)$ on $M_2^\dagger$ on setting
$F_?^+M_2^\dagger:={\rm im}\left(F_?^+T_2^\dagger \hookrightarrow T_2^\dagger\twoheadrightarrow M_2^\dagger\right)\,.$
\end{defn}

\begin{remark}
\label{remark_defn_propagate_local_conditions_via_dual_trace}
In more explicit terms, we have put
$F_\f^+M_2^\dagger:=F^+T_\hf^\dagger \,\widehat\otimes\,\Ad^0(T_\g)$\,, and
\begin{align*}
F_{\g}^+M_2^\dagger &:=T_\hf^\dagger \,\widehat\otimes\,\underbrace{\ker\left(\Ad^0(T_\g)\hookrightarrow\Ad(T_\g)\to {\rm Hom}(F^+T_\g,F^-T_\g)\right)}_{=:F_\g^+\, \Ad^0(T_\g)}\,,
\\
F_{\rm bal}^+M_2^\dagger &:=F^+T_\hf^\dagger \,\widehat\otimes\,F_\g^+\, \Ad^0(T_\g) +  T_\hf^\dagger\,\widehat\otimes\,\underbrace{F^+T_\g\otimes_{\cR_\g}(F^-T_\g)^*}_{=:F_\g\, \Ad^0(T_\g)}\,. 
\end{align*}
\end{remark}

\begin{lemma}
\label{lemma_remark_ranks_3_3_2021_11_05} We have:
\item[i)] 
${\rm rank}_{\cR_\g}\, F_\f^+M_2^\dagger \,=\, 3\,=\, {\rm rank}_{\cR_\g}\, F_{\rm bal}^+M_2^\dagger\,, \hbox{ and } \, {\rm rank}_{\cR_\g}\, F_\g^+M_2^\dagger\,=\,4\,.$
\item[ii)] The natural short exact sequence of $G_{p}$-representations
\begin{equation}
    \label{eqn_lemma_remark_ranks_3_3_2021_11_05_2}
    0\lra F_{\rm bal}^+M_2^\dagger\lra F_\g^+M_2^\dagger\lra F^-T_\hf^\dagger \otimes_{\varpi_{2,1}^*}\cR_2\lra 0\,.
\end{equation}
\end{lemma}

\begin{proof}
The assertions in (i) are straightforward computations and we will only explain the proof of (ii). The containment $F_{\rm bal}^+M_2^\dagger\subset F_\g^+M_2^\dagger$ is also evident and one computes that
\begin{equation}
\label{eqn_3_4_21_05_2021}
{\rm coker}\left(F_{\rm bal}^+M_2^\dagger\lra F_\g^+M_2^\dagger \right)\stackrel{\sim}{\lra} F^-T_\f\,\widehat\otimes\,{\rm gr}_\g^+\,\Ad^0(T_\g)
\end{equation}
where ${\rm gr}_\g^+\,\Ad^0(T_\g):=F_\g^+\, \Ad^0(T_\g)\Big{/}F_\g\, \Ad^0(T_\g)$. It therefore remains to verify that
\begin{equation}
\label{eqn_3_5_21_05_2021}
    F^-T_\hf^\dagger\,\widehat\otimes\,{\rm gr}_\g^+\,\Ad^0(T_\g)\,\stackrel{\sim}{\lra} F^-T_\hf^\dagger \otimes_{\varpi_{2,1}^*}\cR_2 
\end{equation}
as $G_p$-representations. This is equivalent to check that the $G_p$-action on ${\rm gr}_\g^+\,\Ad^0(T_\g)$ is trivial. 

To see that, let $\{v_+\}\subset F^+T_\g$ denote an $\cR_\g$-basis and let $\{v_+,v_-\}\subset T_\g$ denote a basis complementing $\{v_+\}$. Let us denote by $\{v_+^*,v_-^*\}\subset T_\g^*$ the dual basis. Then
$$F_\g^+\, \Ad^0(T_\g)=\ker\left(\Ad^0(T_\g)\hookrightarrow\Ad(T_\g)\to {\rm Hom}(F^+T_\g,F^-T_\g)\right)=\langle v_+\otimes v_-^*, v_+\otimes v_+^*-v_-\otimes v_-^* \rangle$$
and $F_\g\,\Ad^0(T_\g)=F^+T_\g\otimes_{\cR_\g}(F^-T_\g)^*=\langle v_+\otimes v_-^*\rangle\,.$ We have therefore reduced to check that
\begin{equation}
\label{eqn_remark_ranks_3_3_2021_11_05_want_to_check}
\left(gv_+\otimes gv_+^*-gv_-\otimes gv_-^*\right) -(v_+\otimes v_+^*-v_-\otimes v_-^* )\in \langle v_+\otimes v_-^*\rangle\,.
\end{equation}
Only in this proof, we let $\chi_\pm$ denote the $\cR_\g$-valued character of $G_p$ giving its action of $F^\pm T_\g$. Then, 
\begin{equation}
\label{eqn_remark_ranks_3_3_2021_11_05_1}
    gv_+=\chi_+(g)v_+\,,\quad gv_+^*=\chi_+(g)^{-1}v_+^*\,,\qquad \hbox{hence } \quad gv_+\otimes gv_+^*=v_+\otimes v_+^*
\end{equation}
and  $gv_-=\chi_-(g)v_-+c(g)v_+\, \hbox{ for some } c(g)\in \cR_\g\,,\quad gv_-^*=\chi_-(g)^{-1}v_-^*\,\,,$ hence
\begin{equation}
\label{eqn_remark_ranks_3_3_2021_11_05_2}
 gv_-\otimes gv_-^*=v_-\otimes v_-^*+c(g)v_+\otimes v_-^*\,.
\end{equation}
Combining \eqref{eqn_remark_ranks_3_3_2021_11_05_1} and \eqref{eqn_remark_ranks_3_3_2021_11_05_2}, we conclude that \eqref{eqn_remark_ranks_3_3_2021_11_05_want_to_check} holds true.
\end{proof}

\begin{remark}\label{remark_Fg+diagram-comuutes}
The  proof of Lemma \ref{lemma_remark_ranks_3_3_2021_11_05} shows that the following diagram commutes:
\[
\xymatrix{
F_\g^+\mathrm{ad}(T_\g) \ar[r]^-{\pi_{\mathrm{tr^*}}}  \ar@{^{(}->}[d] &  \ar@{->>}[r] F_\g^+\mathrm{ad}^0(T_\g) & F_\g^+\mathrm{ad}^0(T_\g)/F_\g\mathrm{ad}^0(T_\g) \ar[d]_{\cong}^{v_+ \otimes v_+^* - v_- \otimes v_-^* \,\mapsto\, 1/2}
\\
\mathrm{ad}(T_\g) \ar[rr]^-{\mathrm{tr}} && \cR_2. 
}
\]
\end{remark}

\begin{defn}
\label{defn_propagate_local_conditions_via_dual_trace_V_fdagger}
For $?\in\{\f,\g,{\rm bal}\}$, we define the Greenberg-local condition ${\rm tr}^*\Delta_?=\Delta(F_?^+T_\f^\dagger)$ on $T_\f^\dagger$ on setting 
$F_?^+T_\f^\dagger:={\rm ker}\left(F_?^+T_\f^\dagger \hookrightarrow T_2^\dagger\twoheadrightarrow M_2^\dagger\right)$\,.
\end{defn}

\begin{remark}
\label{remark_defn_propagate_local_conditions_via_dual_trace_V_fdagger}In more explicit form, we have set $F_\f^+T_\f^\dagger:=F^+T_\hf^\dagger=:F_{\rm bal}^+T_\f^\dagger$\,, and $F_{\g}^+T_\f^\dagger=0$\,.
In other words, as local conditions on $T_\f^\dagger$, we have
\[
{\rm tr}^*\Delta_\f= \Delta_{\mathrm{Pan}} ={\rm tr}^*\Delta_{\rm bal}\,,\qquad {\rm tr}^*\Delta_\g=\Delta_0\,,
\]
where $\Delta_{\mathrm{Pan}} =\Delta(F^+T_\f)$ and $\Delta_0:=\Delta(\{0\})$. 
\end{remark}


\begin{defn}
We define the Greenberg local condition $\Delta_{\emptyset}$ on $T$ by setting $\Delta_{\emptyset} := \Delta(T)$. 
\end{defn}

\subsubsection{}\label{subsubsec_aux_definitions_grading_on_ad0g}
We record in \S\ref{subsubsec_aux_definitions_grading_on_ad0g} the definitions and various conclusions arising from the proof of Lemma~\ref{lemma_remark_ranks_3_3_2021_11_05} that we shall utilize later in \S\ref{subsec_factor_general}. 

As above, we let $\{v_+\}\subset F^+T_\g$ denote an $\cR_\g$-basis and let $\{v_+,v_-\}\subset T_\g$ denote a basis complementing $\{v_+\}$. Let us denote by $\{v_+^*,v_-^*\}\subset T_\g^*$ the dual basis. Recall that we have
$$F_\g^+\, \Ad^0(T_\g)=\langle v_+\otimes v_-^*, v_+\otimes v_+^*-v_-\otimes v_-^* \rangle\, , \qquad F_\g^+ M_2^\dagger=T_\f^\dagger\,\widehat\otimes\,F^+_\g\Ad^0(T_\g)\,.$$
Recall also the $G_p$-stable rank-one direct summand $F_\g\, \Ad^0(T_\g):=\langle v_+\otimes v_-^*\rangle \subset F_\g^+\, \Ad^0(T_\g)$. As we have seen as part of the proof of Lemma~\ref{lemma_remark_ranks_3_3_2021_11_05}, the $G_p$ acts trivially on the quotient ${\rm gr}_\g^+\,\Ad^0(T_\g):=F_\g^+ \Ad^0(T_\g)/ F_\g\Ad^0(T_\g).$

\subsection{Selmer complexes associated to Greenberg local conditions}
\label{subsec_selmer_complexes}
Let $T$ be as in \S\ref{subsec_greenberg_conditions_17_05_2021} and let $\Delta$ be a Greenberg-local condition, in the sense of Definition~\ref{subsubsec_311_17_05_2021}. Let us put
\[
    U_\Sigma(T) := \bigoplus_{v\in \Sigma} U^+_v(T), \qquad C_\Sigma^\bullet(T) := \bigoplus_{v\in \Sigma} C^\bullet(G_v,T)
\]
We then define the Selmer complex associated to $(T,\Sigma, \Delta)$ on setting
\[
    \widetilde{C}^\bullet_{\rm f}(G_{\QQ,\Sigma},T,\Delta) := \mathrm{cone}\left( C^\bullet(G_{\QQ, \Sigma}, T) \oplus U_\Sigma^+(T) \xrightarrow{ \, \res_\Sigma - \iota_\Sigma^+ \, } C_\Sigma^\bullet(T) \right)[-1].
\]
We denote the corresponding object in the derived category by $\widetilde{R\Gamma}_{\rm f}(G_{\QQ,\Sigma},T,\Delta)$ and its cohomology by $\widetilde{H}^\bullet_{\rm f}(G_{\QQ,\Sigma},T,\Delta)$.

\begin{proposition}\label{proposition_euler_poincare_characteristic}
We have 
$$\chi(\widetilde{R\Gamma}_{\rm f}(G_{\QQ,\Sigma},T,\Delta))=\mathrm{rank}(T^{c=1}) - \mathrm{rank}(F^+T)$$ 
for the Euler--Poincar\'e characteristic of $\widetilde{R\Gamma}_{\rm f}(G_{\QQ,\Sigma},T,\Delta)$.
\end{proposition}

\subsubsection{} 
The following exact triangle, which we deduce from our discussion in Remark~\ref{remark_defn_propagate_local_conditions_via_dual_trace_V_fdagger}, can be thought of as the first step towards our factorisation statements and it will be repeatedly used in what follows:
\begin{equation}
\label{eqn_sequence_dual_trace_derived_category}
    \widetilde{R\Gamma}_{\rm f}(G_{\QQ,\Sigma}, {\varpi_{2,1}^*(T_\f^\dagger)}, \Delta_{0})  
    \xrightarrow{{\rm id}\otimes {\rm tr}^*} \widetilde{R\Gamma}_{\rm f}(G_{\QQ,\Sigma},T_2^\dagger,\Delta_{\g})\xrightarrow{\pi_{{\rm tr}^*}} \widetilde{R\Gamma}_{\rm f}(G_{\QQ,\Sigma},M_2^\dagger,{\rm tr}^*\Delta_{\g})\xrightarrow[+1]{\delta}\,.
\end{equation}

\begin{defn}
\label{defn_Delta_g_mod_Delta_bal_quotient}
We let $\res_{/{\Delta_{\rm bal}}}$ denote the composite map
\begin{align*}
    \widetilde{H}^1_{\rm f}(G_{\QQ,\Sigma},M_2^\dagger,{\rm tr}^*\Delta_\g)&\xrightarrow{\res_p} H^1(G_p,F_\g^+M_2^\dagger)\xrightarrow{{\rm pr}_{/\g}}
    H^1(G_p,T_\f^\dagger\,\widehat\otimes\,{\rm gr}_\g^+\,\Ad^0(T_\g))\\
    &\qquad\qquad\lra H^1(G_p,F^-T_\f^\dagger\,\widehat\otimes\,{\rm gr}_\g^+\,\Ad^0(T_\g))
    \xrightarrow[\eqref{eqn_3_4_21_05_2021}]{\sim}  H^1(G_p,F_\g^+M_2^\dagger/F_{\rm bal}^+M_2^\dagger)\,,
\end{align*}
where the map ${\rm pr}_{/\g}$ is the map induced from the canonical surjection (which carries the same name)
$${\rm id}\otimes {\rm pr}_{/\g}: \, F_\g^+M_2^\dagger=T_\f^\dagger\,\widehat\otimes\,F_\g^+\,\Ad^0(T_\g)\xrightarrow{{\rm pr}_{/\g}} T_\f^\dagger\,\widehat\otimes\,{\rm gr}_\g^+\,\Ad^0(T_\g)\,.$$
\end{defn}


\subsubsection{}

We conclude \S\ref{subsec_selmer_complexes} introducing an invariant (that we call the Panchishkin defect), which is useful in identifying the differences between the factorisation problem we presently consider with those in \cite{Gross1980Factorization, Greenberg1981_factorization, Dasgupta2016, palvannan_factorization}. This discussion indicates why the factorisation problem for algebraic $p$-adic $L$-functions we consider in this article is different from those considered in \cite{Greenberg1981_factorization, palvannan_factorization}. 

\begin{defn}
\label{defn_weakly_panchiskin}
Suppose that we are given a free $\ZZ_p$-module $X$ of finite rank, endowed with a continuous action of $G_{\QQ,\Sigma}$. Suppose that $X$ is equipped with Greenberg local conditions $\Delta=\Delta(F^+X)$ given by a $G_p$-stable submodule $F^+X\subset X$. We say that the pair $(X,\Delta)$ is \emph{weakly Panchishkin} if we have $\chi(\widetilde{R\Gamma}_{\rm f}(G_{\QQ,\Sigma},X,\Delta))= 0$. Otherwise, we say that it has \emph{Panchishkin defect}  
$$\mathfrak{d}(X,\Delta)= \left|\chi(\widetilde{R\Gamma}_{\rm f}(G_{\QQ,\Sigma},X,\Delta)) \right| =\left|{\rm rank}\, X^{c=+1}-{\rm rank}\, F^+X\right|>0\,.$$
\end{defn}
We remark that $X$ is Panchishkin ordinary if one may find $F^+X$ so that $(X,\Delta(F^+X))$ is weakly Panchishkin and all Hodge--Tate weights of $F^+X$ (resp. of $X/F^+X$) are positive (resp. non-positive).

\begin{example}
\item[i)] Let $x\in\cW_2^{\rm cl}(\cO)$ be an $\cO$-valued arithmetic point (where $\cO$ is the ring of integers of a finite extension of $\QQ_p$) and let us put $T(x):=T\otimes_{x} \cO$ for $T=T_2^\dagger, M_2^\dagger, T_\f^\dagger$. Then the pair $(T_2^\dagger(x),\Delta_\g)$ is weakly Panchiskin, but
$$\mathfrak{d}(M_2^\dagger(x),{\rm tr}^*\Delta_\g)=1=\mathfrak{d}(T_\f^\dagger(x),{\rm tr}^*\Delta_\g)\,,\qquad \mathfrak{d}(M_2^\dagger(x),{\rm tr}\Delta_\g)=1=\mathfrak{d}(T_\f^\dagger(x),{\rm tr}\Delta_\g)\,.$$
\item[ii)] Suppose that we are in the setting of \S\ref{subsubsec_intro_comparison_palvannan_dasgupta}. Then the pairs $(\Ad(T_\g)(\lambda)\otimes\chi,\Delta_{\rm Gr})$, $(\Ad^0(T_\g)(\lambda)\otimes \chi,{\rm tr}^*\Delta_{\rm Gr})$, and $(\chi,\Delta_{0})$ are weakly Panchishkin, which explains why $p$-adic regulators do not appear in Palvannan's treatment, cf. \cite{palvannan_factorization}. On the other hand,
$$\mathfrak{d}(\Ad^0(T_\g)(\lambda)\otimes \chi,{\rm tr}\Delta_{\rm Gr})=1=\mathfrak{d}(\chi,\Delta_{\emptyset})$$
and this suggests that a comparison of regulators will be involved in the sought-after factorisation $($as in the case of Dasgupta's work$)$.
\end{example}


\subsection{Interlude: Remarks on Tamagawa factors}
\label{subsubsec_2022_09_08_1236}

This subsection is dedicated to a study of Tamagawa numbers. This discussion will play a role in the verification that our Selmer complexes are perfect.

\subsubsection{Generalities}
\label{subsubsec_2022_09_09_1253} Let $v \neq p$ be a rational prime. 
Let $M$ be a continuous $G_{v}$-representation, which is free of finite rank over a complete local Noetherian ring $R$ with finite residue field of characteristic $p$. We write $\rho_M \colon G_v \to \GL(M)$ for the continuous homomorphism induced by the action of $G_v$ on $M$. We consider the following condition: 
\begin{itemize}
    \item The $R$-module $H^1(I_v, M)$ is free. 
\end{itemize}

When $R = \ZZ_p$, the order of $H^1(I_v, M)^{\mathrm{Fr}_v = 1}_{\mathrm{tors}}$ is the local Tamagawa factor at $v$ (see \cite[\S I.4.2.2]{FontainePerrinRiou_Motives}). Therefore, the local Tamagawa factor of $M$ at $v$ vanishes under the condition that the $R$-module $H^1(I_v, M)$ is free. 

Let $I_v^{(p)} < I_v$ denote the unique subgroup satisfying $I_v/I_v^{(p)}  \cong \ZZ_p$.  Note that the maximal pro-$p$ quotient of $I_v$ is isomorphic to $\Zp$.  Let us fix a topological generator $t \in I_v/I_v^{(p)}$. 

\begin{lemma}\label{lemma_I_v_cohomology}\ 
    \item[1)] The $R$-module  $M^{I_v^{(p)}}$ is free of finite rank.  
    \item[2)] The complex $C^{\bullet}(I_v, M)$ is quasi-isomorphic to the perfect complex 
    \[
    \cdots\quad \longrightarrow 0 \longrightarrow M^{I_v^{(p)}} \xrightarrow{t - 1} M^{I_v^{(p)}} \longrightarrow 0 \quad \cdots
    \]
    concentrated in degrees $0$ and $1$. In other words, $R\Gamma(I_v, M) \in D^{[0,1]}_{\rm parf}(_R\mathrm{Mod})$. 
    \item[3)] For any ring homomorphism $R \to S$, the induced map $M^{I_v^{(p)}} \otimes_R S \to (M \otimes_R S)^{I_v^{(p)}}$ is an isomorphism. In particular, $R\Gamma(I_v, M) \otimes^{\mathbb{L}}_{R} S \stackrel{\sim}{\longrightarrow} R\Gamma(I_v, M \otimes_R S)$. 
\end{lemma}
\begin{proof}
Let $\mathfrak{m}_R$ denote the maximal ideal of $R$. Since
$\ker\left(\GL(M/\mathfrak{m}_R^{n+1}) \to \GL(M/\mathfrak{m}_R^{n})\right)$ is a $p$-group for each positive integer $n$, it follows that $\rho_M(I_v^{(p)})$ is finite and $p \nmid \#\rho_M(I_v^{(p)})$. 

\item[1)] Since $\rho_M(I_v^{(p)})$ is finite and $p \nmid \#\rho_M(I_v^{(p)})$, we have 
\[
e := \frac{1}{\# \rho_M(I_v^{(p)})}\sum_{g \in \rho_M(I_v^{(p)})}g \in \ZZ_p[\rho_M(I_v^{(p)})]
\]
and $M^{I_v^{(p)}} = eM$. This shows that the inclusion $M^{I_v^{(p)}} \longrightarrow M$ is split, and hence $M^{I_v^{(p)}}$ is free. 

\item[2)] This portion follows from the quasi-isomorphism $C^{\bullet}(I_v/I_v^{(p)}, M^{I_v^{(p)}}) \to C^{\bullet}(I_v, M)$  induced from the inflation morphism, combined with the fact that  $I_v/I_v^{(p)}$ is pro-cyclic. 

\item[3)] Since the inclusion $M^{I_v^{(p)}} \to M$ is split, the $R$-module $M/M^{I_v^{(p)}}$ is flat, and hence we have an exact sequence of $S$-modules  
\[
0 \longrightarrow   M^{I_v^{(p)}}\otimes_{R} S \longrightarrow M \otimes_{R} S \longrightarrow M/M^{I_v^{(p)}} \otimes_{R} S \longrightarrow 0. 
\]
Moreover, since the group $I_v^{(p)}$ acts trivially on the ring $S$, we have 
\[
(M/M^{I_v^{(p)}} \otimes_{R} S)^{I_v^{(p)}} = e(M/M^{I_v^{(p)}} \otimes_{R} S) = e(M/M^{I_v^{(p)}}) \otimes_{R} S = 0\,. 
\]
This shows that the canonical homomorphism $M^{I_v^{(p)}} \otimes_R S \to (M \otimes_R S)^{I_v^{(p)}}$ is indeed an isomorphism. 
\end{proof}

\begin{remark}\label{remark:inert-coh}
Since $\Gal(\QQ_v^{\rm ur}/\QQ_v)$ acts non-trivially on $I_v/I_v^{(p)}$, the cohomology group $H^1(I_v, M)$ is not isomorphic to $M^{I_v^{(p)}}/(t-1)M^{I_v^{(p)}}$ as $\Gal(\QQ_v^{\rm ur}/\QQ_v)$-modules. As a matter of fact, we have 
\[
H^1(I_v, M) \cong (M^{I_v^{(p)}}/(t-1)M^{I_v^{(p)}})(-1)
\]
as $\Gal(\QQ_v^{\rm ur}/\QQ_v)$-modules. 
\end{remark}


\begin{corollary}\label{corollary_free_potentially_good}
Suppose that $\rho_M(I_v)$ is finite (i.e., $M$ has potentially good reduction) and $p \nmid \# \rho_M(I_v)$, then the $R$-module $H^1(I_v, M)$ is free. 
\end{corollary}
\begin{proof}
Since $p \nmid \# \rho_M(I_v)$, the element $t \in I_v/I_v^{(p)}$ acts trivially on $M^{I_v^{(p)}}$. It follows from Lemma~\ref{lemma_I_v_cohomology} shows that $H^1(I_v, M) \cong M^{I_v^{(p)}}$ is free, as required. 
\end{proof}

\begin{proposition}\label{proposition_unramified_complex}
Suppose that the $R$-module $H^1(I_v, M)$ is free. Then the following claims are valid. 
    \item[1)] The $R$-module $M^{I_v}$ is free of finite rank.
    \item[2)] The complex $U_v^+(M)$ is a perfect complex, and $R\Gamma_{\mathrm{ur}}(G_v, M) \in D^{[0,1]}_{\rm parf}(_R\mathrm{Mod})$.  
    \item[3)]  For any ring homomorphism $R \to S$, the induced morphism $M^{I_v} \otimes_R S \to (M \otimes_R S)^{I_v}$ is an isomorphism. 
    In particular, $R\Gamma_{\mathrm{ur}}(G_v, M) \otimes^{\mathbb{L}}_{R} S \stackrel{\sim}{\longrightarrow} R\Gamma_{\mathrm{ur}}(G_v, M \otimes_R S)$. 
\end{proposition}
\begin{proof}

\item[1)] Since $M^{I_v^{(p)}}/(t-1)M^{I_v^{(p)}} \cong H^1(I_v, M)$ is free by assumption, $M^{I_v} = (M^{I_v^{(p)}})^{t=1}$ is also free. 

\item[2)] This portion is clear thanks to (1). 

\item[3)] Thanks to our running assumption, $R\Gamma(I_v, M)$ is represented by the perfect complex 
\[
\cdots\quad  \longrightarrow 0 \longrightarrow M^{I_v} \xrightarrow{0} H^1(I_v, M) \longrightarrow 0 \quad  \cdots, 
\]
where the differentials are the zero-maps. Since $R\Gamma(I_v, M) \otimes^{\mathbb{L}}_{R} S \stackrel{\sim}{\longrightarrow} R\Gamma(I_v, M \otimes_R S)$ by Lemma \ref{lemma_I_v_cohomology}, it follows that $R\Gamma(I_v, M \otimes_R S)$ is represented by the perfect complex 
\[
\cdots\quad  \longrightarrow 0 \longrightarrow M^{I_v} \otimes_R S \xrightarrow{0} H^1(I_v, M)  \otimes_R S \longrightarrow 0 \quad \cdots. 
\]
We conclude that $H^0(I_v, M \otimes_R S) \cong M^{I_v} \otimes_R S$, as required. 
\end{proof}

\
\subsubsection{Tamagawa numbers in Hida families}
\label{subsubsec_Hida_family_tamagawa_2022_09}
In \S\ref{subsubsec_Hida_family_tamagawa_2022_09}, we apply the general discussion in \S\ref{subsubsec_2022_09_09_1253} to study the behaviours of Tamagawa factors in the families of Galois representations we are interested in.

For notational simplicity, let us put $\rho_{\hf} := \rho_{T_{\hf}^\dagger}$.  For any arithmetic prime $\cP$ (namely, the kernel of an arithmetic specialisation), we set $T(f_\cP):=T_{\hf}^{\dagger}\otimes_{\cR_{\hf}} S_\cP$, where $S_\cP$ is the normalisation of $\cR_{\hf}/\cP$ (see also \cite{nekovar06}, \S12.7.4).

\begin{proposition}
\label{prop_I_v_infinite}\ 
    \item[1)] If $\ord_v(N_{\hf}) = 1$, then  $\rho_{\hf}(I_v)$ is infinite. 
    \item[2)] If $\rho_{\hf}(I_v)$ has infinite cardinality, then there is an exact sequence of $\cR_{\hf}[G_v]$-modules 
    \begin{align*}
            0 \longrightarrow \cR_{\hf}(1) \otimes \mu \longrightarrow T_{\hf}^{\dagger} \longrightarrow \cR_{\hf} \otimes \mu \longrightarrow 0,  
                \end{align*}
    where $\mu \colon G_v \to \{\pm 1\}$ is a quadratic character. 
    Moreover, $\ord_v(N_{\hf}) = 1$ if and only if $\mu$ is unramified. 
\end{proposition}

\begin{proof}
The first assertion follows from \cite[\S12.4.4.2 and Lemma 12.4.5]{nekovar06}, whereas the second from \cite[\S12.4.4 and Proposition 12.7.14.1]{nekovar06}. We remark that Proposition 12.7.14.1 of \cite{nekovar06} needs to be adjusted in this manner since we work with a twist verifying $\det(T_\f^{\dagger}) = R_\f(1)$. 
 \end{proof}

 \begin{corollary}\label{cor:tam-mult-reduction-case-gg^c}
Suppose that $\rho_{\hg}(I_v)$ has infinite cardinality and $\varepsilon_\hg = \mathds{1}$ is the trivial character.  Then the following are equivalent. 
     \item[1)] $H^1(I_v, T_{\hg}^\dagger \hatotimes_{\ZZ_p}T_{\hg^{c}}^\dagger)$ is free. 
     \item[2)] $H^1(I_v, T(g_{\cP}) \otimes_{S_{\cP}} T(g^{c}_{\cP}) )$ is free for some arithmetic prime $\cP$. 
     \item[3)] The Tamagawa factor for $ T(g_{\cP}) \otimes_{S_{\cP}} T(g^{c}_{\cP})$ equals $1$ for some arithmetic prime $\cP$. 
     \item[4)] The Tamagawa factor for $T(g_{\cP}) \otimes \nu$ equals $1$ for some arithmetic prime $\cP$ and any quadratic character $\nu \colon G_v \to \{\pm1\}$.  

Moreover, whenever these equivalent conditions are satisfied, then $H^1(I_w, T_{\hg}^\dagger)$ is free for any finite index subgroup $I_w$ of $I_v$ with $p \nmid [I_v \colon I_w]$.  
 \end{corollary}
 \begin{proof}
The condition (1) implies the condition (2) thanks to Lemma~\ref{lemma_I_v_cohomology}(3). 
 As we have remarked at the start of \S\ref{subsubsec_2022_09_08_1236}, the condition (2) implies the condition  (3).  
 
By Proposition \ref{prop_I_v_infinite}, we have an exact sequence of $G_v$-modules 
\[
 0 \longrightarrow \cR_{\hg}(1)\otimes \mu \longrightarrow T_{\hg}^\dagger \longrightarrow \cR_{\hg}\otimes \mu  \longrightarrow 0. 
\]
We will prove that the conditions (3) and (4) are equivalent, and that (3) implies (2). We shall treat the scenario where $\mu=\mathds{1}$ is the trivial character since the conditions (1)--(4) does not depend on the quadratic character $\mu$. Note that, in this case, $\mathrm{ord}_v(N_\hg) = 1$. 

Let us take a  basis $\{e_1, e_2\}$ of $T_{\hg}^\dagger$ so that 
 $\cR_{\hg}e_1 = \cR_{\hg}(1)$ and 
\[
\rho_{\hg}(t) = 
\begin{pmatrix}
1 & a
\\
0 & 1
\end{pmatrix}  
\]
 for some element $0 \neq a \in \cR_{\hg}$. Let us similarly choose a basis $\{e_1', e_2'\}$ of $T_{\hg^c}^\dagger$. We then have 
 \begin{align}\label{eq:t-gotimesg^c}
      \rho_{\hg \otimes \hg^c}(t) = 
\begin{pmatrix}
1 & a & a & a^2
\\
& 1 & & 2a
\\
&&1&2a
\\
 &&&1 
\end{pmatrix}  
  \end{align}
 with respect to the basis $\{e_1 \otimes e_1', e_1 \otimes e_2', e_2 \otimes e_1', e_2 \otimes e_2'\}$. 
 Let us define the $G_v$-module
 \[
 C := \cR_{\hg} \hatotimes \cR_{\hg}(e_1 \otimes e_1') + \cR_{\hg} \hatotimes \cR_{\hg}(e_1 \otimes e_2' + e_2 \otimes e_1')\,.
 \]
Let us take an arithmetic prime $\cP$. Since $\mathrm{ord}_v(N_\hg) = 1$, the image of the compositum 
$$I_v \longrightarrow \mathrm{GL}_2(\cR_\hg) \longrightarrow  \mathrm{GL}_2(\cR_\hg/\cP) \subset \mathrm{GL}_2(S_\cP)$$  
has infinite cardinality  (see also \cite{nekovar06}, \S12.4.4.2 and Proposition 12.7.14.1(ii)). We therefore infer that $a \not\in \cP$. 
 Then,
 \[
  H^1(I_v, T(g_{\cP}) \otimes_{S_{\cP}} T(g^{c}_{\cP}))_{\rm tors}=(C \otimes S_{\cP}/aS_{\cP})(-1)\,. 
 \]
The exact sequence of $\Gal(\QQ_v^{\rm ur}/\QQ_v)$-modules 
 \[
 0 \longrightarrow \cR_{\hg} \hatotimes \cR_{\hg}(1) \longrightarrow C(-1) \longrightarrow \cR_{\hg} \hatotimes \cR_{\hg} \longrightarrow 0, 
 \]
shows that 
\[
\det\left((C \otimes S_{\cP}/aS_{\cP})(-1) \xrightarrow{\mathrm{Fr}_v - 1 } (C \otimes S_{\cP}/aS_{\cP})(-1)\right) = 0\,.
\]
Since $S_{\cP}/aS_{\cP}$ is a finite ring (as $a\not\in \cP$), the condition (3) is equivalent to the requirement that $a \in \cR_{\hg}^\times$.

 We explain that the condition (4) is also equivalent to asking that $a \in \cR_{\hg}^\times$. Indeed, the Tamagawa factor for $T(g_{\cP}) \otimes \nu$ equals $1$ for any non-trivial quadratic character $\nu \colon G_v \to \{\pm1\}$ and, when $\nu=\mathds{1}$ is the trivial character, we have $H^1(I_v, T(g_{\cP}))_{\rm tors}^{\mathrm{Fr}_v = 1} = S_{\cP}/aS_{\cP}$. In particular, $H^1(I_v, T(g_{\cP}))$ is torsion-free (hence free) if and only if $a \in \cR_{\hg}^\times$.
 
 This shows that the conditions (3) and (4) are equivalent. Moreover, when $a \in \cR_{\hg}^\times$ (which is the case if (3) holds), the equation \eqref{eq:t-gotimesg^c} shows that the module $H^1(I_v, T_{\hg} \hatotimes_{\ZZ_p}T_{\hg^{c}})$ is free (of rank 2), which is the condition (1).  

The very final assertion follows from the fact that $I_w/(I_w \cap I_v^{(p)}) = I_v/I_v^{(p)}$ combined with our observation that $a \in \cR_{\hg}^\times$. 
 \end{proof}

The following  lemma follows from \S12.4.4 and Proposition 12.7.14.2 of \cite{nekovar06}. 

\begin{lemma}\label{lemma_rho_cP_classification}
Suppose that $\rho_{\hf}(I_v)$ is finite but $\rho_{\hf}(I_v) \neq 0$. 
For any arithmetic prime $\cP$, one of the following holds.
    \item[i)] There is a (ramified) character $\mu \colon G_v \to S_\cP^\times$ such that $T(f_\cP)|_{I_v} = \mu \oplus \mu^{-1}$.  
    \item[ii)] $T(f_\cP)$ is monomial, namely, $T(f_\cP) = \mathrm{Ind}_{G_w}^{G_v}(\mu)$, where $E_w/\QQ_v$ is a ramified quadratic extension with absolute Galois group $G_w$, and $\mu \colon G_w \to S_\cP^\times$ is a ramified character. 
    \item[iii)] $v = 2$ and there exists a Galois extension of $\QQ_v$ with Galois group isomorphic to $A_3$ or $S_3$, over which $T(f_\cP)$ becomes monomial. 
\end{lemma}


\begin{corollary}\label{corollary_Tamagawa_potentially_good_case}
Suppose that $\rho_{\hf}(I_v)$ is finite but $\rho_{\hf}(I_v) \neq 0$. 
Then the following conditions are equivalent. 
    \item[1)] The $\cR_{\hf}$-module $H^1(I_v, T_{\hf}^{\dagger})$ is free. 
    \item[2)] The $S_{\cP}$-module $H^1(I_v, T(f_\cP))$ is free for some arithmetic prime $\cP$. 
    \item[3)] $H^1(I_v, T(f_\cP)) = 0$ for some arithmetic prime $\cP$. 
    \item[4)] $H^1(I_v, T_{\hf}^{\dagger}) = 0$. 
\end{corollary}
\begin{proof}
Note that $T(f_\cP)^{I_v} = 0$ for any arithmetic prime $\cP$ by Lemma \ref{lemma_rho_cP_classification}, and hence the $S_{\cP}$-module $H^1(I_v, T(f_\cP))$ is torsion by Lemma \ref{lemma_I_v_cohomology}.2. 
This shows that the conditions (2) and (3) are equivalent. 
The implications (1) $\Rightarrow$ (2) and (3) $\Rightarrow$ (4) both follow from Lemma \ref{lemma_I_v_cohomology}.3. 
Finally, the implication (4) $\Rightarrow$ (1) is tautological. 
\end{proof}

\begin{corollary}\label{cor:tam_vanish_fin}
Suppose that $\rho_{\hf}(I_v)$ is finite and $\rho_{\hf}(I_v) \neq 0$. 
If the $\cR_{\hf}$-module $H^1(I_v, T_{\hf}^{\dagger})$ is free, then $(T_{\hf}^\dagger)^{I_v^{(p)}} = 0$. 
\end{corollary}

\begin{proof}
By Lemma \ref{lemma_I_v_cohomology}.3, it suffices to show that $T(f_\cP)^{I_v^{(p)}} = 0$ for some arithmetic prime $\cP$. 

We first consider the case (i)  in Lemma~\ref{lemma_rho_cP_classification}, so that  there is a (ramified) character $\mu \colon G_v \to S_\cP^\times$ with the property that $T(f_\cP)|_{I_v} = \mu \oplus \mu^{-1}$. Suppose that $\mu(I_v^{(p)}) = \{1\}$. By Corollary \ref{corollary_Tamagawa_potentially_good_case}, we have 
\[
0 = H^1(I_v, T(f_\cP)) = (S_{\cP}/(\mu(t)-1)S_{\cP})^{\oplus 2}, 
\]
and hence, $\mu(t)-1 \in S_{\cP}^\times$. 
Since $t \in I_v/I_v^{(p)} \cong \ZZ_p$ and $\#\rho_{\hf}(I_v) < \infty$,  $\mu(t)$ is a $p$-power root of unity. 
This contradicts the fact that $\mu(t)-1 \in S_{\cP}^\times$ and shows that $\mu(I_v^{(p)}) \neq \{1\}$, as required.

In the situation of (ii) or (iii) of Lemma~\ref{lemma_rho_cP_classification}, we note that $I_v^{(p)}$ acts non-trivially on $T(f_\cP)$. 
Thence, if $T(f_\cP)^{I_v^{(p)}} \neq 0$, then $T(f_\cP)^{I_v^{(p)}} = \nu$, where $\nu$ is a character of  $I_v/I_v^{(p)}$. 
Since $H^1(I_v, T(f_\cP)) = 0$ by Corollary \ref{corollary_Tamagawa_potentially_good_case}, we have 
 $$(t-1)T(f_\cP)^{I_v^{(p)}} = T(f_\cP)^{I_v^{(p)}},$$ 
 which implies $\nu(t) - 1 \in S_{\cP}^\times$. 
Since $\nu(t)$ is a $p$-power root of unity, this cannot happen, proving that $T(f_\cP)^{I_v^{(p)}}= 0$ also in these scenarios.
\end{proof}

\begin{proposition}[{\cite{nekovar06}, Proposition 12.4.10.3}]
Let $h$ be a newform and $L$ the number field generated by the Hecke eigenvalues of $h$. 
Suppose that $\rho_{h}(I_v)$ is finite and $\rho_{h}(I_v) \neq 0$. 
If $\zeta_p + \zeta_p^{-1} \not\in L$, then $H^1(I_v, T(h)) = 0$. 
\end{proposition}

\begin{remark}\label{rem:nek_prime_to_p}
In the proof of \cite[Proposition 12.4.10.3]{nekovar06}, Nekov\'a\v{r}  shows along the way that $p \nmid \# \rho_h(I_v)$. We review his argument in this remark.

Nekov\'a\v{r} proves in \cite[Proposition 12.4.10.3]{nekovar06} that for any non-trivial element $g \in \rho_h(I_v)$, there is an integer $n$ coprime to $p$ such that the eigenvalues of $g^n$ are 1. 
This shows that the finite order element $g^n\in \mathrm{GL}_2(L)$ is unipotent, and hence $g^n = 1$ since $ \# \rho_h(I_v) < \infty$. 
This observation shows that $p \nmid \# \rho_h(I_v)$. 
\end{remark}

\begin{proposition}\label{prop_additive_tamagawa_vanish_1}
Suppose that $\rho_{\hf}(I_v)$ is finite and $\rho_{\hf}(I_v) \neq 0$. 
If $p \nmid v - 1$ and $v \neq 2$, then $p \nmid \# \rho_{\hf}(I_v)$ and 
$H^1(I_v, T_{\hf}^{\dagger}) = 0$. 
\end{proposition}
\begin{proof}
Let us take an arithmetic prime $\cP$. 
Since $p \nmid v - 1$, for any ramified quadratic extension $E_w/\QQ_v$, the group $\cO_{E_w}^\times$ is $p$-divisible. Lemma~ \ref{lemma_rho_cP_classification} then shows that $p \nmid \#\rho_{T(f_{\cP})}(I_v)$, and we infer from Corollary \ref{corollary_free_potentially_good} that $H^1(I_v, T(f_\cP))$ is free. 
We then conclude by Corollary \ref{corollary_Tamagawa_potentially_good_case} that $H^1(I_v, T_{\hf}^{\dagger}) = 0$, as required. 
\end{proof}

\begin{proposition}\label{prop_additive_tamagawa_vanish_2}
Suppose that $\rho_{\hf}(I_v)$ is finite and $\rho_{\hf}(I_v) \neq 0$. 
If $p \neq 3, 7$ and $v = 2$, then $p \nmid \# \rho_{\hf}(I_v)$ and $H^1(I_v, T_{\hf}^{\dagger}) = 0$. 
\end{proposition}
\begin{proof}
Let us take an arithmetic prime $\cP$. It suffices to show that $p \nmid \#\rho_{T(f_{\cP})}(I_v)$ and $H^1(I_v, T(f_\cP))$ is free. 
We only consider the case (iii) in Lemma \ref{lemma_rho_cP_classification}. 
Since $p \nmid 63 = 2^6-1$, the same argument of the proof of Proposition \ref{prop_additive_tamagawa_vanish_1} shows that $p \nmid \#\rho_{T(f_{\cP})}(I_v)$, and Corollary \ref{corollary_free_potentially_good} shows that $H^1(I_v, T(f_\cP))$ is free. 
 \end{proof}

We summarize the results we have just proved in \S\ref{subsubsec_Hida_family_tamagawa_2022_09}. 

\begin{corollary}\label{cor:summary_tam_add}
Let $v$ be a prime such that $\# \rho_{\hf}(I_v) < \infty$ and $\rho_{\hf}(I_v) \neq 0$. 
Suppose that one of the following conditions holds:
\item[i)] $v \neq 2$ and $p \nmid v-1$. 
\item[ii)] $v=2$ and $p \neq 3,7$. 
\item[iii)] $\zeta_{p} + \zeta_{p}^{-1} \not\in L_f$, where $L_f$ denote the number field generated by the Hecke eigenvalues of $\hf_\kappa^\circ$ for some arithmetic specialisation $\kappa \in \cW_\hf^{\rm cl}$. 

Then we have $p \nmid \# \rho_{\hf}(I_v)$ and $(T_\f^\dagger)^{I_v^{(p)}} = 0$. In particular, $H^1(I_v, T_\f^\dagger) = 0$. 
\end{corollary}




\begin{assumption}\label{assumption_tamagawa}
For any (rational) prime $v$ such that $0 < \# \rho_{\hf}(I_v) < \infty$, one of the following holds:
\begin{itemize}
\item[i)] $v \neq 2$ and $p \nmid v-1$.
\item[ii)] $v=2$ and $p \neq 3,7$. 
\item[iii)] $\zeta_{p} + \zeta_{p}^{-1} \not\in L_f$.  
\end{itemize}
\end{assumption}

\begin{remark}
If $0 < \# \rho_{\hf}(I_v) < \infty$, then $v^2 \mid N_\f$. In particular, when $N_\f$ is square-free, Assumption \ref{assumption_tamagawa} is satisfied. 
\end{remark}


\begin{proposition}
Suppose that Assumption~\ref{assumption_tamagawa} holds both for $\hf$ and $\hg$. Assume also that the $p$-part of the Tamagawa factor for a single member of the Hida family $\hf$ and a single member of the family $\hg$ equals $1$. 
Then $H^1(I_v,M)$ is free for any $M \in \{T_3^{\dagger}, T_2^{\dagger}, M_2^{\dagger}\}$. 
\end{proposition}

\begin{proof}
We note that $T_2^{\dagger} = T_3^{\dagger} \otimes_{\cR_3} \cR_2$ and $T_2^{\dagger} = M_2^{\dagger} \oplus (T_{\hf}^{\dagger} \otimes_{\cR_{\hf}} \cR_2)$. 
By Lemma \ref{lemma_I_v_cohomology}, we only need to consider the case that $M = T_3^{\dagger}$. We may assume without loss of generality that $\varepsilon_\hg = \mathds{1}$.

If  $0 < \#\rho_{\hf}(I_v) < \infty$ and $0 < \#\rho_{\hg}(I_v) < \infty$, then $p\nmid\#\rho_{\hf}(I_v)$ and $p\nmid\#\rho_{\hg}(I_v)$ by  Corollary \ref{cor:summary_tam_add}. 
Since $T_3^{\dagger} = T_{\hf}^{\dagger} \hatotimes_{\ZZ_p} T_{\hg}^{\dagger} \hatotimes_{\ZZ_p}T_{\hg^{c}}^{\dagger}$ as $I_{v}$-modules, we see that $\#\rho_{T_3^{\dagger}}(I_v) < \infty$ and  $p\nmid\#\rho_{T_3^{\dagger}}(I_v)$. 
This fact combined Corollary \ref{corollary_free_potentially_good} show that $H^1(I_v, T_3^{\dagger})$ is free.

Next, we consider the case that $0 < \#\rho_{\hf}(I_v) < \infty$ and $\#\rho_{\hg}(I_v) = \infty$. 
In this case, $I_v^{(p)}$ acts trivially on $T_{\hg}^{\dagger} \hatotimes_{\ZZ_p}T_{\hg^{c}}^{\dagger}$ by Proposition \ref{prop_I_v_infinite} (note that the character $\mu$ in Proposition \ref{prop_I_v_infinite}.2 is quadratic).
We therefore have
\[
(T_3^{\dagger})^{I_v^{(p)}} = (T_{\hf}^{\dagger})^{I_v^{(p)}} \hatotimes_{\ZZ_p} T_{\hg}^{\dagger}\hatotimes_{\ZZ_p} T_{\hg^{c}}^{\dagger} = 0\,,
\]
where to vanishing holds thanks to Corollary~\ref{cor:summary_tam_add}. In particular, $H^1(I_v, T_3^{\dagger}) = 0$.

Suppose next that $\#\rho_{\hf}(I_v) = \infty$ and $0 < \#\rho_{\hg}(I_v) < \infty$. 
By  Corollary \ref{cor:summary_tam_add}, there is a finite-index subgroup $I_w$ of $I_v$ such that $p \nmid [I_v \colon I_w]$ and $I_w$ acts trivially on $T_{\hg}^{\dagger}\hatotimes_{\ZZ_p} T_{\hg^{c}}^{\dagger}$. 
Then we have 
\[
H^1(I_v, T_3^{\dagger}) = (H^1(I_w, T_{\hf}^{\dagger} ) \hatotimes_{\ZZ_p} T_{\hg}^{\dagger} \hatotimes_{\ZZ_p}T_{\hg^{c}}^{\dagger})^{I_v/I_w}. 
\]
By Corollary \ref{cor:tam-mult-reduction-case-gg^c}, the module $H^1(I_w, T_{\hf}^{\dagger})$ is free, and hence $H^1(I_v, T_3^{\dagger})$ is also free since $p\nmid[I_v \colon I_w]$.

Suppose that $\#\rho_{\hf}(I_v) = \infty$ and $\#\rho_{\hg}(I_v) = \infty$. 
Since $(T_3^{\dagger})^{I_v^{(p)}} = (T_{\hf}^{\dagger})^{I_v^{(p)}} \hatotimes_{\ZZ_p} T_{\hg}^{\dagger}\hatotimes_{\ZZ_p} T_{\hg^{c}}^{\dagger}$ (as $I_v^{(p)}$ acts trivially on $T_{\hg}^{\dagger} \hatotimes_{\ZZ_p}T_{\hg^{c}}^{\dagger}$ by Proposition \ref{prop_I_v_infinite}), we may assume without loss of generality that 
$I_v^{(p)}$ acts trivially on $T_{\hf}^{\dagger}$ and $T_{\hg}^{\dagger}$ by Proposition \ref{prop_I_v_infinite}. 
In this case, Proposition \ref{prop_I_v_infinite} tells us that 
\[
\rho_{?}(t) = 
\begin{pmatrix}
1 & a_{?}
\\
0 & 1
\end{pmatrix}  \,,\qquad ? \in \{\f, \hg, \hg^c\}
\]
in $\GL(T_{?}^{\dagger}) \cong \GL_2(\cR_{?})$. 
Moreover, we have explained in the proof of Corollary \ref{cor:tam-mult-reduction-case-gg^c} that our running assumptions on the Tamagawa factors is equivalent to the requirement that $a_{?} \in \cR_{?}^\times$. A direct computation shows that the double coset 
$\GL_8(\cR_{3})\ (\rho_{T_3^{\dagger}}(t) - 1)\ \GL_8(\cR_{3})$ is represented by the matrix 
\[
\begin{pmatrix}
0 & 1 & 0&&&&&
\\
 &  0& 0&1&&&&
\\
 & & 0&0&1&&&
\\
 & & &0&0&1&&
\\
 &  & &&0&0&1&
\\
 &&&&&0&0&1
\\
&& &&&&0&0
\\
&  & &&&&&0
\end{pmatrix}\,. 
\]
This shows that $H^1(I_v,T_3^{\dagger})$ is free, as required. 

If $\#\rho_{\hf}(I_v) = 0$, then $H^1(I_v, T_3^{\dagger}) = T_\f^{\dagger} \hatotimes_{\ZZ_p} H^1(I_v, T_{\hg}^{\dagger}\hatotimes_{\ZZ_p} T_{\hg^{c}}^{\dagger})$ is free by Corollary \ref{cor:tam-mult-reduction-case-gg^c}. 
If $\#\rho_{\hg}(I_v) = 0$, then $H^1(I_v, T_3^{\dagger}) = H^1(I_v, T_\f^{\dagger}) \hatotimes_{\ZZ_p} T_{\hg}^{\dagger}\hatotimes_{\ZZ_p} T_{\hg^{c}}^{\dagger}$ is free, still by Corollary \ref{cor:tam-mult-reduction-case-gg^c}.

The proof of our proposition is complete.
\end{proof}


\subsection{Determinants of perfect complexes, characteristic ideals and algebraic \texorpdfstring{$p$}{}-adic \texorpdfstring{$L$}{}-functions}

Our goal in this section is to introduce the module of algebraic $p$-adic $L$-functions in terms of our Selmer complexes. This will require to establish that the relevant Selmer complexes are perfect, which will crucially rely on input from \S\ref{subsubsec_2022_09_08_1236}.

\subsubsection{Determinants of perfect complexes}

We start with the general definition of the determinant for a perfect complex following \cite{KM76}. Let $R$ be a Noetherian ring. 

\begin{defn}
For any bounded complex $P^\bullet$ of finitely generated projective $R$-modules, we define the determinant $\det(P^\bullet)$ of $P^\bullet$ on setting 
\[
\det(P^\bullet) := \bigotimes_{i \in \ZZ} \det(P^i)^{(-1)^{i}}.  
\]
See \cite[\href{https://stacks.math.columbia.edu/tag/0FJ9}{Tag 0FJ9}]{stacks-project} for the definition of the determinant of a finite projective module. For any $R$-module $M$ with finite projective dimension, we define the determinant of $M$ by 
\[
\det(M) := \det(P^\bullet), 
\]
where $P^\bullet$ is a finite projective resolution of $M$. 
\end{defn}


The following lemma is clear from the definition of the determinant. 

\begin{lemma}
For any integer $n \in \ZZ$ and any bounded complex $P^{\bullet}$ of finite projective $R$-modules, 
we have $\det(P^{\bullet}[n]) = \det(P^{\bullet})^{(-1)^n}$. 
\end{lemma}

For any  bounded acyclic complex $P^\bullet$ of finite projective $R$-modules, 
we have a natural identification  
\[
\det(P^\bullet) = R 
\]
in the following way. 
Since $P^\bullet$ is acyclic and $P^{j}$ is projective for any $j \in \ZZ$, one can decompose 
\[
P^\bullet = \bigoplus  P_i^{\bullet}[n_i], 
\]
where $n_i \in \ZZ$ and $P_i^{\bullet}$ is the acyclic complex of finite projective $R$-modules  
\[
\cdots \longrightarrow 0 \longrightarrow P_i^{-1} \xrightarrow{d_i^{-1}} P_i^{0} \longrightarrow 0 \longrightarrow \cdots  
\]
concentrated in degrees $-1$ and $0$. 
Using the canonical isomorphism  
\[
\det(P_i^{-1})^{-1} \otimes  \det(P_i^{0}) \stackrel{\sim}{\longrightarrow} R, 
\]
induced by the isomorphism $d_i^{-1}$ gives rise to the natural identification $\det(P_i^{\bullet}) = R$ for each $i$. Since $$\det(P^\bullet) = \bigotimes \det(P_i^{\bullet})^{(-1)^{n_i}},$$ 
we obtain a natural identification  
$\det(P^\bullet) = R$, which does not depend  the choice of the decomposition $P^\bullet = \bigoplus  P_i^{\bullet}[n_i]$. 

\begin{lemma}\label{lemma_determinant}
Let $P_i^\bullet$ be a bounded complex of finite projective $R$-modules. 
\item[1)] If we have a quasi-isomorphism $P_1^\bullet \to P_2^\bullet$, then we have  a natural identification  $\det(P_1^\bullet)  =  \det(P_2^\bullet)$. 
\item[2)] If we have an exact triangle  
\[
P_1^\bullet \longrightarrow P_2^\bullet \longrightarrow P_3^\bullet \longrightarrow P_1^\bullet[1],  
\]
in the derived category, then there exists a natural identification  $\det(P_2^\bullet) = \det(P_1^\bullet) \otimes_R \det(P_3^\bullet)$. 
\end{lemma}
Using Lemma~\ref{lemma_determinant}, we define the determinant $\det(C^\bullet)$ for any perfect complex of $R$-modules. 
\begin{proof}
\item[1)] If we denote by $P_3^{\bullet}$ the mapping cone of $P_1^\bullet \to P_2^\bullet$, then  $P_3^{\bullet}$ is acyclic. 
Since $P_3^{i} = P_1^{i+1} \oplus P_2^i$ by the definition  of the mapping cone, we have $R = \det(P_3^{\bullet}) = \det(P_1^{\bullet})^{-1} \otimes \det(P_2^{\bullet})$, and we deduce that $\det(P_1^{\bullet}) = \det(P_2^{\bullet})$. 

\item[2)] This portion is immediate after part (1). 
\end{proof}

Throughout this paper, 
for any perfect complex $C^\bullet$ of $R$-modules with torsion cohomology modules, 
we regard $\det(C^\bullet)$ as an (invertible) fractional ideal of $\mathrm{Frac}(R)$ via the injection
\[
\det(C^\bullet) \longrightarrow \det(C^\bullet) \otimes_R \mathrm{Frac}(R) = \det(C^\bullet \otimes^{\mathbb{L}}_R \mathrm{Frac}(R)) = \mathrm{Frac}(R). 
\]

\begin{example}
For any regular element $r \in R$, we have 
\[
\det(R/r R) = \det(\longrightarrow 0 \longrightarrow R \stackrel{\times r}{\longrightarrow} R \longrightarrow 0 \longrightarrow)  = rR. 
\]
\end{example}

\begin{proposition}
For any exact sequence 
\[
0 \longrightarrow M_1 \longrightarrow M_2 \longrightarrow M_3 \longrightarrow 0  
\]
of torsion $R$-modules with finite projective dimension, we have an equality of ideals 
\[
\det(M_2) = \det(M_1) \det(M_3). 
\]
\end{proposition}
\begin{proof}
This is an immediate consequence of Lemma \ref{lemma_determinant}. 
\end{proof}

\begin{lemma}[{\cite{sakamoto_tot_real}, Remark B.2 and Lemma C.11}]\label{lem:reflexible}
Suppose that $R$ is normal. 
Let $I$ and $J$ be reflexive ideals  of $R$. 
If $IR_\fp = JR_\fp$ for any prime $\fp$ of $R$ with $\mathrm{ht}(\fp) \leq 1$, then $I=J$. 
\end{lemma}


\begin{proposition}
Suppose that $R$ is normal. 
For any torsion $R$-module $M$ with finite projective dimension, we have $\det(M) = \mathrm{char}_R(M)$. 
\end{proposition}
\begin{proof}
Note that any invertible ideal is reflexive.
We may therefore assume by Lemma \ref{lem:reflexible} that $R$ is a discrete valuation ring. 
Let $\pi \in R$ be a generator of the maximal ideal of $R$. 
We then have an isomorphism $M \cong \bigoplus R/\pi^{n_i}R$, and hence
\begin{align*}
    \det(M) = \prod \det(R/\pi^{n_i}R)
    = \pi^{\sum n_i}R 
    = \mathrm{char}_R(M). 
\end{align*}
\end{proof}

\begin{corollary}\label{cor:det=char}
Suppose that $R$ is normal. 
Let $C^\bullet$ be a perfect complex. 
If the $R$-module $H^i(C^\bullet)$ is torsion for any $i \in \ZZ$ and has a finite projective dimension, then 
\[
\det(C^\bullet) = \prod_{i \in \ZZ} \mathrm{char}_R(H^i(C^\bullet))^{(-1)^i}. 
\]
\end{corollary}

\subsubsection{Algebraic $p$-adic $L$-functions}
\label{subsubsection_alg_padic_L_2022_09_09}

Let $R$ be a complete Noetherian local ring with finite residue field of characteristic $p \geq 3$ and let $T$ be a free $R$-module of finite rank with a continuous $G_{\QQ, \Sigma}$-action. 
Suppose that 
we have an $R[G_p]$-submodule $F^+T$ of $T$ such that the quotient $F^-T := T/F^+T$ is free as an $R$-module. 
We then have the associated Greenberg local conditions $\Delta := \Delta_{F^+}$, cf. \S\ref{subsubsec_311_17_05_2021}. 

Let $\mathfrak{m}_R$ denote the maximal ideal of $R$ and let us put $\overline{T} := T \otimes_R R/\mathfrak{m}_R$. 
We assume in \S\ref{subsubsection_alg_padic_L_2022_09_09} that 
\begin{itemize}
     \item[\mylabel{item_Irr_general}{\bf ($H^0$)}] $\overline{T}^{G_\QQ} =0= (\overline{T}^\vee(1))^{G_\QQ}$\,, 
     \item[\mylabel{item_Tam}{(\bf Tam)}] $H^1(I_v, T)$ is free for any prime $v \in \Sigma^p := \Sigma \setminus \{p\}$. 
\end{itemize}

\begin{proposition}\label{prop_[1,2]_parf}\ 
    \item[i)] For any ideal $I$ of $R$, 
    the complex $\widetilde{R\Gamma}_{\rm f}(G_{\QQ, \Sigma}, T/IT, \Delta)$ is perfect and 
    we have a natural isomorphism 
    \[
    \widetilde{R\Gamma}_{\rm f}(G_{\QQ, \Sigma}, T, \Delta) \otimes_R^{\mathbb{L}} R/I \stackrel{\sim}{\longrightarrow} \widetilde{R\Gamma}_{\rm f}(G_{\QQ, \Sigma}, T/IT, \Delta).
    \]
    \item[ii)] $\widetilde{R\Gamma}_{\rm f}(G_{\QQ, \Sigma}, T, \Delta) \in D^{[1,2]}_{\mathrm{parf}}(_{R}\mathrm{Mod})$. 
\end{proposition}

\begin{proof}
Let $I$ be an ideal of $R$. 
Let us set
\[
U^{-}_{\Sigma}(T/IT) := \mathrm{Cone}\left(U^{+}_{\Sigma}(T/IT) \xrightarrow{-i^+_{\Sigma}} C^\bullet_{\Sigma}(T/IT) \right). 
\]
Since we assume that $H^1(I_v, T)$ is free for any prime $v \in \Sigma^p$, $U_{\Sigma}^{-}(T/IT)$ is a perfect complex by Proposition \ref{proposition_unramified_complex} for any $I$, and 
$U_{\Sigma}^{-}(T) \otimes_R^{\mathbb{L}}R/I \stackrel{\sim}{\longrightarrow} U_{\Sigma}^{-}(T/IT)$. 
Moreover, we have an exact triangle 
\[
\widetilde{R\Gamma}_{\rm f}(G_{\QQ, \Sigma}, T/IT, \Delta) \longrightarrow R\Gamma(G_{\QQ, \Sigma}, T/IT) \longrightarrow U^{-}_{\Sigma}(T/IT) \xrightarrow{+1}
\]
which shows that the homomorphism $\widetilde{R\Gamma}_{\rm f}(G_{\QQ, \Sigma}, T, \Delta) \otimes_R^{\mathbb{L}} R/I \longrightarrow \widetilde{R\Gamma}_{\rm f}(G_{\QQ, \Sigma}, T/IT, \Delta)$ is an isomorphism. This completes the proof of (i).

To prove (ii), we note that we have, by definition,  $\widetilde{H}^{i}_{\rm f}(G_{\QQ, \Sigma}, T, \Delta) = 0$ for any $i \geq 4$.   
Let us prove that $\widetilde{H}_{\rm f}^3(G_{\QQ, \Sigma}, T, \Delta) = 0$. 
Since $(\overline{T}^\vee(1))^{G_\QQ} = 0$, global duality shows that the homomorphism 
\[
H^2(G_{\QQ, \Sigma}, T) \longrightarrow \bigoplus_{v \in \Sigma} H^2(G_{v}, T)
\]
is surjective. 
Since $H^2(G_{p}, T) \longrightarrow H^2(G_{p}, F^-T)$ is surjective, we have 
\[
\widetilde{H}_{\rm f}^3(G_{\QQ, \Sigma}, T, \Delta) = \mathrm{coker}\left(H^2(G_{\QQ, \Sigma}, T) \longrightarrow H^2(G_{p}, F^-T) \oplus \bigoplus_{v \in \Sigma \setminus \{p\}} H^2(G_{v}, T)\right) = 0\,.
\]
This shows that $\widetilde{R\Gamma}_{\rm f}(G_{\QQ, \Sigma}, T, \Delta) \in D^{[a,2]}_{\mathrm{parf}}(_{R}\mathrm{Mod})$ 
for some integer $a \leq 2$. 
To complete the proof, it is enough to show that $H^{i}(\widetilde{R\Gamma}_{\rm f}(G_{\QQ, \Sigma}, T, \Delta) \otimes^{\mathbb{L}}_{R} R/\mathfrak{m}_R) = 0$ for any $i \leq 0$
(cf. the proof of \cite{BS2020}, Proposition 2.7). By Part (i), we have 
\[
H^{i}(\widetilde{R\Gamma}_{\rm f}(G_{\QQ, \Sigma}, T, \Delta) \otimes^{\mathbb{L}}_{R} R/\mathfrak{m}_R)
= 
\widetilde{H}_{\rm f}^i(G_{\QQ, \Sigma}, \overline{T}, \Delta). 
\]
Thence, 
$H^{i}(\widetilde{R\Gamma}_{\rm f}(G_{\QQ, \Sigma}, T, \Delta) \otimes^{\mathbb{L}}_{R} R/\mathfrak{m}_R) = 0$ for any $i < 0$ and 
\[
H^{0}(\widetilde{R\Gamma}_{\rm f}(G_{\QQ, \Sigma}, T, \Delta) \otimes^{\mathbb{L}}_{R} R/\mathfrak{m}_R) 
= \widetilde{H}_{\rm f}^0(G_{\QQ, \Sigma}, \overline{T}, \Delta) \subset \overline{T}^{G_\QQ} = 0
\]
by \ref{item_Irr_general}. The proof of our proposition is complete.
\end{proof}

\begin{defn}
Whenever the Euler--Poincar\'e characteristic $\chi(\widetilde{R\Gamma}_{\rm f}(G_{\QQ, \Sigma}, T, \Delta))$ vanishes,  we call the invertible module $\det(\widetilde{R\Gamma}_{\rm f}(G_{\QQ, \Sigma}, T, \Delta))$  the module of algebraic $p$-adic $L$-functions for $(T, \Delta)$. 
\end{defn}

\begin{remark}
The following is the list of pairs of a Galois representation and a Greenberg local condition, where all the Euler-Poincar\'e characteristics of the corresponding Selmer complexes are zero:  
\[
(T_3^\dagger, \Delta_{F^+_\hf} )\,,\quad (T_3^\dagger, \Delta_{F^+_\hg})\,,\quad (T_3^\dagger, \Delta_{F^+_{\mathrm{bal}}})\,,\quad
(T_2^\dagger, \Delta_{F^+_\hf} )\,,\quad (T_2^\dagger, \Delta_{F^+_\hg})\,,\quad (T_2^\dagger, \Delta_{F^+_{\mathrm{bal}}})\,,\quad
\]
\[
(M_2^\dagger, \Delta_{F^+_\hf} )\,,\quad  
(M_2^\dagger, \Delta_{F^+_{\mathrm{bal}}})\,,\quad
(M_2^\dagger, \Delta_{F^+_{\mathrm{bal}}})\,,\quad
(T_\hf^\dagger, \Delta_{F^+_\hf} )\,. 
\]
\end{remark}

In view of the Iwasawa Main Conjectures for $T$, the following corollary justifies the nomenclature ``the module of algebraic $p$-adic $L$-functions'' for $(T, \Delta)$ that we have chosen for $\det(\widetilde{R\Gamma}_{\rm f}(G_{\QQ, \Sigma}, T, \Delta))$.

\begin{corollary}\label{corollary_normal_det=char}
Suppose that $R$ is normal and $\chi(\widetilde{R\Gamma}_{\rm f}(G_{\QQ, \Sigma}, T, \Delta)) = 0$. 
If moreover the $R$-module $\widetilde{H}^2_{\rm f}(G_{\QQ, \Sigma}, T, \Delta)$ is torsion, then $\widetilde{H}^1_{\rm f}(G_{\QQ, \Sigma}, T, \Delta)$ vanishes,  the projective dimension of $\widetilde{H}^2_{\rm f}(G_{\QQ, \Sigma}, T, \Delta)$ equals to 1, and 
\[
\det(\widetilde{R\Gamma}_{\rm f}(G_{\QQ, \Sigma}, T, \Delta)) = \mathrm{char}_R(\widetilde{H}^2_{\rm f}(G_{\QQ, \Sigma}, T, \Delta)). 
\]
\end{corollary}
\begin{proof}
By Proposition \ref{prop_[1,2]_parf} and the fact that $\chi(\widetilde{R\Gamma}_{\rm f}(G_{\QQ, \Sigma}, T, \Delta)) = 0$, we have an exact sequence 
\[
0 \longrightarrow 
\widetilde{H}^1_{\rm f}(G_{\QQ, \Sigma}, T, \Delta) \longrightarrow P \longrightarrow P \longrightarrow
\widetilde{H}^2_{\rm f}(G_{\QQ, \Sigma}, T, \Delta) \longrightarrow 0, 
\]
where $P$ is a finitely generated projective $R$-module. Since $\widetilde{H}^2_{\rm f}(G_{\QQ, \Sigma}, T, \Delta)$ is torsion, 
$P \otimes \mathrm{Frac}(R) \longrightarrow P \otimes \mathrm{Frac}(R)$ is an isomorphism, and hence $\widetilde{H}^1_{\rm f}(G_{\QQ, \Sigma}, T, \Delta) = 0$ 
and the projective dimension of $\widetilde{H}^2_{\rm f}(G_{\QQ, \Sigma}, T, \Delta)$ is 1. 
By Corollary \ref{cor:det=char}, we have $\det(\widetilde{R\Gamma}_{\rm f}(G_{\QQ, \Sigma}, T, \Delta)) = \mathrm{char}_R(\widetilde{H}^2_{\rm f}(G_{\QQ, \Sigma}, T, \Delta))$, as required.
\end{proof}


\section{Modules of leading terms}
\label{sec_Koly_Sys_Dec_4_18_05_2021}
Our goal in this section is to introduce the module of leading terms of algebraic $p$-adic $L$-functions, dwelling on ideas borrowed from the general theory of Kolyvagin systems. In \S\ref{sec_Koly_Sys_Dec_4_18_05_2021}, we shall work in great level of generality and apply our results in \S\ref{subsec_KS} to the factorisation problem at hand. 

We also refer the reader to Remark~\ref{rmk:comparisonkolyvagin} for the connection of the module of leading terms to $p$-adic $L$-functions.

\subsection{The set up}
Let $R$ be a  complete Gorenstein local ring with finite residue field of characteristic $p\geq3$ and $T$ be a free $R$-module of finite rank with a continuous $G_{\QQ, \Sigma}$-action. 
Throughout \S\ref{sec_Koly_Sys_Dec_4_18_05_2021}, we assume that 
we have an $R[G_p]$-submodule $F^+T$ of $T$ such that the quotient $T/F^+T$ is free as an $R$-module. 
Then we have the Greenberg local conditions $\Delta := \Delta_{F^+}$; cf. \S\ref{subsubsec_311_17_05_2021}. 
We also assume throughout \S\ref{sec_Koly_Sys_Dec_4_18_05_2021} that the condition \ref{item_Tam} holds. For simplicity, let us put 
\[
r=r(T,\Delta) := - \chi(\widetilde{R\Gamma}_{\rm f}(G_{\QQ, \Sigma}, T, \Delta)) = \mathrm{rank}_R(T^{c=-1}) - \mathrm{rank}_R(F^-T) \in \ZZ. 
\]

\subsection{Construction of special elements in the extended Selmer module}


In this subsection, we construct  a special cyclic submodule (module of leading terms)
\[
\delta(T, \Delta) \subset {\bigcap}^{r}_{R}\widetilde{H}^1_{\rm f}(G_{\QQ, \Sigma}, T, \Delta)
\]
(assuming that $r\geq 0$) and study its properties. 
Here, for any (commutative) ring $S$ and any finitely generated $S$-module $M$, we define 
\[
M^* := \Hom_{S}(M, S) \,\,\, \textrm{ and } \,\,\, {\bigcap}^t_{S}M := \left({\bigwedge}^t_{S}M^*\right)^*. 
\]
We call ${\bigcap}^t_{S}M$ the $t$-th exterior bi-dual of $M$. 

\subsubsection{} We first recall the basic facts concerning the exterior bi-duals following \cite{BS19}. 

\begin{defn}
For any $S$-homomorphism $\varphi \colon M \longrightarrow N$ and any integer $t \geq 1$, we define 
\begin{align*}
    \widetilde{\varphi} \colon {\bigwedge}^t_S M &\longrightarrow  N \otimes_{S} {\bigwedge}^{t-1}_S M\\
     m_1 \wedge \cdots \wedge m_t &\longmapsto \sum_{i=1}^t(-1)^{i-1}\varphi(m_i) \otimes m_1 \wedge \cdots \wedge m_{i-1 } \wedge m_{i+1} \wedge \cdots \wedge  m_t. 
\end{align*}
\end{defn}

\begin{proposition}[{\cite{BS19}, Lemma B.2}]\label{proposition_exterior_wedge}
Suppose that $t \geq 1$ and we have an exact sequence of $R$-modules 
\[
0 \longrightarrow M \longrightarrow P^1 \stackrel{\varphi}{\longrightarrow} P^2, 
\]
where $P^1$ and $P^2$ are finitely generated free $R$-modules. 
Then the canonical map 
\[
{\bigcap}^t_{R}M \longrightarrow {\bigcap}^t_{R}P^1 = {\bigwedge}^t_{R}P^1
\]
induces an isomorphism 
\[
{\bigcap}^t_{R}M \stackrel{\sim}{\longrightarrow} \ker\left({\bigwedge}^t_{R}P^1 \stackrel{\widetilde{\varphi}}{\longrightarrow} P^2 \otimes_R  {\bigwedge}^{t-1}_{R}P^1\right). 
\]
\end{proposition}

Since $\widetilde{R\Gamma}_{\rm f}(G_{\QQ, \Sigma}, T, \Delta) \in D^{[1,2]}_{\mathrm{parf}}(_{R}\mathrm{Mod})$ by Proposition \ref{prop_[1,2]_parf}, we have 
an exact sequence of $R$-modules 
\begin{align}
\begin{split}\label{exact:stark}
    0 \longrightarrow \widetilde{H}^1_{\rm f}(G_{\QQ, \Sigma}, T, \Delta) \longrightarrow R^{a+r} 
\stackrel{\varphi}{\longrightarrow} R^{a} \longrightarrow \widetilde{H}^2_{\rm f}(G_{\QQ, \Sigma}, T, \Delta) \longrightarrow 0. 
\end{split}
\end{align}
We put $\varphi := (\varphi_1, \ldots, \varphi_a) \colon R^{a+r} 
\to R^{a}$, 
and obtain an $R$-homomorphism 
\[
\widetilde{\varphi}_a \circ \cdots \circ \widetilde{\varphi}_1\, \colon\quad 
{\bigwedge}^{a+r}_{R}R^{a+r} \longrightarrow {\bigwedge}^{r}_{R}R^{a+r}. 
\]
If $r \geq 1$, then Proposition \ref{proposition_exterior_wedge} shows that 
\begin{align}\label{eq_image_bi-dual}
    \mathrm{im}(\widetilde{\varphi}_a \circ \cdots \circ \widetilde{\varphi}_1) \subset {\bigcap}^{r}_{R}\widetilde{H}^1_{\rm f}(G_{\QQ, \Sigma}, T, \Delta)\,,
\end{align}
since we have $\widetilde{\varphi} \circ \widetilde{\varphi}_a \circ \cdots \circ \widetilde{\varphi}_1 = 0$ by construction. 
When $r=0$, we also have $\mathrm{im}(\widetilde{\varphi}_a \circ \cdots \circ \widetilde{\varphi}_1) \subset R$ and 
${\bigcap}^{0}_{R}\widetilde{H}^1_{\rm f}(G_{\QQ, \Sigma}, T, \Delta) = R$. In other words, \eqref{eq_image_bi-dual}  holds true also if $r=0$. 

\begin{defn}
Suppose that $r \geq 0$. We define the cyclic $R$-module $\delta(T, \Delta)$ to be the image of the homomorphism 
\begin{align}\label{eq:mor_det_to_bdual}
\det(R^{a+r}) = {\bigwedge}^{a+r}_{R}R^{a+r} \longrightarrow {\bigcap}^{r}_{R}\widetilde{H}^1_{\rm f}(G_{\QQ, \Sigma}, T, \Delta)    
\end{align}
induced by the exact sequence \eqref{exact:stark}. 
\end{defn}
We call  $\delta(T, \Delta)$ the module of leading terms associated to the data $(T,\Delta)$. 

\begin{remark}
The construction of the morphism \eqref{eq:mor_det_to_bdual} is inspired by the work of Burns and Sano \cite[Proposition~2.21]{Burns_Sano_2021}. 
Kataoka also considers (essentially) the same morphism in \cite[Definition~3.2]{Kataoka2022}. 
\end{remark}

\begin{lemma}
The module $\delta(T, \Delta)$ is independent of the choice of the representative 
$$[ \, R^{a+r} \stackrel{\varphi}{\longrightarrow} R^{a} \, ]$$ 
of the Selmer complex $\widetilde{R\Gamma}_{\rm f}(G_{\QQ, \Sigma}, T, \Delta)$. 
\end{lemma}

\begin{proof}
Let us put $m := \dim\,\widetilde{H}_{\rm f}^2(G_{\QQ, \Sigma}, \overline{T}, \Delta)$ and consider a surjection $R^m \twoheadrightarrow \widetilde{H}_{\rm f}^2(G_{\QQ, \Sigma}, T, \Delta)$. 
We then have an induced surjection $f \colon R^{a} \twoheadrightarrow R^m$ such that the diagram 
\[
\xymatrix{
R^{a} \ar[r]^-{f} \ar[rd] & R^m \ar[d] 
\\
& \widetilde{H}_{\rm f}^2(G_{\QQ, \Sigma}, T, \Delta)
}
\]
commutes. 
Since $R^m$ is free, there is a splitting  $R^{a} = R^{m}  \oplus R^{a-m}$ such that $f= \mathrm{pr}_1$. 
Let us set 
\[
K := \ker\left(R^m \longrightarrow \widetilde{H}_{\rm f}^2(G_{\QQ, \Sigma}, T, \Delta)\right). 
\]
Then, since $f= \mathrm{pr}_1$, we have 
\[
K \oplus R^{a-m} = \ker(R^{a}  \longrightarrow \widetilde{H}_{\rm f}^2(G_{\QQ, \Sigma}, T, \Delta)). 
\]
Moreover, as $R^{a+r} \longrightarrow K \oplus R^{a-m}$ is surjective, there is a splitting $R^{a+r} = R^{m+r} \oplus R^{a-m}$ such that $$\mathrm{im}(R^{m+r} \longrightarrow  K \oplus R^{a-m}) = K.$$ 
Then the complex $[ \, R^{m+r} \longrightarrow R^{m} \, ]$ is a representative of $\widetilde{R\Gamma}_{\rm f}(G_{\QQ, \Sigma}, T, \Delta)$ and 
\[
[ \, R^{a+r} \longrightarrow R^{a} \, ] = [ \, R^{m+r} \longrightarrow R^{m} \, ] \oplus 
[ \, R^{a-m} \stackrel{\sim}{\longrightarrow} R^{a-m} \, ]. 
\]
We therefore obtain a commutative diagram 
\[
\xymatrix{
\det(R^{a+r}) \ar[r] \ar[d]_-{\cong} &  \ar[d]^-{=} {\bigcap}^{r}_{R}\widetilde{H}^1_{\rm f}(G_{\QQ, \Sigma}, T, \Delta)
\\
\det(R^{m+r}) \ar[r] & {\bigcap}^{r}_{R}\widetilde{H}^1_{\rm f}(G_{\QQ, \Sigma}, T, \Delta),  
}
\]
which completes the proof. 
\end{proof}

\begin{proposition}\label{proposition_tortion_non-zero}
The $R$-module $\widetilde{H}^2_{\rm f}(G_{\QQ, \Sigma}, T, \Delta)$  is torsion if and only if $\delta(T, \Delta)$ is generated by an $R$-regular element of ${\bigcap}^{r}_{R}\widetilde{H}^1_{\rm f}(G_{\QQ, \Sigma}, T, \Delta)$. 
\end{proposition}
\begin{proof}
Let us set $Q := \mathrm{Frac}(R)$. If $\widetilde{H}^2_{\rm f}(G_{\QQ, \Sigma}, T, \Delta)$  is torsion, then the homomorphism $Q^{a+r} \to Q^a$ is surjective. We therefore have 
\[
\delta(T, \Delta) \otimes_R Q   \cong Q 
\]
by construction, which means that $\delta(T, \Delta)$ is generated by an $R$-regular element of ${\bigcap}^{r}_{R}\widetilde{H}^1_{\rm f}(G_{\QQ, \Sigma}, T, \Delta)$. 

Let us prove the converse. Since $\delta(T, \Delta)$ is generated by an $R$-regular element, the homomorphism 
\[
\widetilde{\varphi}_a \circ \cdots \circ \widetilde{\varphi}_1  \colon {\bigwedge}^{r+a}_Q Q^{r+a} \longrightarrow {\bigwedge}^r_{Q} Q^{r+a}
\]
is injective. 
Then $\widetilde{\varphi}_1 \colon {\bigwedge}^{r+a}_Q Q^{r+a} \to {\bigwedge}^{r+a-1}_Q Q^{r+a}$ is injective, and hence $\varphi_1 \colon Q^{r+a} \to Q$ is surjective. 
Then there is a splitting $Q^{r+a} = Q \oplus \ker(\varphi_1)$ such that 
\[
\widetilde{\varphi}_a \circ \cdots \circ \widetilde{\varphi}_2  \colon {\bigwedge}^{r+a-1}_Q \ker(\varphi_1) \lra {\bigwedge}^r_{Q} \ker(\varphi_1)
\]
is injective. 
Since $\ker(\varphi_1) \cong Q^{r+a-1}$, we see that $\varphi_2 \colon \ker(\varphi_1) \to Q$ is surjective. 
Repeating the same argument, we conclude that $\varphi \colon Q^{r+a} \to Q^r$ is surjective, that is, 
$\widetilde{H}^2_{\rm f}(G_{\QQ, \Sigma}, T, \Delta) \otimes_R Q = 0$.  
\end{proof}

\begin{theorem}
Suppose that $R$ is normal. 
If $\delta(T, \Delta) \neq 0$, then the $R$-module $\widetilde{H}^2_{\rm f}(G_{\QQ, \Sigma}, T, \Delta)$  is torsion and 
\[
\mathrm{char}_R\left( \left. 
{\bigcap}^r_R \widetilde{H}^1_{\rm f}(G_{\QQ, \Sigma}, T, \Delta)   \middle/  \delta(T, \Delta)  \right. \right) 
= 
\mathrm{char}_R\left( \widetilde{H}^2_{\rm f}(G_{\QQ, \Sigma}, T, \Delta) \right). 
\]
\end{theorem}
See \cite[Definition 2.1]{BBL2014} for the definition of the characteristic ideal in the case of a normal ring.
\begin{proof}
The first assertion follows from Proposition \ref{proposition_tortion_non-zero}. 
Let us show that
\[
\mathrm{char}_R\left( \left. 
{\bigcap}^r_R \widetilde{H}^1_{\rm f}(G_{\QQ, \Sigma}, T, \Delta)   \middle/  \delta(T, \Delta)  \right. \right) 
= 
\mathrm{char}_R\left( \widetilde{H}^2_{\rm f}(G_{\QQ, \Sigma}, T, \Delta) \right).
\]
To that end, let $\fp$ be a height-1 prime of $R$. 
Then $\delta(T, \Delta)R_\fp$ coincides with the image of the map $\det(R_\fp^{r+a}) \to {\bigwedge}^r_{R_\fp}R_\fp^{r+a}$ 
induced by the homomoprhism $\varphi \colon R_\fp^{r+a} \longrightarrow R_\fp^{r}$. 
Since $R_\fp$ is a discrete valuation ring and $\widetilde{H}^2_{\rm f}(G_{\QQ, \Sigma}, T, \Delta)$  is torsion, 
we may assume that $\varphi_i = \pi^{a_i} \cdot \mathrm{pr}_i$ for some integer $a_i \geq 0$. 
Here $\pi$ denotes an uniformizer of $R_\fp$ and $ \mathrm{pr}_i$ projection onto the $i$th factor. 
If we denote by $\{e_1, \ldots, e_{r+a}\}$ the standard basis of $R_\fp^{r+a}$,  we then have 
$\widetilde{H}^1_{\rm f}(G_{\QQ, \Sigma}, T, \Delta) = R_\fp e_{a+1} + \cdots + R_\fp e_{a+r}$ and 
\begin{align*}
    \delta(T, \Delta)R_\fp 
    = \mathrm{im}\left( \det(R_\fp^{r+a}) \longrightarrow {\bigwedge}^r_{R_\fp}R_\fp^{r+a} \right) &= \widetilde{\varphi}_a \circ \cdots \circ \widetilde{\varphi}_1(e_1 \wedge \cdots \wedge e_{r+a})R_\fp
     \\
     &= \pi^{\sum a_i} {\bigwedge}^r_{R_\fp} (\widetilde{H}^1_{\rm f}(G_{\QQ, \Sigma}, T, \Delta) \otimes_R R_\fp). 
\end{align*}
We therefore conclude that 
\begin{align*}
        \mathrm{char}_R\left( \left. {\bigcap}^r_R \widetilde{H}^1_{\rm f}(G_{\QQ, \Sigma}, T, \Delta)   \middle/  \delta(T, \Delta)  \right. \right) R_\fp
    &= \pi^{\sum a_i}R_\fp = \mathrm{char}_R\left( \widetilde{H}^2_{\rm f}(G_{\QQ, \Sigma}, T, \Delta) \right)R_\fp, 
\end{align*}
which completes the proof. 
\end{proof}

\begin{lemma}[{\cite{sakamoto_stark}, Lemma 4.8}]\label{lemma_fitt_stark}
Let $S$ be a $0$-dimensional Gorenstein local ring. 
Let $s$ and $t$ be non-negative integers and
\[
0 \longrightarrow M \longrightarrow S^{s+t} \longrightarrow S^{s} \longrightarrow N \longrightarrow 0
\]
an exact sequence of $S$-modules. 
Let 
$$\delta \in {\bigcap}_{S}^{t}M = \Hom_{S}({\bigwedge}_{S}^{t}M^*, S)$$ denote the image of $1 \in S = \det(S^{s+t})$ under the homomorphism $\det(S^{s+t}) \longrightarrow {\bigcap}_{S}^{t}M$. 
We then have  
\[
\mathrm{im}(\delta) = \mathrm{Fitt}_{S}^0(N). 
\]
\end{lemma}

\begin{lemma}\label{lemma_dual_image}
Let $S$ be a $0$-dimensional Gorenstein local ring, and $M$ a finitely generated $S$-module. Let $\delta \in M^*$. 
Then we have a natural isomorphism 
\[
 (S \delta)^* \stackrel{\sim}{\longrightarrow} \mathrm{im}(\delta); \,\quad  \varphi \mapsto \varphi(\delta). 
\]
\end{lemma}
\begin{proof}
By Matlis duality, the canonical homomorphism $M \longrightarrow M^{**}$ is an isomorphism. 
We therefore have 
\[
\mathrm{im}(\delta) = \{\varphi(\delta) \mid \varphi \in M^{**} = \Hom_S(M^*, S)\}. 
\]
Since the functor $(-)^*$ is exact,  $\Hom_S(M^*, S) \longrightarrow \Hom_S(S\delta, S)$ is surjective. 
Hence, the homomorphism 
\[
 (S \delta)^* \longrightarrow \mathrm{im}(\delta); \,\qquad \varphi \mapsto \varphi(\delta)
\]
is surjective.  The injectivity of this map is clear. 
\end{proof}

Note that one can also define 
\[
\delta(T/IT, \Delta) \subset {\bigcap}^r_{R/I}\widetilde{H}^1_{\rm f}(G_{\QQ, \Sigma}, T/IT, \Delta) 
\]
for any ideal $I$ of $R$ such that $R/I$ is Gorenstein, since $\widetilde{R\Gamma}_{\rm f}(G_{\QQ, \Sigma}, T/IT, \Delta) \in D^{[1,2]}_{\rm parf}(_{R/I}\mathrm{Mod})$. 
We deduce the following from Lemma \ref{lemma_fitt_stark} and Lemma \ref{lemma_dual_image}:

\begin{proposition}\label{proposition_delta=fitt}
If $R/I$ is a $0$-dimensional Gorenstein local ring, then 
any generator $\delta$ of $\delta(T/IT, \Delta)$ induces an isomorphism  
\[
\delta(T/IT, \Delta)^* \stackrel{\sim}{\longrightarrow} \mathrm{Fitt}_{R/I}^{0}\,\widetilde{H}^2_{\rm f}(G_{\QQ, \Sigma}, T/IT, \Delta)\,; \,\qquad  
\varphi \mapsto \varphi(\delta). 
\]
\end{proposition}

Since we have
\[
\widetilde{R\Gamma}_{\rm f}(G_{\QQ, \Sigma}, T, \Delta) \otimes^{\mathbb{L}}_{R} R/I \stackrel{\sim}{\longrightarrow} 
\widetilde{R\Gamma}_{\rm f}(G_{\QQ, \Sigma}, T/IT, \Delta)
\]
by Proposition \ref{prop_[1,2]_parf}, it follows from Proposition \ref{proposition_exterior_wedge} that we have a commutative diagram 
\[
\xymatrix{
\det(R^{r+a}) \ar[r] \ar[d] & {\bigcap}^r_{R}\widetilde{H}^1_{\rm f}(G_{\QQ, \Sigma}, T, \Delta) \ar[d] 
\\
\det((R/I)^{r+a}) \ar[r]  & {\bigcap}^r_{R/I}\widetilde{H}^1_{\rm f}(G_{\QQ, \Sigma}, T/IT, \Delta). 
}
\]
As 
\[
\varprojlim_{I}{\bigcap}^{r}_{R/I}\widetilde{H}^1_{\rm f}(G_{\QQ, \Sigma}, T/IT, \Delta) 
= {\bigcap}^{r}_{R}\widetilde{H}^1_{\rm f}(G_{\QQ, \Sigma}, T, \Delta),  
\]
we have the following: 

\begin{proposition}
\label{prop_51_2022_09_12_1507}
We have a natural surjection
$\delta(T, \Delta) \twoheadrightarrow \delta(T/IT, \Delta)$, and $\delta(T, \Delta) = \varprojlim_{I} \delta(T/IT, \Delta)$. 
\end{proposition}

\subsubsection{The case $r=0$}
\label{subsubsec_EP_char_is_zero}

Note that ${\bigcap}^0_S M = S$ for any (commutative) ring $S$ and any finitely generated $S$-module $M$. 
As a result, we have $\delta(T, \Delta) \subset R$ when $r=0$.
\begin{theorem}\label{theorem_delta_rank0}
Suppose that $r = 0$. 

    \item[1)] $\delta(T, \Delta) =  \mathrm{Fitt}_{R}^0\,\widetilde{H}^2_{\rm f}(G_{\QQ, \Sigma}, T, \Delta)$. 
    \item[2)] If $\widetilde{H}^1_{\rm f}(G_{\QQ, \Sigma}, T, \Delta) = \{0\}$, then $\widetilde{H}^2_{\rm f}(G_{\QQ, \Sigma}, T, \Delta)$ is torsion and  $\delta(T, \Delta) = \det(\widetilde{R\Gamma}_{\rm f}(G_{\QQ, \Sigma}, T, \Delta))$\,. In other words, $\delta(T, \Delta)$ coincides with the module of algebraic $p$-adic $L$-functions for $(T, \Delta)$. 
  \item[3)] If $\delta(T, \Delta)$ is generated by a regular element of $R$, then $\widetilde{H}^1_{\rm f}(G_{\QQ, \Sigma}, T, \Delta) = \{0\}$.


\end{theorem}
\begin{proof}
\item[1)] 
Let us take an ideal $I \subset R$  such that $R/I$ is a 0-dimensional Gorenstein ring.  
Note that for any ideal $J$ of $R/I$, we have $\mathrm{Ann}_{R/I}(\mathrm{Ann}_{R/I}(J)) = J$. 
Since $\delta(T/IT, \Delta)$ is cyclic, any generator $\delta \in \delta(T/IT, \Delta)$ induces an isomorphism 
\[
\delta(T/IT, \Delta)^* \stackrel{\sim}{\longrightarrow} \delta(T/IT, \Delta) \subset R/I; \,\qquad \varphi \mapsto \varphi(\delta), 
\]
and Proposition \ref{proposition_delta=fitt} tells us that 
$\delta(T/IT, \Delta) =  \mathrm{Fitt}_{R/I}^0(\widetilde{H}^2_{\rm f}(G_{\QQ, \Sigma}, T/IT, \Delta))$. We therefore conclude 
\[
\delta(T, \Delta) =  \varprojlim_{I} \delta(T/IT, \Delta) = \varprojlim_{I} \mathrm{Fitt}_{R/I}^0(\widetilde{H}^2_{\rm f}(G_{\QQ, \Sigma}, T/IT, \Delta)) = \mathrm{Fitt}_{R}^0\,\widetilde{H}^2_{\rm f}(G_{\QQ, \Sigma}, T, \Delta)\,. 
\]
\item[2)]
Since $r=0$ and $\widetilde{H}^1_{\rm f}(G_{\QQ, \Sigma}, T, \Delta)$ vanishes, the $R$-module $\widetilde{H}^2_{\rm f}(G_{\QQ, \Sigma}, T, \Delta)$ is torsion and its projective dimension is $1$ by Proposition \ref{prop_[1,2]_parf}. 
Hence,
\[
\delta(T, \Delta) =  \mathrm{Fitt}_{R}^0\,\widetilde{H}^2_{\rm f}(G_{\QQ, \Sigma}, T, \Delta) = \det(\widetilde{R\Gamma}_{\rm f}(G_{\QQ, \Sigma}, T, \Delta)). 
\]

\item[3)] This portion is an immediate consequence of Proposition \ref{proposition_tortion_non-zero}. 
\end{proof}

\begin{remark}\label{rem:non-trivial_r=0_case}
    Suppose that $r = 0$. When $\delta(T, \Delta) = 0$,  by Proposition \ref{proposition_tortion_non-zero}, the $R$-module $\widetilde{H}^2_{\rm f}(G_{\QQ, \Sigma}, T, \Delta)$ is not torsion. 
    Since $r=0$ and $\widetilde{R\Gamma}_{\rm f}(G_{\QQ, \Sigma}, T, \Delta) \in D^{[1,2]}_{\rm parf}(_{R}\mathrm{Mod})$, it follows that  $\widetilde{H}^1_{\rm f}(G_{\QQ, \Sigma}, T, \Delta)$ is a non-trivial  torsion-free $R$-module. 
\end{remark}

\subsubsection{The case $r=1$}
\label{subsubsec_EP_char_is_1}
In this subsection, we study closely the case $r=1$. 
We note for any 0-dimensional Gorenstein local ring $S$ and any finitely generated $S$-module $M$ that the canonical homomorphism $M \to {\bigcap}^1_S M = M^{**}$ is an isomorphism by Matlis duality. In this particular case, we therefore have 
\[
\delta(T, \Delta) = \varprojlim_{I} \delta(T/IT, \Delta) \subset \varprojlim_{I}\widetilde{H}^1_{\rm f}(G_{\QQ, \Sigma}, T/IT, \Delta) = \widetilde{H}^1_{\rm f}(G_{\QQ, \Sigma}, T, \Delta)\,. 
\]

\begin{theorem}\label{theorem_delta_rank1} 
Suppose that $\delta(T, \Delta)$ is generated by an $R$-regular element. We then have an exact sequence of $R$-modules 
    \[
    0 \longrightarrow \mathrm{Fitt}^0_R(\widetilde{H}^2_{\rm f}(G_{\QQ, \Sigma}, T, \Delta))  
    \longrightarrow \mathrm{Ind}(\delta(T, \Delta))
    \longrightarrow \mathrm{Ext}_R^2(\widetilde{H}^2_{\rm f}(G_{\QQ, \Sigma}, T, \Delta), R)
    \longrightarrow 0. 
    \]
    Here, we have put 
    \[
    \mathrm{Ind}(\delta(T, \Delta)) := \{\varphi(\delta) \mid \delta \in \delta(T, \Delta),\, \varphi \in \Hom_R(\widetilde{H}^1_{\rm f}(G_{\QQ, \Sigma}, T, \Delta), R)\}. 
    \]
\end{theorem}

\begin{proof}
Let us take an ideal $I \subset R$ such that $R/I$ is a $0$-dimensional Gorenstein local ring. Since we have
\[
\widetilde{R\Gamma}_{\rm f}(G_{\QQ, \Sigma}, T, \Delta) \otimes^{\mathbb{L}}_R R/I \cong \widetilde{R\Gamma}_{\rm f}(G_{\QQ, \Sigma}, T/IT, \Delta) \,\,\, \textrm{ and } \,\,\, 
\widetilde{R\Gamma}_{\rm f}(G_{\QQ, \Sigma}, T, \Delta) \in D^{[1,2]}_{\rm parf}(_R{\rm Mod}) 
\]
by Proposition \ref{prop_[1,2]_parf}, we have an exact sequence 
\begin{align*}
    0 \longrightarrow \mathrm{Tor}^R_2(\widetilde{H}^2_{\rm f}(G_{\QQ, \Sigma}, T, \Delta), R/I) 
    \longrightarrow & \widetilde{H}^1_{\rm f}(G_{\QQ, \Sigma}, T, \Delta) \otimes_R R/I 
    \\
    &\longrightarrow \widetilde{H}^1_{\rm f}(G_{\QQ, \Sigma}, T/IT, \Delta) 
    \longrightarrow 
    \mathrm{Tor}^R_1(\widetilde{H}^2_{\rm f}(G_{\QQ, \Sigma}, T, \Delta), R/I) \longrightarrow 0
\end{align*} 
of $R/I$-modules. Since $R/I$ is injective, we obtain the following
an exact sequence of $R/I$-modules by passing to $R/I$-duals in the exact sequence above:
\begin{align*}
    0 \longrightarrow \mathrm{Ext}_R^1(\widetilde{H}^2_{\rm f}(G_{\QQ, \Sigma}, T, \Delta), R/I) 
    \longrightarrow & \Hom_{R/I}(\widetilde{H}^1_{\rm f}(G_{\QQ, \Sigma}, T/IT, \Delta), R/I )
    \\
    &\longrightarrow \Hom_{R}(\widetilde{H}^1_{\rm f}(G_{\QQ, \Sigma}, T, \Delta), R/I) 
    \longrightarrow 
    \mathrm{Ext}^2_R(\widetilde{H}^2_{\rm f}(G_{\QQ, \Sigma}, T, \Delta), R/I) \longrightarrow 0. 
\end{align*} 
Let us write $F_I$ for the image of the homomorphism 
\[
\Hom_{R/I}(\widetilde{H}^1_{\rm f}(G_{\QQ, \Sigma}, T/IT, \Delta), R/I ) \longrightarrow \Hom_{R}(\widetilde{H}^1_{\rm f}(G_{\QQ, \Sigma}, T, \Delta), R/I)\,.
\]
Then by Proposition \ref{proposition_delta=fitt},
\[
\mathrm{Fitt}_{R/I}^0(\widetilde{H}^2_{\rm f}(G_{\QQ, \Sigma}, T/IT, \Delta)) 
= \{\varphi(\delta) \mid \delta \in \delta(T, \Delta),\, \varphi \in F_I\}. 
\]
Since $\delta(T, \Delta)$ is generated by an $R$-regular element, we have 
\begin{align*}
  \Hom_R(\widetilde{H}^1_{\rm f}(G_{\QQ, \Sigma}, T, \Delta), R) &\stackrel{\sim}{\longrightarrow} \mathrm{Ind}(\delta(T, \Delta)); \, \qquad \delta \mapsto \varphi(\delta), 
  \\
   \varprojlim_{I}F_I  &\stackrel{\sim}{\longrightarrow} \mathrm{Fitt}_{R}^0\,\widetilde{H}^2_{\rm f}(G_{\QQ, \Sigma}, T, \Delta)\, ; \, \qquad \delta \mapsto \varphi(\delta), 
\end{align*}
where $\delta \in \delta(T, \Delta)$ is a generator. 
The desired exact sequence is thus constructed and our proof is complete. 
\end{proof}

\begin{corollary}\label{corollary_regular_delta_r=1}
Suppose that $R$ is regular and $\delta(T, \Delta) \neq \{0\}$. 
We then have 
\[
\det(\widetilde{H}^1_{\rm f}(G_{\QQ, \Sigma}, T, \Delta)/\delta(T, \Delta)) = \det(\widetilde{H}^2_{\rm f}(G_{\QQ, \Sigma}, T, \Delta)). 
\]
\end{corollary}

\begin{remark} \label{rmk:comparisonkolyvagin}
Let us consider the case that $R$ is regular and the $R$-module $\widetilde{H}^2_{\rm f}(G_{\QQ, \Sigma}, T, \Delta)$  is torsion.  We also assume the standard assumptions of the theory of Euler--Kolyvagin systems, e.g., that the Galois representation $G_{\QQ}\to \GL(T)$ has a large image, $H^0(G_p, F^{\pm}\overline{T})$ and $H^2(G_p, F^{\pm}\overline{T})$ vanish, and the complex $R\Gamma(G_v, T)$ is acyclic for each $v \in \Sigma^p$. 
One can then show that  the module of Kolyvagin systems $\mathrm{KS}_1(T, \Delta)$ associated with the pair $(T,\Delta)$  is free of rank $1$; cf. \cite[Theorem 5.2.10(ii)]{mr02}, \cite[Theorem A]{kbbdeform}, and \cite[Theorem 5.2(i)]{bss_kolyvagin-stark}. Moreover, we have a homomorphism 
\[
\mathrm{KS}_1(T, \Delta) \longrightarrow \widetilde{H}^1_{\rm f}(G_{\QQ, \Sigma}, T, \Delta); \,\qquad \kappa \mapsto \kappa_1= \textrm{the leading class of the Kolyvagin system $\kappa$} 
\]
since $\widetilde{H}^1_{\rm f}(G_{\QQ, \Sigma}, T, \Delta) = H^1_{\cF_{\Delta}}(G_{\QQ, \Sigma}, T) := \ker\left(H^1_{\cF^{\rm can}}(G_{\QQ, \Sigma}, T) \to H^1(G_{\mathbb{Q}_p}, F^-T)\right)$ by \cite[Lemma 9.6.3]{nekovar06} or \cite[Proposition 2.5 and Remark 2.6]{kbbCMH}. 
Moreover, in this case, $\widetilde{H}^2_{\rm f}(G_{\QQ, \Sigma}, T, \Delta)$ identifies with the dual Selmer module associated with the Selmer structure $\cF_{\Delta}$. 
Hence, the assumption that $\widetilde{H}^2_{\rm f}(G_{\QQ, \Sigma}, T, \Delta)$ is torsion shows that this homomorphism is injective (in fact, the converse to this statement also holds). 
Let us denote by $\kappa(T,\Delta)$ for its image, that is, the module of the leading classes of Kolyvagin systems. Then the theory of Kolyvagin systems tells us that 
\[
\det\,(\widetilde{H}^1_{\rm f}(G_{\QQ, \Sigma}, T, \Delta)/\kappa(T,\Delta)) = \det\,(\widetilde{H}^2_{\rm f}(G_{\QQ, \Sigma}, T, \Delta)), 
\]
and Corollary \ref{corollary_regular_delta_r=1} implies that $
\kappa(T,\Delta) = \delta(T, \Delta)$. In other words, the module $\delta(T, \Delta)$ is generated by the leading classes of Kolyvagin systems. 
It is worth mentioning that, in order to construct the module $\delta(T, \Delta)$, we only need the assumption concerning Tamagawa factors, which is considerably milder than the standard assumptions of the theory of  Kolyvagin systems.

The rigidity of Kolyvagin systems $($cf. \cite{mr02,kbb,kbbdeform}$)$ combined with Poitou--Tate global duality allows one to recast the Iwasawa Main Conjecture as the statement that the ground-level Euler system class $($which should map under a suitable Perrin-Riou logarithm to the relevant $p$-adic $L$-function$)$ generates the module of leading terms. 

Such a comparison $($of the ground-level Euler system class and the module of leading terms$)$ can be deduced without relying on the structure of the module of Kolyvagin systems, directly comparing Corollary~\ref{corollary_regular_delta_r=1} to the statement of Iwasawa Main Conjecture ``without $p$-adic $L$-functions'' $($in the sense of Kato; cf. \cite[Conjecture~12.10]{kato04}$)$ that one formulates in terms of the ground-level Euler system class. This form of the Main Conjecture is equivalent to its more familiar form in terms of $p$-adic $L$-functions $($granted a reciprocity law linking the said class to the $p$-adic $L$-function$)$ by the Poitou--Tate global duality; cf. \cite[\S17.13]{kato04}.  We refer the reader to Proposition~\ref{prop_eqn_2022_09_12_1953_bis}, where we go through a special case of this discussion.
\end{remark}

\subsubsection{}
Let us consider a pair of $R[G_p]$-submodules $F_1^+T \subset F_2^+T$ of  $T$ such that the quotients $T/F^+_1T$ and $T/F^+_2T$ are free as $R$-modules. 
We then have the Greenberg local conditions $\Delta_1 := \Delta_{F^+_1}$ and $\Delta_2 := \Delta_{F^+_2}$. Let us assume that
\begin{align*}
    \mathrm{rank}_R(T^{c=-1}) - \mathrm{rank}_R(T/F^+_1 T) &= 0, 
    \\
    \mathrm{rank}_R(T^{c=-1}) - \mathrm{rank}_R(T/F^+_2 T) &= 1. 
\end{align*}
Suppose in addition that 
\begin{itemize}
    \item $H^0(G_p, F_2^+\overline{T}/F_1^+\overline{T}) = 0=  H^2(G_p, F_2^+\overline{T}/F_1^+\overline{T})$. 
\end{itemize}
Then the $R$-module $H^1(G_p, F_2^+T/F_1^+T)$ is free of rank one by the local Euler characteristic formula. Let us fix a trivialisation 
$$\mathrm{LOG}_{\Delta_2/\Delta_1}\,:\quad H^1(G_p, F_2^+T/F_1^+T) \stackrel{\sim}{\lra} R$$
and abusively denote the compositum of the arrows
$$\widetilde{H}^1_{\rm f}(G_{\QQ, \Sigma}, T, \Delta_2)\xrightarrow{\res_p}H^1(G_p, F_2^+T)\lra H^1(G_p, F_2^+T/F_1^+T)\xrightarrow{\mathrm{LOG}_{\Delta_2/\Delta_1}}R$$
also by $\mathrm{LOG}_{\Delta_2/\Delta_1}$. We then have an exact sequence 
\begin{align*}
    0 \longrightarrow \widetilde{H}^1_{\rm f}(G_{\QQ, \Sigma}, T, \Delta_1) 
    \longrightarrow \widetilde{H}^1_{\rm f}(G_{\QQ, \Sigma}, T, \Delta_2)  
    \xrightarrow{\mathrm{LOG}_{\Delta_2/\Delta_1}} R 
    \longrightarrow \widetilde{H}^2_{\rm f}(G_{\QQ, \Sigma}, T, \Delta_1)
    \longrightarrow \widetilde{H}^2_{\rm f}(G_{\QQ, \Sigma}, T, \Delta_2) 
    \longrightarrow 0. 
\end{align*}

\begin{lemma}\label{lemma_log_H^1_vanishes}
If $\mathrm{LOG}_{\Delta_2/\Delta_1}(\delta(T, \Delta_2))$ is generated by a regular element of $R$,  
then $\widetilde{H}^1_{\rm f}(G_{\QQ, \Sigma}, T, \Delta_1) = 0$. 
\end{lemma}
\begin{proof}
By Theorem \ref{theorem_delta_rank1}.1, the $R$-module $\widetilde{H}^2_{\rm f}(G_{\QQ, \Sigma}, T, \Delta_2)$ is torsion. 
Since the quotient $R/\mathrm{LOG}_{\Delta_2/\Delta_1}(\delta(T, \Delta_2))$ is torsion by assumption, 
$\widetilde{H}^2_{\rm f}(G_{\QQ, \Sigma}, T, \Delta_1)$ is also torsion. 
Since we have $\chi(\widetilde{R\Gamma}_{\rm f}(G_{\QQ, \Sigma}, T, \Delta_1)) = 0$, it follows that $\widetilde{H}^1_{\rm f}(G_{\QQ, \Sigma}, T, \Delta_1) = 0$, as required. 
\end{proof}


\begin{theorem}
\label{thm_PR_formal}
Suppose that $R$ is regular. 
Then we have 
    \[
\mathrm{LOG}_{\Delta_2/\Delta_1}(\delta(T, \Delta_2)) = \delta(T, \Delta_1). 
\]
In particular, if $\widetilde{H}^2_{\rm f}(G_{\QQ, \Sigma}, T, \Delta_1)$ is torsion, then 
$\mathrm{LOG}_{\Delta_2/\Delta_1}(\delta(T, \Delta_2)) = \det(\widetilde{R\Gamma}_{\rm f}(G_{\QQ, \Sigma}, T, \Delta_1))$. 
\end{theorem}

\begin{proof}
If $\delta(T, \Delta_2) = 0$, then $\widetilde{H}^2_{\rm f}(G_{\QQ, \Sigma}, T, \Delta_2)$ is non-torsion by Theorem \ref{theorem_delta_rank1}. Therefore, $\widetilde{H}^2_{\rm f}(G_{\QQ, \Sigma}, T, \Delta_1)$ is also non-torsion. Theorem \ref{theorem_delta_rank0}.1 shows that $\delta(T, \Delta_1) = 0$, and $\mathrm{LOG}_{\Delta_2/\Delta_1}(\delta(T, \Delta_2)) = 0 = \delta(T, \Delta_1)$. 

Suppose now that $\delta(T, \Delta_2) \neq 0$. Then the $R$-module $\widetilde{H}^2_{\rm f}(G_{\QQ, \Sigma}, T, \Delta_2)$ is  torsion. 
Since $\chi(\widetilde{R\Gamma}_{\rm f}(G_{\QQ, \Sigma}, T, \Delta_2)) = -1$, we have 
$\mathrm{rank}_R(\widetilde{H}^1_{\rm f}(G_{\QQ, \Sigma}, T, \Delta_2)) = 1$. 
Therefore, if $\mathrm{LOG}_{\Delta_2/\Delta_1}(\delta(T, \Delta_2)) = 0$, then $\widetilde{H}^2_{\rm f}(G_{\QQ, \Sigma}, T, \Delta_1)$ is non-torsion, and $\mathrm{LOG}_{\Delta_2/\Delta_1}(\delta(T, \Delta_2)) = 0 = \delta(T, \Delta_1)$. 

If $\mathrm{LOG}_{\Delta_2/\Delta_1}(\delta(T, \Delta_2)) \neq 0$, then Lemma \ref{lemma_log_H^1_vanishes} shows that 
$\widetilde{H}^2_{\rm f}(G_{\QQ, \Sigma}, T, \Delta_1) = 0$ and we have an exact sequence  
\begin{align*}
    0 \to \widetilde{H}^1_{\rm f}(G_{\QQ, \Sigma}, T, \Delta_2)/&\delta(T, \Delta_2)  
    \longrightarrow R/ \mathrm{LOG}_{\Delta_2/\Delta_1}(\delta(T, \Delta_2)) \longrightarrow \widetilde{H}^2_{\rm f}(G_{\QQ, \Sigma}, T, \Delta_1)
    \longrightarrow \widetilde{H}^2_{\rm f}(G_{\QQ, \Sigma}, T, \Delta_2) 
    \to 0
\end{align*}
of torsion $R$-modules. Theorem \ref{theorem_delta_rank0} combined with Corollary \ref{corollary_regular_delta_r=1} show that 
$$
 \mathrm{LOG}_{\Delta_2/\Delta_1}(\delta(T, \Delta_2)) = \det(R/ \mathrm{LOG}_{\Delta_2/\Delta_1}(\delta(T, \Delta_2)))
= \det(\widetilde{H}^2_{\rm f}(G_{\QQ, \Sigma}, T, \Delta_1))
= \delta(T, \Delta_1). 
$$
\end{proof}

\section{Factorisation of algebraic \texorpdfstring{$p$}{}-adic \texorpdfstring{$L$}{}-functions}
\label{subsec_KS}

Our objective in \S\ref{subsec_KS} is to establish one of our main factorisation results (Theorem~\ref{thm_main_8_4_4_factorisation}). Our approach draws from the general ETNC philosophy and we expect that it is only an instance of a much more general factorisation principle one can establish in other settings. The central role of Nekov\'a\v{r}'s general formalism of Selmer complexes will be clear to the reader\footnote{As this was the case also in our earlier works \cite{kbbMTT, benoisbuyukboduk, BS19, BS2020} that concern arithmetic properties of $p$-adic $L$-functions.}, and we would like to underline its importance as a most natural medium to study the arithmetic properties of $p$-adic $L$-functions.

\subsection{Modules of leading terms (bis)}
\label{subsec_leading_terms_bis}
The factorisation formulae we shall prove will amount to a comparison of the following modules of leading terms.

\subsubsection{Hypotheses}
\label{subsubsec_hypo_section_6}
Throughout \S\ref{subsec_KS}, we shall work under the following list of hypotheses.

Assumption~\ref{assumption_tamagawa} is enforced both with $\f$ and $\g$. Note that this condition trivially holds if the tame conductors of the Hida families $\f$ and $\g$ are square-free. We assume in addition that \ref{item_Irr} holds for the family $\f$, \ref{item_Dist} for both $\f$ and $\g$, as well as the following strengthening of \ref{item_Irr} for $\g$:
\begin{itemize}
       \item[\mylabel{item_Irr_plus}{\bf (Irr$+$)}]  The residual representation $\bar{\rho}_\hg$ is absolutely irreducible when restricted $G_{\QQ(\zeta_p)}$.
    \end{itemize}
We further assume that the $p$-part of the Tamagawa factor for a single member of the Hida family {$\hf$} and a single member of the family {$\hg$} equals $1$.
    
We remark that these hypotheses guarantee the validity of \ref{item_Irr_general} and \ref{item_Tam} for $T_3^\dagger,T_2^\dagger, M_2^\dagger$ and $T_\f^\dagger$. In particular, our main conclusions in \S\ref{sec_selmer_complexes} apply and show that the Selmer complexes associated with these Galois representations are perfect. The constructions of \S\ref{sec_Koly_Sys_Dec_4_18_05_2021} apply as well; we will recall these in \S\ref{subsec_leading_terms_bis} below.

Finally, we assume that the rings $\cR_\f$ and $\cR_\g$ are both regular. We note this assumption can be ensured on passing to a smaller disc in the weight space (cf. \cite[\S5]{BSV}, where $\cR_?$ here, for $?=\f,\g$, plays the role of $\LL_?$ in op. cit.). That is, if we replace $\Lambda_{\mathrm{wt}}$ by the ring of bounded-by-$1$ functions $\LL_U$ on a sufficiently small wide-open disc $U$ in the weight space, and consider the base change of $\mathcal{R}_\f$ and $\mathcal{R}_\g$ along this restriction, the resulting rings are regular. 

We work in the scenario of \S\ref{subsubsec_2017_05_17_1800}, so that we have

\begin{itemize}
    \item[\mylabel{item_root_numbers_general}{${\bf (Sign)}$}] 
    \quad $\varepsilon(\hf)=-1$ \quad and \quad $\varepsilon(\hf_\kappa\otimes \Ad^0(\hg_\lambda))= \begin{cases}
         +1& \hbox{ if } (\kappa,\lambda)\in \cW_2^\Ad\,\\
         -1& \hbox{ if } (\kappa,\lambda)\in \cW_2^\hf\,
     \end{cases}$\,
\end{itemize}
for the global root numbers. In order to compare the module of leading terms associated with the pair $(T_\f^\dagger,\Delta_\emptyset)$, we will need to assume also the following conditions:
\begin{itemize}
\item[\mylabel{item_MC}{\bf MC})]
 $\f$ admits a crystalline specialisation of weight $k\equiv 2 \pmod{p-1}$, and either there exists a prime $q|| N$ such that $\overline{\rho}_\f(I_q)\neq \{1\}$, or there exists a real quadratic field $F$ verifying the conditions of \cite[Theorem 4]{xinwanwanhilbert}.
 \item[\mylabel{item_BI}{\bf BI})] $\rho_\f(G_{\QQ(\zeta_{p^\infty})})$ contains a conjugate of ${\rm SL}_2(\ZZ_p)$.     
 \item[\mylabel{item_non_anom}{\bf NA})] $a_p(\f)-1\in \cR_\f^\times$\,.
 \end{itemize}
We note that the hypothesis \eqref{item_MC} and the big image condition \eqref{item_BI}, which is a strengthening of \ref{item_Irr}, are to guarantee that the results on the validity of the Iwasawa main conjectures for $\f$ established in \cite{skinnerurbanmainconj,xinwanwanhilbert} apply. The hypothesis \eqref{item_non_anom} is only needed to ensure that certain local cohomology groups are free and it is certainly possible to dispense with it with more work.

\subsubsection{}
\label{subsubsec_leading_terms_bis_3}
We begin our discussion by noting that we have
$$\chi(\widetilde{R\Gamma}_{\rm f}(G_{\QQ, \Sigma}, T_\f^\dagger,\Delta_\emptyset))=-1\,$$
and we are in the situation of \S\ref{subsubsec_EP_char_is_1}.  The cyclic submodule
\begin{equation}
    \label{eqn_2022_09_12_1953}
    \delta(T_\f^\dagger,\Delta_\emptyset)\subset \widetilde{H}^1_{\rm f}(G_{\QQ, \Sigma},T_\f^\dagger, \Delta_\emptyset)
    \end{equation}
is generated by the Beilinson--Kato element ${\rm BK}_\f^\dagger\in H^1(\QQ,T_\f^\dagger)$ under the hypotheses recorded in \S\ref{subsubsec_hypo_section_6}:

\begin{proposition}
\label{prop_eqn_2022_09_12_1953_bis}
In the setting of \S\ref{subsubsec_hypo_section_6}, we have
\begin{equation}
\label{eqn_2022_09_12_1953_bis}
    \delta(T_\f^\dagger,\Delta_\emptyset)=\cR_\f\cdot {\rm BK}_\f^\dagger\,.
    \end{equation}    
\end{proposition}
\begin{proof} 
Let us put $\Sigma_\f := \{\ell \mid N_\f\} \cup \{p\}$; note that $\Sigma$ contains $\Sigma_\f$. We remark that we have an identification $\widetilde{R\Gamma}_{\rm f}(G_{\QQ, \Sigma},T_\f^\dagger, \Delta_\emptyset) = \widetilde{R\Gamma}_{\rm f}(G_{\QQ, \Sigma_\f},T_\f^\dagger, \Delta_\emptyset)$ of Selmer complexes by \cite[Proposition 7.8.8(ii)]{nekovar06}. By the definition of $\widetilde{R\Gamma}_{\rm f}(G_{\QQ, \Sigma},T_\f^\dagger, \Delta_\emptyset)$, we have
 \begin{align}
 \begin{aligned}
     \label{eqn_2026_05_13_1022}
     0 \lra \widetilde{H}^1_{\rm f}(G_{\QQ, \Sigma},T_\f^\dagger, \Delta_\emptyset) &\lra H^1(G_{\QQ, \Sigma_\f},T_\f^\dagger) \lra \bigoplus_{\ell \in \Sigma_\f \setminus \{p\}} \frac{ H^1(G_{\QQ_\ell},T_\f^\dagger)}{H^1_{\mathrm{ur}}(G_{\QQ_\ell},T_\f^\dagger) }
     \\
     &\lra \widetilde{H}^2_{\rm f}(G_{\QQ, \Sigma},T_\f^\dagger, \Delta_\emptyset) \lra H^2(G_{\QQ, \Sigma_\f},T_\f^\dagger) 
     \lra \bigoplus_{\ell \in \Sigma_\f \setminus \{p\}} H^2(G_{\QQ_\ell},T_\f^\dagger) \lra 0\,,
 \end{aligned}
 \end{align}
where we use our running assumption \eqref{item_non_anom} for the left-most vanishing. For each prime $\ell \in \Sigma_\f \setminus \{p\}$, the vanishing of the Tamagawa factors implies that $H^1(I_\ell, T_\f^\dagger)^{\mathrm{Fr}_\ell = 1} = \{0\}$ 
(see Proposition \ref{prop_I_v_infinite} and Corollary \ref{cor:tam_vanish_fin}), 
which tells us (by the inflation-restriction sequence) that 
\[
\frac{ H^1(G_{\QQ_\ell},T_\f^\dagger)}{H^1_{\mathrm{ur}}(G_{\QQ_\ell},T_\f^\dagger) } = \{0\}\,. 
\]
This fact, combined with \eqref{eqn_2026_05_13_1022}, tells us that 
\begin{equation}
\label{eqn_2026_05_13_1023}
    H^1(G_{\QQ, \Sigma_\f},T_\f^\dagger) = \widetilde{H}^1_{\rm f}(G_{\QQ, \Sigma},T_\f^\dagger, \Delta_\emptyset)\,. 
\end{equation}
Moreover, our running hypothesis \eqref{item_non_anom} implies $H^2(G_{\QQ_p},T_\f^\dagger) = 0$, which implies that 
\begin{equation}
\label{eqn_2026_05_13_1057}
    \widetilde{H}^2_{\rm f}(G_{\QQ, \Sigma},T_\f^\dagger, \Delta_\emptyset) = \ker\left( H^2(G_{\QQ, \Sigma_\f},T_\f^\dagger) 
     \lra \bigoplus_{\ell \in \Sigma_\f} H^2(G_{\QQ_\ell},T_\f^\dagger) \right)=\Sha^2_{\Sigma_\f}(T_\f^\dagger)\,,  
\end{equation}
where $\Sha^2_{\Sigma_\f}(T_\f^\dagger)$  is given as in \cite[Definition 2.3.2]{ochiaiAIF2005}. We then have 
\begin{align}
 \begin{aligned}
    \label{eqn_2026_05_13_1059}
 \mathrm{char}_{\mathcal{R}_\f}(\widetilde{H}^1_{\rm f}(G_{\QQ, \Sigma},T_\f^\dagger, \Delta_\emptyset)/\cR_\f\cdot {\rm BK}_\f^\dagger) 
 &=  \mathrm{char}_{\mathcal{R}_\f}(H^1(G_{\QQ, \Sigma_\f},T_\f^\dagger)/\cR_\f\cdot {\rm BK}_\f^\dagger) 
 \\
 &=  \mathrm{char}_{\mathcal{R}_\f}(\Sha^2_{\Sigma_\f}(T_\f^\dagger))
 \\
 &= 
 \mathrm{char}_{\mathcal{R}_\f}(\widetilde{H}^2_{\rm f}(G_{\QQ, \Sigma},T_\f^\dagger, \Delta_\emptyset)),    
  \end{aligned}
 \end{align}
 where the first equality follows from \eqref{eqn_2026_05_13_1023}, the second equality is the Iwasawa main conjecture ``without $p$-adic $L$-functions'', and the third equality follows from \eqref{eqn_2026_05_13_1057}. We remark that this formulation of the Iwasawa Main Conjecture (namely, the second equality in \eqref{eqn_2026_05_13_1059}) is equivalent to its more familiar form in terms of the Mazur--Kitagawa $p$-adic $ L$-function, thanks to the Kato--Ochiai reciprocity law (linking ${\rm BK}_\f^\dagger$ to the $p$-adic $L$-function; see \cite{ochiaiAIF2005}, Theorem B) and Poitou--Tate global duality; cf. \cite[\S17.13]{kato04} and \cite[p. 125]{ochiaiAIF2005} (Proof of Theorem C in op. cit.).

If ${\rm BK}_\f^\dagger = 0$, then the $\mathcal{R}_\f$-module $\widetilde{H}^2_{\rm f}(G_{\QQ, \Sigma},T_\f^\dagger, \Delta_\emptyset)$ is not torsion by \eqref{eqn_2026_05_13_1059}. As a result, it follows from Proposition~\ref{proposition_tortion_non-zero} that $\delta(T_\f^\dagger,\Delta_\emptyset) = \{0\}$ and our claimed is proved in this case. 
On the other hand, if ${\rm BK}_\f^\dagger \neq 0$, then the $\mathcal{R}_\f$-module $\widetilde{H}^2_{\rm f}(G_{\QQ, \Sigma},T_\f^\dagger, \Delta_\emptyset)$ is  torsion. Again by  Proposition~\ref{proposition_tortion_non-zero}, it follows that $\delta(T_\f^\dagger,\Delta_\emptyset) \neq 0$. 
Then the $\mathcal{R}_\f$-module $\widetilde{H}^1_{\rm f}(G_{\QQ, \Sigma},T_\f^\dagger, \Delta_\emptyset)$ is free of rank one (see  Proposition \ref{prop_suport_range_Selmer_complexes_8_6_2_2} below). Therefore, Corollary \ref{corollary_regular_delta_r=1} shows that 
$\delta(T_\f^\dagger,\Delta_\emptyset)=\cR_\f\cdot {\rm BK}_\f^\dagger$, as required.  
\end{proof}


\begin{proposition}
\label{prop_suport_range_Selmer_complexes_8_6_2_2}
We have 
\[
\widetilde{R\Gamma}(G_{\QQ,\Sigma},T_\f^\dagger,\Delta_{0})\,,\, \widetilde{R\Gamma}_{\rm f}(G_{\QQ,\Sigma},T_\f^\dagger,\Delta_{\emptyset}) \in D_{\rm parf}^{[1,2]}({}_{\cR_\f}{\rm Mod}).
\]
If $\delta(T_\f^\dagger,\Delta_\emptyset)\neq \{0\}$, then:   
\item[i)] The image of $\delta(T_\f^\dagger,\Delta_\emptyset)$ under the composite map
\[
\res_p^-: H^1(G_{\QQ,\Sigma},T_\f^\dagger)\lra H^1(G_p,T_\f^\dagger)\lra H^1(G_p,F^-T_\f^\dagger)
\]
is zero and $\widetilde{H}^1_{\rm f}(G_{\QQ,\Sigma},T_\f^\dagger,\Delta_{\mathrm{Pan}} )=\widetilde{H}^1_{\rm f}(G_{\QQ,\Sigma},T_\f^\dagger,\Delta_{\emptyset})$ are both free $\cR_\f$-modules of rank one. Here, we recall $\Delta_{\mathrm{Pan}} = \Delta_{F^+T^\dagger_{\hf}}$ are the Greenberg-local conditions on $T_\f^\dagger$ given by the submodule $F^+T_\f^\dagger\subset T_\f^\dagger$. 

\item[ii)]  $\widetilde{H}^2_{\rm f}(G_{\QQ,\Sigma},T_\f^\dagger,\Delta_{\emptyset})$ is torsion. Moreover,
\begin{equation}
    \label{eqn_prop_suport_range_Selmer_complexes_8_6_2_2_ii}
    \det\left(\widetilde{H}^1_{\rm f}(G_{\QQ,\Sigma},T_\f^\dagger,\Delta_{\emptyset})\big{/}\delta(T_\f^\dagger,\Delta_\emptyset)\right)=\det\left(\widetilde{H}^2_{\rm f}(G_{\QQ,\Sigma},T_\f^\dagger,\Delta_{\emptyset})\right)\,.
\end{equation}
\item[iii)] If in addition $ \res_p\left(\delta(T_\f^\dagger,\Delta_\emptyset)\right)\neq \{0\}$, then $\widetilde{H}^1_{\rm f}(G_{\QQ,\Sigma},T_\f^\dagger,\Delta_{0})=\{0\}$ and the $\cR_\f$-module $\widetilde{H}^2_{\rm f}(G_{\QQ,\Sigma},T_\f^\dagger,\Delta_{0})$ has rank one. 
\end{proposition}

\begin{proof}
The fact that the Selmer complexes are perfect and that they are concentrated in the indicated degrees follows from  Proposition \ref{prop_[1,2]_parf}.

Let us first prove that $\widetilde{H}^1_{\rm f}(G_{\QQ,\Sigma},T_\f^\dagger,\Delta_{\emptyset})$ is free of rank $1$. Since $\delta(T_\f^\dagger,\Delta_\emptyset)\neq 0$, the $\cR_\hf$-module $\widetilde{H}^2_{\rm f}(G_{\QQ,\Sigma},T_\f^\dagger,\Delta_{\emptyset})$ is torsion by Corollary~\ref{corollary_regular_delta_r=1}. The fact that $\widetilde{R\Gamma}_{\rm f}(G_{\QQ,\Sigma},T_\f^\dagger,\Delta_{\emptyset}) \in D_{\rm parf}^{[1,2]}({}_{\cR_\f}{\rm Mod})$ combined with the calculation $\chi(\widetilde{R\Gamma}_{\rm f}(G_{\QQ,\Sigma},T_\f^\dagger,\Delta_{\emptyset})) = -1$ show that the rank of $\widetilde{H}^1_{\rm f}(G_{\QQ,\Sigma},T_\f^\dagger,\Delta_{\emptyset})$ equals one. Moreover, since $R_{\hf}$ is a 2-dimensional regular local ring, 
\[
\mathrm{Ext}_{R_{\hf}}^1(\widetilde{H}^1_{\rm f}(G_{\QQ,\Sigma},T_\f^\dagger,\Delta_{\emptyset}), M) \cong \mathrm{Ext}_{R_{\hf}}^3(\widetilde{H}^2_{\rm f}(G_{\QQ,\Sigma},T_\f^\dagger,\Delta_{\emptyset}), M) = \{0\}
\] 
for any finitely generated $R_{\hf}$-module $M$. Thence, $\widetilde{H}^1_{\rm f}(G_{\QQ,\Sigma},T_\f^\dagger,\Delta_{\emptyset})$ is free.



Since $\varepsilon(\f) = -1$, the vanishing 
\begin{equation}
\label{eqn_2022_09_12_2030}
\res_p^-\left(\delta(T_\f^\dagger,\Delta_\emptyset)\right)=\{0\}
\end{equation}
follows from the comparison of the class $\delta(T_\f^\dagger,\Delta_\emptyset)$ with the Beilinson--Kato element as in \eqref{eqn_2022_09_12_1953_bis} and Kato's reciprocity laws. The equality  $\widetilde{H}^1_{\rm f}(G_{\QQ,\Sigma},T_\f^\dagger,\Delta_{\mathrm{Pan}})=\widetilde{H}^1_{\rm f}(G_{\QQ,\Sigma},T_\f^\dagger,\Delta_{\emptyset})
$ is a consequence of the vanishing \eqref{eqn_2022_09_12_2030} and
the fact $\widetilde{H}^1_{\rm f}(G_{\QQ,\Sigma},T_\f^\dagger,\Delta_{\emptyset})$ is free of rank one. This completes the proof of (i), and Part (ii) now follows from Corollary \ref{corollary_regular_delta_r=1}. 

Suppose now that  $\res_p\left(\delta(T_\f^\dagger,\Delta_\emptyset)\right)\neq \{0\}$. The assertion that $\widetilde{H}^1_{\rm f}(G_{\QQ,\Sigma},T_\f^\dagger,\Delta_0)=\{0\}$ follows from this assumption combined with the conclusions of (i), whereas the claim that the $\cR_\f$-module $\widetilde{H}^2_{\rm f}(G_{\QQ,\Sigma},T_\f^\dagger,\Delta_{0})$ is of rank one follows from the calculation $\chi(\widetilde{R\Gamma}_{\rm f}(G_{\QQ,\Sigma},T_\f^\dagger,\Delta_{0})) = 1$ of the global Euler--Poincar\'e characteristic. Our proof is now complete.
\end{proof}

\begin{remark}
\label{remark_Greenberg_conjecture}
A conjecture of Greenberg predicts that the derivative of the Manin-Vi\v{s}ik $p$-adic $L$-function $L_p(\f_\kappa,\sigma)$ at the central critical points is non-zero for some arithmetic specialisation $\kappa$; cf. \cite{Greenberg_1994_families} $($see also \cite{TrevorArnold_Greenberg_Conj}, Conjectures 1.3 and 1.4$)$. Assuming Greenberg's conjecture, combining with the main conjectures $($proved by Kato, Skinner--Urban, and Wan$)$ and using \eqref{eqn_2022_09_12_1953_bis} together with Kato's reciprocity laws, we deduce that $\res_p\left(\delta(T_\f^\dagger,\Delta_\emptyset)\right)\neq \{0\}$. 
\end{remark}

Before we conclude \S\ref{subsubsec_leading_terms_bis_3}, we describe a canonical trivialisation of the free module $H^1(G_p,F_p^+T_\f^\dagger)$ of rank one in terms of Coleman maps. We recall that the non-anomaly hypothesis \eqref{item_non_anom} is in effect.

Let us consider the unramified twist 
$$
M:=F^+T_\f(\chi_\cyc\bbchi_\f^{-1})\simeq \cR_\f(\alpha_\f^{-1})
$$ 
of the $G_p$-representation $F^+T_\f$, where $\alpha_\f$ is the unramified character that assigns the geometric Frobenius the value $a_p(\f)$ and gives the action of $G_p$ and $F^-T_\f$. Let us put $\cR_\f^\cyc := \cR_{\f} \widehat{\otimes} \LL(\Gamma_\cyc)$ to ease our notation.

The theory of Coleman maps gives rise to an isomorphism (cf. \cite{KLZ2}, \S8.2)
\begin{equation}
\label{eqn_2022_09_13_0843}    
{\rm LOG}_{M,0}\,:\,H^1(G_p,M\widehat{\otimes}_{\cR_\f}\cR_\f^\cyc)\stackrel{\sim}{\lra}\mathbf{D}_{\rm cris}(M)\,\widehat{\otimes}_{\cR_\f}\, \cR_\f^\cyc \,,
\end{equation}
where $\mathbf{D}_{\rm cris}(M)=(M\, \widehat\otimes\, \widehat{\Zp^{\mathrm ur}})^{G_{p}}$. 
Let us define the morphism 
\begin{equation}
\label{eqn_2022_09_13_0900}    
{\rm Log}_{F^+T_\f^\dagger}\,:\, H^1(G_p,F^+T_\f^\dagger)\lra \mathbf{D}_{\rm cris}(M) 
\end{equation}
on tensoring the isomorphism \eqref{eqn_2022_09_13_0843} of 
$\cR_\f^\cyc$-modules by $\cR_\f^\cyc\big{/}(\gamma-\chi_\cyc^{-1}{\chi_\f^{\frac{1}{2}}}(\gamma))$ (where $\gamma\in \Gamma_\cyc$ is a topological generator) 
and relying on the fact that the natural injection 
\[
H^1(G_p,F^+T_\f^\dagger\widehat{\otimes}_{\cR_\f}\cR_\f^\cyc)\Big{/}(\gamma-\chi_\cyc^{-1}{\chi_\f^{\frac{1}{2}}}(\gamma))\lra H^1(G_p,F^+T_\f^\dagger)
\]
is an isomorphism thanks to our running hypothesis \eqref{item_non_anom}. 
This in turn shows that the map ${\rm Log}_{F^+T_\f^\dagger}$ is indeed an isomorphism. 

To canonically trivialize the module $\mathbf{D}_{\rm cris}(M)$, we need Ohta's work~\cite{ohta00} as in \cite[Proposition 10.1.1]{KLZ2}. According to op. cit., there exists a canonical isomorphism
\begin{equation}
\label{eqn_2022_09_13_0945}   
\omega_\f\,:\,\mathbf{D}_{\rm cris}(M) \stackrel{\sim}{\lra} \cR_\f
\end{equation}
of $\cR_\f$-modules, since the branch $\f$ of the Hida family we are working with is cuspidal by assumption. Combining the isomorphism \eqref{eqn_2022_09_13_0900} with \eqref{eqn_2022_09_13_0945}, we obtain a canonical trivialisation
\begin{equation}
    \label{eqn_2022_09_13_0946}   
{\rm Log}_{\omega_\f}\,:\,H^1(G_p,F^+T_\f^\dagger) \xrightarrow[{\rm Log}_{F^+T_\f^\dagger}]{\sim} \mathbf{D}_{\rm cris}(M) \xrightarrow[\omega_\f]{\sim} \cR_\f\,.
\end{equation}

\subsubsection{}
\label{subsubsec_leading_terms_bis_2}
Since we have
\begin{equation}
    \label{eqn_2022_09_16_1342}
    \chi(\widetilde{R\Gamma}_{\rm f}(G_{\QQ, \Sigma}, M_2^\dagger,{\rm tr}^*\Delta_\g))=-1\,,
        \end{equation}
we are in the situation of \S\ref{subsubsec_EP_char_is_1} and we have a cyclic submodule
\begin{equation}
    \label{eqn_2022_09_12_1553}
    \delta(M_2^\dagger,{\rm tr}^*\Delta_\g)\subset \widetilde{H}^1_{\rm f}(G_{\QQ, \Sigma}, M_2^\dagger, {\rm tr}^*\Delta_\g)\,.
    \end{equation}
Whenever $\delta(M_2^\dagger,{\rm tr}^*\Delta_\g)\neq \{0\}$, we have thanks to Corollary~\ref{corollary_regular_delta_r=1} that
    \begin{align}
\begin{aligned}
    \label{eqn_2022_09_12_1543}
     &\widetilde{H}^1_{\rm f}(G_{\QQ, \Sigma}, M_2^\dagger,{\rm tr}^*\Delta_{\g}) \hbox{ is torsion-free of rank one},\quad  \widetilde{H}^2_{\rm f}(G_{\QQ, \Sigma}, M_2^\dagger,{\rm tr}^*\Delta_{\g}) \hbox{ is torsion},
     \\
    &\det\left(\widetilde{H}^1_{\rm f}(G_{\QQ, \Sigma}, M_2^\dagger,{\rm tr}^*\Delta_\g)\big{/}\delta(M_2^\dagger,{\rm tr}^*\Delta_\g)\right) = \det\left(\widetilde{H}^2_{\rm f}(G_{\QQ, \Sigma}, M_2^\dagger,{\rm tr}^*\Delta_\g)\right)\,. 
\end{aligned}
    \end{align}
On the other hand, we have 
$\chi(\widetilde{R\Gamma}_{\rm f}(G_{\QQ, \Sigma}, M_2^\dagger,{\rm tr}^*\Delta_{\rm bal}))=0\,,$ and we are in the situation of \S\ref{subsubsec_EP_char_is_zero}. We have a submodule
\begin{equation}
    \label{eqn_2022_09_12_1552}
    \delta(M_2^\dagger,{\rm tr}^*\Delta_{\rm bal})\subset \cR_2
    \end{equation} 
    of leading terms. According to Theorem~\ref{thm_PR_formal}, the submodules \eqref{eqn_2022_09_12_1553} and \eqref{eqn_2022_09_12_1552} are related to one another in the following manner. The short exact sequence \eqref{eqn_lemma_remark_ranks_3_3_2021_11_05_2} allows us to define the morphism
    $$\res_{/\rm bal}: \widetilde{H}^1_{\rm f}(G_{\QQ, \Sigma}, M_2^\dagger,{\rm tr}^*\Delta_{\g})\lra H^1(G_p, F^-T_\f^\dagger)\otimes_{\varpi_{2,1}^*}\cR_2\,.$$
    Let us fix a trivialisation\footnote{A canonical trivialisation of this module is given in terms of Perrin-Riou's large dual exponential maps; see the final portion of \S\ref{subsubsec_leading_terms_bis_3} for a prototypical example. Since we need not specify the choice of ${\rm Exp}_{F^-T_\f^\dagger}^*$ for our purposes, we will not dwell further on this point.} 
    \begin{equation}
    \label{eqn_2022_09_16_1244}
        {\rm Exp}_{F^-T_\f^\dagger}^*:\,H^1(G_p, F^-T_\f^\dagger) \stackrel{\sim}{\lra} \cR_\f
    \end{equation}
    and denote also by ${\rm EXP}_{F^-T_\f^\dagger}^*$ the composite map 
    \begin{equation}
        \label{eqn_2022_09_12_1647}
        \widetilde{H}^1_{\rm f}(G_{\QQ, \Sigma}, M_2^\dagger,{\rm tr}^*\Delta_{\g}) \lra H^1(G_p, F^-T_\f)\otimes_{\varpi_{2,1}^*}\cR_2\xrightarrow[{\rm EXP}_{F^-T_\f^\dagger}^*\otimes {\rm id}]{\sim} \cR_2.
    \end{equation}
    Then Theorem~\ref{thm_PR_formal} combined with \eqref{eqn_lemma_remark_ranks_3_3_2021_11_05_2} tells us that 
    \begin{equation}
    \label{eqn_2022_09_12_1620}
    {\rm Exp}_{F^-T_\f^\dagger}^*\left(\delta(M_2^\dagger,{\rm tr}^*\Delta_{\g})\right)=\delta(M_2^\dagger,{\rm tr}^*\Delta_{\rm bal})\,.
    \end{equation} 
    Moreover, if $ {\rm Exp}_{F^-T_\f^\dagger}^*\left(\delta(M_2^\dagger,{\rm tr}^*\Delta_{\g})\right)\neq 0$, then it follows from Theorem~\ref{theorem_delta_rank0}, Theorem~\ref{thm_PR_formal} and \eqref{eqn_2022_09_12_1620} that
    \begin{equation}
    \label{eqn_2022_09_12_1626}
    {\rm Exp}_{F^-T_\f^\dagger}^*\left(\delta(M_2^\dagger,{\rm tr}^*\Delta_{\g})\right)=\det\left(\widetilde{H}^2_{\rm f}(G_{\QQ, \Sigma}, M_2^\dagger,{\rm tr}^*\Delta_{\rm bal})\right)
    \end{equation} 
    is the module of $p$-adic $L$-functions, which is expected to be generated by the degree-6 balanced $p$-adic $L$-function $\cL_p^\Ad(\hf\otimes \Ad^0\hg)$. A generalisation of Greenberg's conjecture leads to the prediction that $\cL_p^\Ad(\hf\otimes \Ad^0\hg)\neq 0$ under our assumption \ref{item_root_numbers_general} on global root numbers.
    
    We note that $\widetilde{R\Gamma}_{\rm f}(G_{\QQ, \Sigma}, M_2^\dagger,{\rm tr}^*\Delta_{?})\in D_{\rm parf}^{[1,2]}({}_{\cR_2}{\rm Mod})$ (for $?=\g,{\rm bal}$) and hence
    \begin{equation}
    \label{eqn_2022_09_12_1644}
    \widetilde{H}^1_{\rm f}(G_{\QQ, \Sigma}, M_2^\dagger,{\rm tr}^*\Delta_{\rm bal})=\{0\},\quad  \widetilde{H}^2_{\rm f}(G_{\QQ, \Sigma}, M_2^\dagger,{\rm tr}^*\Delta_{\rm bal}) \hbox{ is torsion}\,,
    \end{equation}
    \begin{align}
    \label{eqn_2022_09_12_1651}
    \begin{aligned}
         &\widetilde{H}^1_{\rm f}(G_{\QQ, \Sigma}, M_2^\dagger,{\rm tr}^*\Delta_{\g}) \lra H^1(G_p, F^-T_\f)\otimes_{\varpi_{2,1}^*}\cR_2\xrightarrow{\sim}  \cR_2 \quad \hbox{ is injective}
    \end{aligned}
    \end{align}
    whenever $ {\rm Exp}_{F^-T_\f^\dagger}^*\left(\delta(M_2^\dagger,{\rm tr}^*\Delta_{\g})\right)\neq \{0\}$.

\subsubsection{}
\label{subsubsec_leading_terms_bis_1}
Note that we have $\widetilde{R\Gamma}_{\rm f}(G_{\QQ, \Sigma}, T_?^\dagger,\Delta_\g)\in D_{\rm parf}^{[1,2]}({}_{\cR_?}{\rm Mod})$ thanks to Proposition~\ref{prop_[1,2]_parf}, and we have $\chi(\widetilde{R\Gamma}_{\rm f}(G_{\QQ, \Sigma}, T_?^\dagger,\Delta_\g))=0$ for the Euler--Poincar\'e characteristics of the indicated Selmer complex (?=2,3). We are therefore in the situation of \S\ref{subsubsec_EP_char_is_zero} and we have submodules
\begin{align}
\label{eqn_2022_09_16_1209}
\begin{aligned}
 \delta(T_3^\dagger,\Delta_\g) \subset \cR_3\,, 
\qquad \qquad 
\delta(T_2^\dagger,\Delta_\g) \subset \cR_2
\end{aligned}
\end{align}
of leading terms which are related via
$\delta(T_3^\dagger,\Delta_\g)\otimes_{\cR_3}\cR_2=\delta(T_2^\dagger,\Delta_\g)\,,$ cf. Proposition~\ref{prop_51_2022_09_12_1507}. Whenever $\delta(T_3^\dagger,\Delta_\g)\neq \{0\}$, we have thanks to Theorem~\ref{theorem_delta_rank0} that 
\[
\delta(T_3^\dagger,\Delta_\g)=\det \left(\widetilde{R\Gamma}_{\rm f}(G_{\QQ, \Sigma}, T_3^\dagger,\Delta_\g)\right)
\]
is the module of algebraic $p$-adic $L$-functions, and it is expected to be generated by the $p$-adic $L$-function $\cL_p^{\hg}(\f\otimes\g\otimes\g^c)^2$, which is given as the image of a diagonal cycle under a large Perrin-Riou logarithm map.

A similar discussion applies when $\delta(T_2^\dagger,\Delta_\g)\neq \{0\}$: 
\begin{equation}
    \label{eqn_2022_09_12_1711}
    \delta(T_2^\dagger,\Delta_\g)=\det \left(\widetilde{R\Gamma}_{\rm f}(G_{\QQ, \Sigma}, T_2^\dagger,\Delta_\g)\right) = \det \left(\widetilde{H}^2_{\rm f}(G_{\QQ, \Sigma}, T_2^\dagger,\Delta_\g)\right)
    \end{equation}
and it is expected to be generated by the $p$-adic $L$-function $\cL_p^{\hg}(\f\otimes\g\otimes\g^c)^2(\kappa,\lambda,\lambda)$ whose range of interpolation is empty.

Recall also the local conditions $\Delta_+$ given as in Example~\ref{example_local_conditions}(iv). We have  $\widetilde{R\Gamma}_{\rm f}(G_{\QQ, \Sigma}, T_?^\dagger,\Delta_+)\in D_{\rm parf}^{[1,2]}({}_{\cR_?}{\rm Mod})$ (where $?=2,3$) and $\chi\left(\widetilde{R\Gamma}_{\rm f}(G_{\QQ, \Sigma}, T_?^\dagger,\Delta_+) \right)=-1$. As a result, we are in the situation of \S\ref{subsubsec_EP_char_is_1} and we have a submodule
\begin{equation}
    \label{eqn_2022_09_16_1210}\delta(T_?^\dagger, \Delta_+) \in \widetilde{H}^1_{\rm f}(G_{\QQ, \Sigma}, T_?^\dagger, \Delta_+)\,,\qquad ?=2,3.
\end{equation}
We explain the manner the submodules \eqref{eqn_2022_09_16_1209} and \eqref{eqn_2022_09_16_1210} are related to one another in view of Theorem~\ref{thm_PR_formal}. 

Let us consider the natural map (induced by the inclusion $F^+_{\g}T_?^\dagger \longrightarrow F^+_{\textrm{bal}+}T_?^\dagger$, cf. Example~\ref{example_local_conditions}(iv)) 
\begin{equation}
\label{eqn_2022_09_16_1417}
    \res^{(/\g)}_{p}\,:\quad\widetilde{H}^1_{\rm f}(G_{\QQ,\Sigma},T_?^\dagger,\Delta_{+})\lra {H}^1(G_p,F^+T_\f^\dagger\,\widehat\otimes\,F^-T_\g \,\widehat\otimes\, F^+T_\g^*)\xrightarrow{\sim}{H}^1(G_p,F^+T_\f^\dagger)\,\widehat\otimes\,\cR_? \,.
\end{equation}

The canonical trivialisation \eqref{eqn_2022_09_13_0946} gives rise to a morphism
$$
\widetilde{H}^1_{\rm f}(G_{\QQ,\Sigma},T_?^\dagger,\Delta_{+})\xrightarrow{\res_p^{(/\g)}}{H}^1(G_p,F^+T_\f^\dagger)\,\widehat\otimes\,\cR_?\xrightarrow[{\rm Log}_{\omega_\f}]{\sim} \cR_?
$$
which we also denote by ${\rm Log}_{\omega_\f}$.  Then Theorem~\ref{thm_PR_formal} combined with the discussion above  yields 
\begin{equation}
    \label{eqn_2022_09_16_1252}
    {\rm Log}_{\omega_\f}\left(\delta(T_?^\dagger,\Delta_{+})\right)=\delta(T_?^\dagger,\Delta_\g)\,.
    \end{equation} 
    Moreover, if ${\rm Log}_{\omega_\f}\left(\delta(T_?^\dagger,\Delta_{+})\right)\neq 0$, then it follows from Theorem~\ref{theorem_delta_rank0}, Theorem~\ref{thm_PR_formal} and \eqref{eqn_2022_09_16_1252} that
    \begin{equation}
    \label{eqn_2022_09_16_1253}
    {\rm Log}_{\omega_\f}\left(\delta(T_?^\dagger,\Delta_{+})\right) = \delta(T_?^\dagger,\Delta_\g)=\det\left(\widetilde{H}^2_{\rm f}(G_{\QQ, \Sigma}, T_?^\dagger,\Delta_{\g})\right)
    \end{equation} 
    is the module of $p$-adic $L$-functions, which is expected to be generated by the degree-8 balanced $p$-adic $L$-function $\cL_p^{(\g)}(\hf\otimes \hg\otimes \hg^c)$ (if $?=3$) and by $\cL_p^{(\g)}(\hf\otimes \hg\otimes \hg^c)(\kappa,\lambda,\lambda)$ (if $?=2$). 
    
    We further remark that we have
    \begin{equation}
    \label{eqn_2022_09_16_1301}
    \widetilde{H}^1_{\rm f}(G_{\QQ, \Sigma}, T_?^\dagger,\Delta_{\g})=\{0\},\quad  \widetilde{H}^2_{\rm f}(G_{\QQ, \Sigma}, T_2^\dagger,\Delta_{\g}) \hbox{ is torsion}\,,
    \end{equation}
    \begin{align}
    \label{eqn_2022_09_16_1302}
    \begin{aligned}
         &\widetilde{H}^1_{\rm f}(G_{\QQ, \Sigma}, T_?^\dagger,\Delta_{+}) \hbox{ is torsion-free of rank one},\quad  \widetilde{H}^2_{\rm f}(G_{\QQ, \Sigma}, T_?^\dagger,\Delta_{+}) \hbox{ is torsion},\\
         &\widetilde{H}^1_{\rm f}(G_{\QQ, \Sigma}, T_?^\dagger,\Delta_{+}) \xrightarrow{\res_p^{(/\g)}} H^1(G_p, F^-T_\f)\,\widehat{\otimes}\,\cR_?\xrightarrow{\sim}  \cR_? \quad \hbox{ is injective}
    \end{aligned}
    \end{align}
    provided that $ \delta(T_?^\dagger,\Delta_\g)={\rm Log}_{\omega_\f}\left(\delta(T_?^\dagger,\Delta_{+})\right)\neq \{0\}$.

\subsubsection{Further consequences}
\label{subsubsec_consequences}
We conclude \S\ref{subsec_leading_terms_bis} with a non-vanishing criterion for the leading terms to supplement the discussion in \S\ref{subsubsec_leading_terms_bis_3}--\S\ref{subsubsec_leading_terms_bis_1}.

Similar to \eqref{eqn_2022_09_16_1417}, we may also consider the natural map  
$$
\res^{(/{\rm bal})}_{p}\,:\quad\widetilde{H}^1_{\rm f}(G_{\QQ,\Sigma},T_?^\dagger,\Delta_{+})\lra {H}^1(G_p,F^-T_\f^\dagger\,\widehat\otimes\,F^+T_\g \,\widehat\otimes\, F^-T_\g^*)\xrightarrow{\sim}{H}^1(G_p,F^-T_\f^\dagger)\,\widehat\otimes\,\cR_? \,.
$$
induced by the inclusion $F^+_{\rm bal}T_?^\dagger \longrightarrow F^+_{\textrm{bal}+}T_?^\dagger$, cf. Example~\ref{example_local_conditions}(iv).

The trivialisation \eqref{eqn_2022_09_16_1244} gives rise to a morphism
$$\widetilde{H}^1_{\rm f}(G_{\QQ,\Sigma},T_?^\dagger,\Delta_{+})\xrightarrow{\res_p^{(/{\rm bal})}}{H}^1(G_p,F^-T_\f^\dagger)\,\widehat\otimes\,\cR_?\xrightarrow[{\rm Exp}_{F^-T_\f^\dagger}^*]{\sim} \cR_?$$
which we also denote by ${\rm Exp}_{F^-T_\f^\dagger}^*$.  Then Theorem~\ref{thm_PR_formal} combined with the discussion above  yields 
\begin{equation}
    \label{eqn_2022_09_16_1422}
   {\rm Exp}_{F^-T_\f^\dagger}^*\left(\delta(T_?^\dagger,\Delta_{+})\right)=\delta(T_?^\dagger,\Delta_{\rm bal})\,.
    \end{equation} 
In view of the Iwasawa main conjectures and given the condition 
\[
\varepsilon(\f_\kappa\otimes \g_\lambda\otimes\g^*_\mu)=-1\,,\qquad (\kappa,\lambda,\mu)\in \cW_3^{\rm bal}
\]
on the global root numbers in the scenario we have placed ourselves, we expect in light of Theorem~\ref{theorem_delta_rank0} that 
\begin{equation}
    \label{eqn_2022_09_16_1430} 
    {\rm Exp}_{F^-T_\f^\dagger}^*\left(\delta(T_?^\dagger,\Delta_{+})\right)=\delta(T_?^\dagger,\Delta_{\rm bal})= \{0\}.
        \end{equation} 
        This is a reflection of the fact that the balanced triple product $p$-adic $L$-function vanishes identically.

\begin{proposition}
\label{prop_suport_range_Selmer_complexes_8_6_2_8}
If
$\delta(T_2^\dagger,\Delta_\g)\neq \{0\}$, 
then also $\res_p\left(\delta(T_\f^\dagger,\Delta_\emptyset)\right)\neq \{0\} \neq \delta(M_2^\dagger,{\rm tr}^*\Delta_{\g})$. Moreover, $\delta(T_2^\dagger,\Delta_{\rm bal})= \{0\}$.
\end{proposition}

\begin{proof}
We begin by noting that the natural map
$$\widetilde{H}^1_{\rm f}(G_{\QQ,\Sigma},T_\f^\dagger,\Delta_{0})\otimes_{\varpi_{2,1}^*}\cR_2\lra \widetilde{H}^1_{\rm f}(G_{\QQ,\Sigma},T_2^\dagger,\Delta_{\g})$$
is injective  thanks to \eqref{eqn_sequence_dual_trace_derived_category}  and our running hypotheses that \ref{item_Irr} holds for $\f$ and \ref{item_Irr_plus} holds for $\g$. This together with \eqref{eqn_2022_09_16_1301} show that  $\widetilde{H}^1_{\rm f}(G_{\QQ,\Sigma},T_\f^\dagger,\Delta_{0})=\{0\}$. We infer using duality of Selmer complexes (and the self-duality of $T_\f^\dagger$) that $\widetilde{H}^2_{\rm f}(G_{\QQ,\Sigma},T_\f^\dagger,\Delta_{\emptyset})$ is torsion. Theorem~\ref{theorem_delta_rank1} tells us that $\delta(T_\f^\dagger,\Delta_\emptyset)\neq \{0\}$. Moreover, since 
$$\{0\}=\widetilde{H}^1_{\rm f}(G_{\QQ,\Sigma},T_\f^\dagger,\Delta_{0})=\ker\left(\widetilde{H}^1_{\rm f}(G_{\QQ,\Sigma},T_\f^\dagger,\Delta_{\emptyset})\xrightarrow{\res_p} H^1(G_p, T_\f^\dagger)\right)\,,$$
we conclude that $\res_p\left(\delta(T_\f^\dagger,\Delta_\emptyset)\right)\neq \{0\}$, as required.

The prove the second asserted non-vanishing, we note that the exact sequence \eqref{eqn_sequence_dual_trace_derived_category} combined with the discussion in the preceding paragraph gives rise to an exact sequence
\begin{equation}
    \label{eqn_2022_09_16_1335}
    0\lra \widetilde{H}^1_{\rm f}(G_{\QQ,\Sigma},M_2^\dagger,{\rm tr}^*\Delta_{\g})\lra \underbrace{\widetilde{H}^2_{\rm f}(G_{\QQ,\Sigma},T_\f^\dagger,\Delta_{0})\otimes_{\varpi_{2,1}^*} \cR_2}_{\mathrm{rank} \, = \, 1}\lra \underbrace{\widetilde{H}^2_{\rm f}(G_{\QQ,\Sigma},T_2^\dagger,\Delta_{\g})}_{\rm torsion}
\end{equation}
of $\cR_2$-modules. The exact sequence \eqref{eqn_2022_09_16_1335} then shows that $\widetilde{H}^1_{\rm f}(G_{\QQ,\Sigma},M_2^\dagger,{\rm tr}^*\Delta_{\g})$ is of rank one. Combining the discussion in \S\ref{subsubsec_EP_char_is_1} with \eqref{eqn_2022_09_16_1342}, we conclude that
$\delta(M_2^\dagger,{\rm tr}^*\Delta_{\g})\neq \{0\}\,.$

By Remark~\ref{remark_defn_propagate_local_conditions_via_dual_trace_V_fdagger}, we have an injection \begin{align}\label{eqn_exact_sequence_dual_trace_derived_category_balanced}
   \widetilde{H}^1_{\rm f}(G_{\QQ,\Sigma},T_\f^\dagger,\Delta_{\mathrm{Pan}})\otimes_{\varpi_{2,1}^*}\cR_2\xrightarrow{{\rm id}\otimes {\rm tr}^*} \widetilde{H}^1_{\rm f}(G_{\QQ,\Sigma},T_2^\dagger,\Delta_{\rm bal})\,.
\end{align}
Since we have $\delta(T_\f^\dagger,\Delta_\emptyset)\neq \{0\}$, it follows from Proposition~\ref{prop_suport_range_Selmer_complexes_8_6_2_2} that the $\cR_\f$-module $\widetilde{H}^1_{\rm f}(G_{\QQ,\Sigma},T_\f^\dagger,\Delta_{\mathrm{Pan}})$ is free of rank one. This together with \eqref{eqn_exact_sequence_dual_trace_derived_category_balanced} show that $\delta(T_2^\dagger,\Delta_{\rm bal}) = \{0\}$. 
\end{proof}

We will considerably strengthen Proposition~\ref{prop_suport_range_Selmer_complexes_8_6_2_8} in Theorem~\ref{thm_main_8_4_4_factorisation}.

\subsection{Factorisation}
\label{subsec_factor_general}
Our main goal in \S\ref{subsec_factor_general} is to establish Theorem~\ref{thm_main_8_4_4_factorisation}, which is one of the main results in the present article. Before doing so in \S\ref{subsubsec_factor_general_20_05_2021}, we will need to study the connecting morphism $\delta^1$.

    The homomorphism $\mathrm{tr} \colon T_2^\dagger \longrightarrow T_\f^\dagger \otimes_{\varpi_{2,1}^*} \cR_2$ induces a morphism 
    \[
    \widetilde{R\Gamma}_{\rm f}(G_{\QQ,\Sigma}, T_2^\dagger,\Delta_\g)
     \longrightarrow 
    \widetilde{R\Gamma}_{\rm f}(G_{\QQ,\Sigma},T_\f^\dagger,\Delta_\emptyset) \otimes_{\varpi_{2,1}^*} \cR_2\,. 
    \]
    Moreover, we also have a morphism (see the proof of Lemma \ref{lemma_remark_ranks_3_3_2021_11_05}) 
    \[
    \widetilde{R\Gamma}_{\rm f}(G_{\QQ,\Sigma}, M_2^\dagger, \mathrm{tr}^*\Delta_\g)
     \longrightarrow 
    \widetilde{R\Gamma}_{\rm f}(G_{p}, F_\g^+M_2^\dagger/F_\g M_2^\dagger) \cong \widetilde{R\Gamma}_{\rm f}(G_{\QQ,\Sigma},T_\f^\dagger,\Delta_\emptyset) \otimes_{\varpi_{2,1}^*} \cR_2\,. 
    \]
    Combining Remark \ref{remark_Fg+diagram-comuutes} and \eqref{eqn_sequence_dual_trace_derived_category}, we obtain a commutative diagram 
    \[
    \xymatrix{
 \widetilde{R\Gamma}_{\rm f}(G_{\QQ,\Sigma}, T_\f^\dagger, \Delta_{0})  \otimes_{\varpi_{2,1}^*} \cR_2 
    \ar[r] \ar[d]^{=} &
    \widetilde{R\Gamma}_{\rm f}(G_{\QQ,\Sigma},T_2^\dagger,\Delta_{\g}) \ar[r]  \ar[d] & \widetilde{R\Gamma}_{\rm f}(G_{\QQ,\Sigma},M_2^\dagger,{\rm tr}^*\Delta_{\g}) \ar[d]
    \ar[r]^-{\delta}_-{+1} &
    \\
      \widetilde{R\Gamma}_{\rm f}(G_{\QQ,\Sigma},T_\f^\dagger,\Delta_0)  \otimes_{\varpi_{2,1}^*} \cR_2   \ar[r] &  \widetilde{R\Gamma}_{\rm f}(G_{\QQ,\Sigma},T_\f^\dagger,\Delta_\emptyset) \otimes_{\varpi_{2,1}^*} \cR_2 
          \ar[r] &
      {R\Gamma}(G_p,T_\f^\dagger) \otimes_{\varpi_{2,1}^*} \cR_2
       \ar[r]^-{\delta}_-{+1} &, 
    }
\]   
which yields the following commutative diagram with exact rows: 
    \begin{align}
    \begin{split}
\label{diagram:connectinghom-M2trDelta_g}
       \xymatrix{
    \widetilde{H}_{\rm f}^1(G_{\QQ,\Sigma},T_2^\dagger,\Delta_{\g}) \ar[r]  \ar[d] & \widetilde{H}_{\rm f}^1(G_{\QQ,\Sigma},M_2^\dagger,{\rm tr}^*\Delta_{\g}) \ar[d]
    \ar[r]^-{\delta^1} & \widetilde{H}_{\rm f}^2(G_{\QQ,\Sigma}, T_\f^\dagger, \Delta_{0})  \otimes_{\varpi_{2,1}^*} \cR_2 \ar[d]^{=}
    \\
\widetilde{H}_{\rm f}^1(G_{\QQ,\Sigma},T_\f^\dagger,\Delta_\emptyset) \otimes_{\varpi_{2,1}^*} \cR_2 
          \ar[r]^-{\mathrm{res}_p \otimes {\rm id}} &
      H^1(G_p,T_\f^\dagger) \otimes_{\varpi_{2,1}^*} \cR_2
       \ar[r]^-{\partial^1_\f \otimes {\rm id}} &      \widetilde{H}_{\rm f}^2(G_{\QQ,\Sigma},T_\f^\dagger,\Delta_0)  \otimes_{\varpi_{2,1}^*} \cR_2. 
       }
\end{split}
\end{align}

\begin{corollary}
\label{cor_prop_20_05_2021_4_10}
Suppose that 
$\delta(T_2^\dagger,\Delta_\g)\neq \{0\}\,.
$ 
Then the map $\delta^1$ factors as 
\begin{equation}
    \label{eqn_20052021_5_4_bis}
    \begin{aligned}
        \xymatrix{
    \widetilde{H}^1_{\rm f}(G_{\QQ,\Sigma},M_2^\dagger,{\rm tr}^*\Delta_{\g})\ar@{^{(}->}[rr]^{\delta^1} \ar@{^{(}->}[rd]_{\res_{/\emptyset}}&& \widetilde{H}^2_{\rm f}(G_{\QQ,\Sigma},T_\f^\dagger,\Delta_{0})\otimes_{\varpi_{2,1}^*}\cR_2\\
     &\dfrac{H^1(G_p,T_\f^\dagger)}{\res_p(\widetilde{H}^1_{\rm f}(G_{
     \QQ,S},T_\f^\dagger,\Delta_{\emptyset}))}\otimes_{\varpi_{2,1}^*}\cR_2\ar@{^{(}->}[ru]_(.58){\partial^1_\f \otimes {\rm id}}&
    }
    \end{aligned}
\end{equation}
and all the $\cR_2$-modules that appear in \eqref{eqn_20052021_5_4_bis} are of rank one.
\end{corollary}

\begin{proof}
This follows from \eqref{eqn_2022_09_16_1301} and \eqref{diagram:connectinghom-M2trDelta_g}. 
\end{proof}

\subsubsection{}
\label{subsubsec_factor_general_20_05_2021_prelim}
In \S\ref{subsubsec_factor_general_20_05_2021_prelim}, we work in the setting and under the assumptions of Corollary~\ref{cor_prop_20_05_2021_4_10} and \S\ref{subsec_leading_terms_bis}. 
Recall our map $\res_{/{\rm bal}}$ given as the composite map $$\widetilde{H}^1_{\rm f}(G_{\QQ,\Sigma},M_2^\dagger,{\rm tr}^*\Delta_{\g})\xrightarrow{{\rm pr}_{/\g}\,\circ\,\res_p} H^1(G_p,T_\f^\dagger) \otimes_{\varpi_{2,1}^*} \cR_2\lra H^1(G_p,F^-T_\f^\dagger) \otimes_{\varpi_{2,1}^*} \cR_2\,$$
hence it factors as 
\begin{align}\label{eq:res_bal_factrisation}
\res_{/{\rm bal}}:\,\widetilde{H}^1_{\rm f}(G_{\QQ,\Sigma},M_2^\dagger,{\rm tr}^*\Delta_{\g})\xrightarrow{\res_{/\emptyset}} \dfrac{H^1(G_p,T_\f^\dagger)}{\res_p(\widetilde{H}^1_{\rm f}(G_{\QQ,S},T_\f^\dagger,\Delta_{\emptyset}))}\otimes_{\varpi_{2,1}^*}\cR_2 \stackrel{\pi}{\twoheadrightarrow} H^1(G_p,F^-T_\f^\dagger) \otimes_{\varpi_{2,1}^*} \cR_2\,.     
\end{align}
Thanks to this factorisation, when $\delta(T_2^\dagger,\Delta_\g)\neq \{0\}$,  the map $\res_{/{\rm bal}}$ is injective since 
\begin{itemize}
\item $\res_{/\emptyset}$ is injective,  
\item $\widetilde{H}^1_{\rm f}(G_{\QQ,\Sigma},M_2^\dagger,{\rm tr}^*\Delta_{\g})$ is  torsion-free, and 
\item $\ker(\pi)$ is torsion. (Note that $\widetilde{H}^1_{\rm f}(G_{
     \QQ,S},T_\f^\dagger,\Delta_{\emptyset})=\widetilde{H}^1_{\rm f}(G_{
     \QQ,S},T_\f^\dagger,\Delta_{\mathrm{Pan}})$ and $\mathrm{res}_p$ is injective  by Proposition~\ref{prop_suport_range_Selmer_complexes_8_6_2_2}.) 
\end{itemize}
Thence, we have an exact sequence of torsion $\cR_2$-modules
\begin{align}
    \label{eqn_exact_seq_2152021_5_22}
    \begin{aligned}
    0\lra \dfrac{H^1(G_p,F^+T_\f^\dagger)\otimes_{\varpi_{2,1}^*}\cR_2}{\res_p\left(\widetilde{H}^1_{\rm f}(G_{\QQ,\Sigma},T_\f^\dagger,\Delta_{\emptyset})\right)\otimes_{\varpi_{2,1}^*}\cR_2}
   & \lra 
    \dfrac{\dfrac{H^1(G_p,T_\f^\dagger)\otimes_{\varpi_{2,1}^*}\cR_2}{\res_p\left(\widetilde{H}^1_{\rm f}(G_{\QQ,\Sigma},T_\f^\dagger,\Delta_{\emptyset})\right)\otimes_{\varpi_{2,1}^*}\cR_2}}{\res_{/\emptyset}(\delta(M_2^\dagger,{\rm tr}^*\Delta_{\g}))}
    \\
    &\qquad\qquad\qquad\qquad\lra
   \dfrac{H^1(G_p,F^-T_\f^\dagger)\otimes_{\varpi_{2,1}^*}\cR_2}{{\res}_{/{\rm bal}}(\delta(M_2^\dagger,{\rm tr}^*\Delta_{\g}))} \lra  0. 
   \end{aligned}
\end{align}

\begin{proposition}
\label{prop_useful_step_in_factorisation_21_05_2021}
In the setting of Corollary~\ref{cor_prop_20_05_2021_4_10}, we have
\begin{align*}
{\rm char}\left(\dfrac{\widetilde{H}^2_{\rm f}(G_{\QQ,\Sigma},T_\f^\dagger,\Delta_{0})\otimes_{\varpi_{2,1}^*}\cR_2}{\delta^1(\delta(M_2^\dagger,{\rm tr}^*\Delta_{\g}))}\right)
=
{\rm char}\left( \dfrac{H^1(G_p,F^-T_\f^\dagger)\otimes_{\varpi_{2,1}^*}\cR_2}{{\res}_{/{\rm bal}}(\delta(M_2^\dagger,{\rm tr}^*\Delta_{\g}))}\right)
{\rm char}\left( \dfrac{H^1(G_p,F^+T_\f^\dagger)}{\res_p(\delta(T_\f^\dagger,\Delta_\emptyset))}\otimes_{\varpi_{2,1}^*}\cR_2\right) \,.
\end{align*}
\end{proposition}

\begin{proof}
Combining Corollary~\ref{cor_prop_20_05_2021_4_10} and \eqref{eqn_exact_seq_2152021_5_22}, we infer that
\begin{align}
\label{eqn_proof_step_1_prop_useful_step_in_factorisation_21_05_2021}
\notag {\rm char}\left(\dfrac{\widetilde{H}^2_{\rm f}(G_{\QQ,\Sigma},T_\f^\dagger,\Delta_{0})\otimes_{\varpi_{2,1}^*}\cR_2}{\delta^1(\delta(M_2^\dagger,{\rm tr}^*\Delta_{\g}))}\right)&={\rm char}\left( \dfrac{H^1(G_p,F^-T_\f^\dagger)\otimes_{\varpi_{2,1}^*}\cR_2}{{\res}_{/{\rm bal}}(\delta(M_2^\dagger,{\rm tr}^*\Delta_{\g}))}\right)\\
&\quad\times
{\rm char}\left( \dfrac{H^1(G_p,F^+T_\f^\dagger)}{\res_p\left(\widetilde{H}^1_{\rm f}(G_{\QQ,\Sigma},T_\f^\dagger,\Delta_{\emptyset})\right)}\otimes_{\varpi_{2,1}^*}\cR_2\right) 
{\rm char}\left({\rm coker}(\partial^1_\f\otimes{\rm id}) \right)\,.
\end{align}
The long exact sequence
\begin{equation*}
    \label{eqn_exact_seq_06052021_5_8}
    \begin{aligned}
    \xymatrix@C=0.6cm{0\ar[r]& \widetilde{H}^1_{\rm f}(G_{\QQ,\Sigma},T_\f^\dagger,\Delta_{\emptyset})\ar[r]^(.55){\res_p}&H^1(G_p,T_\f^\dagger)\ar[r]^(.43){\partial^1_{\f}}& \widetilde{H}^2_{\rm f}(G_{\QQ,\Sigma},T_\f^\dagger,\Delta_{0})\ar[r]& \widetilde{H}^2_{\rm f}(G_{\QQ,\Sigma},T_\f^\dagger,\Delta_{\emptyset})\ar[r]&0
    }
    \end{aligned}
\end{equation*}
shows that ${\rm coker}(\partial_\f^1\otimes{\rm id})\xrightarrow{\sim} \widetilde{H}^2_{\rm f}(G_{\QQ,\Sigma},T_\f^\dagger,\Delta_{\emptyset})\otimes_{\varpi_{2,1}^*}\cR_2$. 
We may reorganize \eqref{eqn_proof_step_1_prop_useful_step_in_factorisation_21_05_2021} as
\footnotesize
\begin{align*}\label{eqn_proof_step_2_prop_useful_step_in_factorisation_21_05_2021}
&{\rm char}\left(\dfrac{\widetilde{H}^2_{\rm f}(G_{\QQ,\Sigma},T_\f^\dagger,\Delta_{0})\otimes_{\varpi_{2,1}^*}\cR_2}{\delta^1(\delta(M_2^\dagger,{\rm tr}^*\Delta_{\g}))}\right)\\
&\stackrel{\eqref{eqn_prop_suport_range_Selmer_complexes_8_6_2_2_ii}}{=} 
{\rm char}\left( \dfrac{H^1(G_p,F^-T_\f^\dagger)\otimes_{\varpi_{2,1}^*}\cR_2}{  {\res}_{/{\rm bal}}(\delta(M_2^\dagger,{\rm tr}^*\Delta_{\g}))}\right)   {\rm char}\left( \dfrac{H^1(G_p,F^+T_\f^\dagger)}{\res_p\left(\widetilde{H}^1_{\rm f}(G_{\QQ,\Sigma},T_\f^\dagger,\Delta_{\emptyset})\right)}\otimes_{\varpi_{2,1}^*}\cR_2\right) {\rm char}\left(\dfrac{\widetilde{H}^1_{\rm f}(G_{\QQ,\Sigma},T_\f^\dagger,\Delta_{\emptyset})}{R_\f\cdot\delta(T_\f^\dagger,\Delta_\emptyset)}\otimes_{\varpi_{2,1}^*}\cR_2\right)\\
&={\rm char}\left( \dfrac{H^1(G_p,F^-T_\f^\dagger)\otimes_{\varpi_{2,1}^*}\cR_2}{  {\res}_{/{\rm bal}}(\delta(M_2^\dagger,{\rm tr}^*\Delta_{\g}))}\right)
{\rm char}\left( \dfrac{H^1(G_p,F^+T_\f^\dagger)}{\res_p(\delta(T_\f^\dagger,\Delta_\emptyset))}\otimes_{\varpi_{2,1}^*}\cR_2\right)\,, 
\end{align*}
\normalsize
as required.
\end{proof}

\subsubsection{} 
\label{subsubsec_factor_general_20_05_2021}
We are now ready to explain our factorisation result concerning algebraic $p$-adic $L$-functions, whenever the relevant leading terms are non-trivial. 

\begin{theorem}
\label{thm_main_8_4_4_factorisation}
Under the hypotheses recorded in \S\ref{subsubsec_hypo_section_6}, we have
\begin{equation}
\label{eqn_2022_09_13_1733}
    {\rm Log}_{\omega_\f}\left(\delta(T_2^\dagger,\Delta_{+})\right)=\delta(T_2^\dagger,\Delta_\g)\,{=}\,  {\rm Exp}_{F^-T_\f^\dagger}^*(\delta(M_2^\dagger,{\rm tr}^*\Delta_{\g}))\cdot \varpi_{2,1}^*{\rm Log}_{\omega_\f}(\delta(T_\f^\dagger,\Delta_\emptyset)).
\end{equation}
In particular, if $\delta(T_2^\dagger,\Delta_\g)\neq \{0\}$, then  ${\rm Exp}_{F^-T_\f^\dagger}^*(\delta(M_2^\dagger,{\rm tr}^*\Delta_{\g}))=\delta(M_2^\dagger,\Delta_{\rm bal})\neq \{0\} \neq {\rm Log}_{\omega_\f}(\delta(T_\f^\dagger,\Delta_\emptyset))$.
\end{theorem}

\begin{proof}
The first equality in \eqref{eqn_2022_09_13_1733} is \eqref{eqn_2022_09_16_1252}. 
Let us prove the second equality. Suppose first that $\delta(T_2^\dagger,\Delta_\g) = \{0\}$. We may assume without loss if generality that $\delta(M_2^\dagger,{\rm tr}^*\Delta_{\g}) \neq \{0\} \neq {\rm Log}_{\omega_\f}(\delta(T_\f^\dagger,\Delta_\emptyset))$, since otherwise the claimed equality \eqref{eqn_2022_09_13_1733} reduces to $0=0$. These assumptions imply that 
$\widetilde{H}_{\rm f}^1(G_{\QQ,\Sigma},T_2^\dagger,\Delta_{\g})$ is a non-trivial torsion-free $\cR_2$-module  by Remark \ref{rem:non-trivial_r=0_case}, 
that $\widetilde{H}_{\rm f}^1(G_{\QQ,\Sigma},M_2^\dagger,{\rm tr}^*\Delta_{\g})$ is torsion-free of rank $1$ by \eqref{eqn_2022_09_12_1543}, 
and that $\widetilde{H}_{\rm f}^1(G_{\QQ,\Sigma},T_\f^\dagger,\Delta_{0})=0$ by Proposition \ref{prop_suport_range_Selmer_complexes_8_6_2_2}(iii). 
Hence, it follows from the exact triangle \eqref{eqn_sequence_dual_trace_derived_category} that the morphism $\widetilde{H}_{\rm f}^1(G_{\QQ,\Sigma},T_2^\dagger,\Delta_{\g}) \hookrightarrow \widetilde{H}_{\rm f}^1(G_{\QQ,\Sigma},M_2^\dagger,{\rm tr}^*\Delta_{\g})$ is  injective and its cokernel is a torsion $\cR_2$-module. 
Consequently, by \eqref{diagram:connectinghom-M2trDelta_g} and \eqref{eq:res_bal_factrisation}, 
the homomorphism $\res_{/{\rm bal}}$ is zero. 
In particular, 
\[
{\rm Exp}_{F^-T_\f^\dagger}^*(\delta(M_2^\dagger,{\rm tr}^*\Delta_{\g})) = {\rm char}\left( \dfrac{H^1(G_p,F^-T_\f^\dagger)\otimes_{\varpi_{2,1}^*}\cR_2}{  {\res}_{/{\rm bal}}(\delta(M_2^\dagger,{\rm tr}^*\Delta_{\g}))}\right) = \{0\}. 
\]
This concludes the proof of \eqref{eqn_2022_09_13_1733} when $\delta(T_2^\dagger,\Delta_\g)= \{0\}$. Next, suppose that $\delta(T_2^\dagger,\Delta_\g) \neq \{0\}$. It follows on combining \eqref{eqn_sequence_dual_trace_derived_category} and Proposition~\ref{prop_suport_range_Selmer_complexes_8_6_2_8} that the sequence
\begin{align*}
\begin{aligned}
    0\rightarrow \widetilde{H}^1_{\rm f}(G_{\QQ,\Sigma},M_2^\dagger,{\rm tr}^*\Delta_{\g})\xrightarrow{\delta^1} \widetilde{H}^2_{\rm f}(G_{\QQ,\Sigma},T_\f^\dagger,\Delta_{0})\otimes_{\varpi_{2,1}^*}\cR_2\rightarrow \widetilde{H}^2_{\rm f}(G_{\QQ,\Sigma},T_2^\dagger,\Delta_{\g})\rightarrow \widetilde{H}^2_{\rm f}(G_{\QQ,\Sigma},M_2^\dagger,{\rm tr}^*\Delta_{\g})\rightarrow 0
\end{aligned}
\end{align*}
is exact. This shows that 
\begin{align}
    \label{eqn_eqn_sequence_dual_trace_cohomology_first_consequence}
    \begin{aligned}
        {\rm char} &\left(\widetilde{H}^1_{\rm f}(G_{\QQ,\Sigma},M_2^\dagger,{\rm tr}^*\Delta_{\g})\big{/}\delta(M_2^\dagger,{\rm tr}^*\Delta_{\g})\right)\,{\rm char}\left(\widetilde{H}^2_{\rm f}(G_{\QQ,\Sigma},T_2^\dagger,\Delta_{\g})\right)\\
        &\qquad\qquad\qquad={\rm char}\left(\widetilde{H}^2_{\rm f}(G_{\QQ,\Sigma},T_\f^\dagger,\Delta_{0})\big{/}\delta^1(\delta(M_2^\dagger,{\rm tr}^*\Delta_{\g}))\right) {\rm char}\left(\widetilde{H}^2_{\rm f}(G_{\QQ,\Sigma},M_2^\dagger,{\rm tr}^*\Delta_{\g})\right)\,.
    \end{aligned}
\end{align}
Combining \eqref{eqn_eqn_sequence_dual_trace_cohomology_first_consequence} with \eqref{eqn_2022_09_12_1543}, we conclude that
\begin{equation}
    \label{eqn_eqn_sequence_dual_trace_cohomology_second_consequence}
     {\rm char} \left(\widetilde{H}^2_{\rm f}(G_{\QQ,\Sigma},T_2^\dagger,\Delta_{\g})\right)={\rm char}\left(\frac{\widetilde{H}^2_{\rm f}(G_{\QQ,\Sigma},T_\f^\dagger,\Delta_{0})\otimes_{\varpi_{2,1}^*}\cR_2}{\delta^1(\delta(M_2^\dagger,{\rm tr}^*\Delta_{\g}))}\right)\,.
\end{equation}
This equality together with Proposition~\ref{prop_useful_step_in_factorisation_21_05_2021} then yields
\begin{align}\label{eqn_eqn_sequence_dual_trace_cohomology_second_consequence_2}
    \begin{aligned}
    {\rm char}\left(\widetilde{H}^2_{\rm f}(G_{\QQ,\Sigma},T_2^\dagger,\Delta_{\g})\right)\,&={\rm char}\left( \dfrac{H^1(G_p,F^-T_\f^\dagger)\otimes_{\varpi_{2,1}^*}\cR_2}{  {\res}_{/{\rm bal}}(\delta(M_2^\dagger,{\rm tr}^*\Delta_{\g}))}\right) {\rm char}\left( \dfrac{H^1(G_p,F^+T_\f^\dagger)}{\res_p(\delta(T_\f^\dagger,\Delta_\emptyset))}\otimes_{\varpi_{2,1}^*}\cR_2\right) \\    
    &={\rm Exp}_{F^-T_\f^\dagger}^*(\delta(M_2^\dagger,{\rm tr}^*\Delta_{\g})) \varpi_{2,1}^*{\rm Log}_{\omega_\f}(\delta(T_\f^\dagger,\Delta_\emptyset))\,.
     \end{aligned}
\end{align}
The proof of our theorem is now complete on combining  \eqref{eqn_eqn_sequence_dual_trace_cohomology_second_consequence_2} with  \eqref{eqn_2022_09_12_1711}.
\end{proof}

\begin{corollary}\label{cor_2026_05_11}
Under the hypotheses recorded in \S\ref{subsubsec_hypo_section_6}, we have
\begin{equation*}
\delta(T_2^\dagger,\Delta_\g)\,{=}\,  {\rm Exp}_{F^-T_\f^\dagger}^*(\delta(M_2^\dagger,{\rm tr}^*\Delta_{\g}))\cdot \varpi_{2,1}^*{\rm Log}_{\omega_\f}({\rm BK}_\f^\dagger).
\end{equation*}
\end{corollary}
\begin{proof}
This corollary follows from Theorem \ref{thm_main_8_4_4_factorisation} together with \eqref{eqn_2022_09_12_1953_bis}. 
\end{proof}

\subsubsection{}
    \label{rem_2022_09_13_1624}
    As we have noted in \S\ref{subsubsec_leading_terms_bis_1} and \S\ref{subsubsec_leading_terms_bis_2}, we expect (in view of Iwasawa--Greenberg main conjectures) that the module $\delta(T_2^\dagger,\Delta_\g)$ of leading terms is generated by $\cL_p^{\hg}(\f\otimes\g\otimes\g^c)^2(\kappa,\lambda,\lambda)$, whereas the module ${\rm Exp}_{F^-T_\f^\dagger}^*(\delta(M_2^\dagger,{\rm tr}^*\Delta_{\g}))$ is expected to be generated by the degree-6 balanced $p$-adic $L$-function $\cL_p^\Ad(\hf\otimes \Ad^0\hg)$. 
    
    Moreover, recall that $\delta(T_\f^\dagger,\Delta_\emptyset)$ can be chosen as the Beilinson--Kato element (constructed by Ochiai in this setup). 
    Very loosely speaking, the element ${\rm Log}_{\omega_\f}(\delta(T_\f^\dagger,\Delta_\emptyset))\in \cR_\f$ then interpolates the derivatives of Perrin-Riou's $\mathbf{D}_{\rm cris}(V_f^*)\otimes \mathcal{H}(\Gamma_\cyc)$-valued $p$-adic $L$-functions attached to crystalline members $f$ of the family $\f$ at their central critical points, paired with the canonical vector $\omega_f\in {\rm Fil}^1\mathbf{D}_{\rm cris}(V_f)$ (cf. \cite{perrinriou93}, \S2.2.2). 
    
    The precise relation of $\delta(T_\f^\dagger,\Delta_\emptyset)$  with $p$-adic $L$-functions will be explained in \S\ref{BK-PR}, and it will be clear that \eqref{eqn_2022_09_13_1733} is indeed an algebraic analogue of the factorisation predicted by Conjecture~\ref{conj_main_6_plus_2}.

\subsubsection{}
\label{BK-PR}
Let us assume that $\kappa$ is a crystalline specialisation and $f=f_\kappa$. Let $f_0$ denote the newform associated to $f$, and let us denote by $\alpha_0,\beta_0$ the roots of its Hecke polynomial at $p$. We assume that $v_p(\alpha_0)=0$ and $f_0^{\beta_0}$ is non-theta-critical.  Combining \cite[\S2.2.2]{perrinriou93} with Kato's reciprocity laws \cite[Theorem 16.6]{kato04} (as enhanced in \cite{bb_CK1} to cover the case of critical-slope $p$-adic $L$-functions), \cite[Proposition 10.1.1(1)]{KLZ2}, and the interpolative properties of the big logarithm map given as in \eqref{eqn_2022_09_13_0843}, we deduce that
\begin{equation}
    \label{eqn_2022_09_20_1257}
    {\rm Log}_{\omega_\f}({\rm BK}_{\f}^\dagger) (\kappa)=d_\alpha L_{p,\alpha_0}^\prime(f_0,{w(\kappa)}/{2}) +d_\beta L_{p,\beta_0}^\prime(f_0,{w(\kappa)}/{2})\,,\qquad d_\alpha,d_\beta\in \overline{\QQ}^\times\,.
\end{equation}


\appendix
\section{A strategy to prove Conjecture~\ref{conj_main_6_plus_2}}
\label{subsec_strategy_to_prove_conj_2_2}
We outline an action plan to prove a weak form\footnote{This form does not require the $p$-optimal construction of the family of degree-6 $p$-adic $L$-functions $\cL_p^\Ad(\hf\otimes \Ad^0\hg)$.} of Conjecture~\ref{conj_main_6_plus_2} in full generality. We remark that this strategy has been carried out in detail in the follow-up paper \cite{BCPV}, and we refer the reader to this paper for details.

\subsection{Reformulation of  Conjecture~\ref{conj_main_6_plus_2}} 
\label{subsec_A1_2024_11}
To prove Conjecture~\ref{conj_main_6_plus_2}, one needs to verify the following:
\begin{itemize}
    \item[\mylabel{item_211}{\bf 2.1.1}.] The ratio $\mathscr{L}:=\dfrac{\cL_p^\hg(\hf\otimes\hg\otimes\hg^c)^2(\kappa,\lambda,\lambda)}{\mathscr C(\kappa)\cdot {\rm Log}_{\omega_\f}({\rm BK}_{\f}^\dagger)} \in {\rm Frac}(\cR_2)$ interpolates the algebraic numbers on the right-side of \eqref{eqn_conj_2022_06_02_0940} (with $?=\Ad$).  
    \item[\mylabel{item_212}{\bf 2.1.2}.] $\mathscr{L} \in \cR_2$.
\end{itemize}

Once (i) is verified, the verification of (ii) follows once the $p$-adic $L$-function $\cL_p^\Ad(\hf\otimes \Ad^0\hg)$ as in Conjecture~\ref{conj_2022_06_02_0940} has been constructed. The weak form of Conjecture~\ref{conj_main_6_plus_2} alluded to above is \eqref{item_211} above.

Thanks to the fundamental results in \cite{BDV, DarmonRotger}, the numerator of \eqref{item_211} can be computed (up to explicit fudge factors) in terms of the square of the image of the diagonal cycles under the large Perrin-Riou exponential. Moreover, via a family-version of Perrin-Riou's conjecture (cf. \cite{BC_Part2}, see also \cite{BCPV}, \S2.2.8),  the term ${\rm Log}_{\omega_\f}({\rm BK}_{\f}^\dagger)$ in the denominator can be related to the square of the logarithm of a big Heegner point. We remark that Heegner points are defined over an imaginary quadratic field $K$ where $p$ splits and that verifies the Heegner hypothesis relative to $\f$ (so that we have a ring isomorphism $\cO_K/\mathfrak{N}\simeq \ZZ/N_\f\ZZ$ for some ideal $\mathfrak{N}$ of $\cO_K$). We henceforth fix such $K$. A natural recourse\footnote{This action plan is, of course, akin to the strategy of  (Gross and) Dasgupta we outlined in \S\ref{subsubsec_intro_comparison_palvannan_dasgupta}. One fundamental difference is that one needs to compare cycles, as opposed to elements in the motivic cohomology groups of non-critical motives (or rather, their realisations in Deligne cohomology). Since these cycles come about to explain the derivatives of $L$-functions at the central critical point, one needs to work with complex analytic heights in the present setting.} to prove Conjecture~\ref{conj_main_6_plus_2}, therefore, is to establish a comparison between diagonal cycles and Heegner cycles. Such a relationship will come about as a consequence of Gross--Kudla conjecture and the Gross--Zagier--Zhang formulae for Heegner cycles. We note that the diagonal cycles and Heegner points live in different spaces (the former in the Chow group of codimension-$2$ cycles on the product of $3$ modular curves, the latter in the Jacobian of a modular curve), hence they have no chance to admit a direct comparison. It is the ``Chow--Heegner points'' of Darmon--Rotger--Sols~\cite{DRS} (which are the images of diagonal cycles under an intersection pairing) that play an intermediary role; see \S\ref{subsec_A2_2024_11} below for their key properties.

From this standpoint, \eqref{item_211} can reformulated as follows\footnote{That \eqref{item_211bis} is equivalent to \eqref{item_211} has been explained in the follow-up paper \cite[\S2.2]{BCPV} to the present work.}. Let us denote by ${\mathscr{R}}_{\kappa,\lambda}\eqref{eqn_conj_2022_06_02_0940}$ the quantity on the right-hand-side of \eqref{eqn_conj_2022_06_02_0940} with $?=\Ad$ for fixed arithmetic specialisations $\kappa$ and $\lambda$.
\begin{itemize}
    \item[\mylabel{item_211bis}{$\mathbf{2.1.1^\prime}$}.] The ratio \, $\dfrac{\log_{\omega_{\f_\kappa}}\left(\Delta^{\rm (tr)}_p(\f\otimes{\rm ad}(\g))_{\vert_{\kappa,\lambda}}\right)^2\, p^{2s_\kappa-2} [L_{s_\kappa}:H_{p^{s_\kappa}}]^{2}}{\mathscr{D}(\kappa)\cdot   a_p(\f_\kappa)^{-2s_\kappa}\cdot\log_{\omega_{\f_\kappa}}(Q_\kappa)^2}$ equals ${\mathscr{R}}_{\kappa,\lambda}\eqref{eqn_conj_2022_06_02_0940}$ for infinitely many specialisations $(\kappa,\lambda)$ of $\cR_2$ of weight $(2,2)$.
\end{itemize}
We will explain below in \S\ref{subsubsec_whyweight2_GK} why we work exclusively with specialisations of weight $2$. In \eqref{item_211bis}:
\begin{itemize}
    \item $\Delta^{\rm (tr)}_p(\f\otimes{\rm ad}(\g)) \in H^1(\QQ_p,T_\f^\dagger)\widehat{\otimes}_{\cR_\f} \cR_2$ is the image of the big diagonal cycle of \cite{DarmonRotger} under the morphism induced from
    $$H^1(\QQ,T_2^\dagger)\xrightarrow{\res_p} H^1(\QQ_p,T_2^\dagger) \xrightarrow{{\rm tr}} T_\f^\dagger \widehat{\otimes}_{\cR_\f} \cR_2\,;$$
    cf. \cite[\S3.1.2]{BCPV}, and $\Delta^{\rm (tr)}_p(\f\otimes{\rm ad}(\g))_{\vert_{\kappa,\lambda}} \in H^1(\QQ_p,T_{\f_\kappa}^\dagger)$ is its specialisation.
    \item $p^{s_\kappa}$ is the conductor of the wild-nebentype $\psi_\kappa$ of $\f_\kappa$ (as well as its level at $p$), $H_{p^{s_\kappa}}/K$ is the ring class field of conductor $p^{s_\kappa}$, and $L_{s_\kappa}=H_{p^{s_\kappa}}(\mu_{p^{s_\kappa}})$.
    \item $Q_\kappa\in H^1(K,T_{\f_\kappa}^\dagger)$ is Howard's twisted Heegner point given as in \cite[\S2.2.2]{BCPV}.
    \item Finally, 
    $$\mathscr{D}(\kappa):=\dfrac{2^{2\wt(\kappa)-1} (-1)^{\wt(\kappa)+1} }{\cL_p^{\rm Kit}(\f\otimes\epsilon_K)(\kappa, \frac{\rmw(\kappa)}{2})}\times\dfrac{\mathfrak C_{\rm exc}(\hf\otimes\hg\otimes \hh) }{ \pmb\xi_\kappa(\mathfrak N^{-1}) (\sqrt{-1})^{\frac{\wt(\kappa)}{2}-1} \sqrt{-D_K} }=\dfrac{-8\cdot  \mathfrak C_{\rm exc}(\hf\otimes\hg\otimes \hh)}{\cL_p^{\rm Kit}(\f\otimes\epsilon_K)(\kappa,1)\cdot \pmb\xi_\kappa(\mathfrak N^{-1}) \sqrt{-D_K}}$$ 
    is the specialisation of $\mathscr{D}\in {\rm Frac}(\cR_2)$, where we use that $\wt(\kappa)=2$ in \eqref{item_211bis}, and: 
    \begin{itemize}
    \item $\pmb\xi$ is the universal anticyclotomic character given as in \cite[Definition 2.8(3)]{CastellapadicvariationofHeegnerpoints}, and $\pmb\xi_\kappa$ is its specialisation, where we treat $\kappa$ as an element of the anticyclotomic weight space as in op. cit. (identification of one with two is given by the canonical anticyclotomic character, which gives the action on the module denoted by $\QQ_p\langle 1 \rangle$ in \cite[\S4.2]{kobayashiota_Iwasawa2017}).
    \item $\mathfrak C_{\rm exc}(\hf\otimes\hg\otimes \hh)$ is as in the statement of Theorem~\ref{thm:unbalancedinterpolation}.
      \end{itemize}    
  \end{itemize}  

Note that we implicitly assume that $\log_{\omega_{\f_\kappa}}(Q_\kappa)\neq 0$ for all but finitely many $(\kappa,\lambda)$ as above. While this is conjectured by Greenberg to be always the case, we may assume without loss of generality that this is indeed the case; cf. \cite[\S7.2.1]{BCPV}.

\subsubsection{}
The discussion in Section~\ref{subsec_A1_2024_11} tells us that one may directly formulate Conjecture~\ref{conj_main_6_plus_2} in terms of the logarithm of a big Heegner point (relative to an auxiliary choice of an imaginary quadratic field) rather than the big Beilinson--Kato class. However, based on our analysis with the algebraic $p$-adic $L$-functions and modules of leading terms in \S\ref{subsec_KS}, we were led to believe that it is more natural to formulate this conjecture in terms of the logarithm of a big Beilinson--Kato element.

\subsection{Reformulation of \texorpdfstring{\eqref{item_211bis}}{} in terms of Chow--Heegner points}
\label{subsec_A2_2024_11}
We retain the notation in \S\ref{subsec_A1_2024_11}. Let us fix throughout this subsection $(\kappa,\lambda)$ as in \eqref{item_211bis}, let us put $f=\f_\kappa \in S_{2}(\Gamma_0(N_\f)\cap \Gamma_1(p^{s_\kappa}), \psi_\kappa)$ and $g=\g_\lambda\in S_{2}(\Gamma_0(N_\g)\cap \Gamma_1(p^{s_\kappa}))$ to ease our notation. Let us denote by $J$ the Jacobian of the compactified modular curve of level $\Gamma_0(N_\f)\cap \Gamma_1(p^{s_\kappa})$. 

Following \cite[\S1.4]{YZZ10} and  \cite[\S2]{DRS}, we define the Chow--Heegner point $P^{\circ}(\kappa,\lambda)\in (J(L_{s_\kappa})\otimes_{\ZZ}\overline{\QQ})^{\psi_{\kappa}^{\frac{1}{2}}}$ as in \cite[\S5.2.1]{BCPV}. By a slight abuse, we denote the image of $P^{\circ}(\kappa,\lambda)$ in $H^1_{\rm f}(L_{s_\kappa}, T_{\f_\kappa}^\dagger[\frac{1}{p}])^{\psi_{\kappa}^{\frac{1}{2}}}\simeq H^1_{\rm f}(K,T_{\f_\kappa}^\dagger[\frac{1}{p}])$ under the Kummer map also by $P^{\circ}(\kappa,\lambda)$, where the isomorphism follows from the inflation-restriction sequence (cf. \cite{BCPV}, \S7.2.1). As explained in Corollary 6.3 of op. cit. (based on the ideas in Daub's thesis~\cite{daubthesis}), we have 
$$\log_{\omega_f}(\Delta^{\rm (tr)}_p(\f\otimes{\rm ad}(\g))_{\vert_{\kappa,\lambda}})=a_p(g)^{s_\kappa}\,  \log_{\omega_f}\circ\,\res_p\,(P^{\circ}(\kappa,\lambda))\,.$$
We may therefore reformulate \eqref{item_211bis} as follows:

\begin{itemize}
    \item[\mylabel{item_211bisbis}{$\mathbf{CH\,vs\, H}$}.] We have $$P^{\circ}(\kappa,\lambda)=\pm a_p(g)^{-\frac{s_\kappa}{2}}a_p(f)^{-s_\kappa} p^{1-s_\kappa} [L_{s_\kappa}:H_{p^{s_\kappa}}]^{-1}{\mathscr{R}}_{\kappa,\lambda}\eqref{eqn_conj_2022_06_02_0940}^{\frac{1}{2}}\mathscr{D}(\kappa)^{\frac{1}{2}} \cdot Q_\kappa \qquad \hbox{ in }\quad J(L_{s_\kappa})\otimes_{\ZZ_p}\overline{\QQ}_p$$ 
   for some choice of $\pm\in \{+,-\}$.
    \end{itemize}

 We refer the reader to \cite[Proposition 6.5]{BCPV} where it is explained that \eqref{item_211bisbis} is indeed equivalent to \eqref{item_211bis}.

\subsection{Generalized Gross--Kudla conjecture and heights of Chow--Heegner points}
\label{subsec_A3_2024_11}
We retain the notation and conventions of \S\ref{subsec_A1_2024_11} and \S\ref{subsec_A2_2024_11}. In order to prove \eqref{item_211bisbis}, the key ingredients are:
\begin{itemize}
    \item[\mylabel{item_ingredient_1}{$\mathbf{I.1}$}.] The twisted Gross--Zagier theorems of Howard~ \cite[Theorem 5.6.2]{howardtwisted2009} (see also \cite{BCPV}, Theorem 7.1), expressing the derivative $L'(f,\psi_{\kappa}^{-\frac{1}{2}},1)$ in terms of the N\'eron--Tate height of $Q_\kappa\in J(L_{s_\kappa})$.
    \item[\mylabel{item_ingredient_2}{$\mathbf{I.2}$}.] The Generalized Gross--Kudla conjecture (\cite{BCPV}, Conjecture~5.1) expressing the derivative $L'(f\otimes g \otimes g^c, \psi_{\kappa}^{-\frac{1}{2}},2)$ in terms of Bloch--Beilinson heights of diagonal cycles. 
    \item[\mylabel{item_ingredient_3}{$\mathbf{I.3}$}.] Comparison of the Bloch--Beilinson heights of diagonal cycles and the N\'eron--Tate heights of the twisted Heegner points $Q_\kappa$ (cf. \cite[\S1.3.1]{YZZ12}; see also \cite[Remark 3.1.1]{YZZ23} and \cite{BCPV}, \S5.2.1).
\end{itemize}

\subsubsection{}
\label{subsubsec_obstacles} 
 The only missing ingredient to complete the proof of \eqref{item_211bisbis}, therefore, is \eqref{item_ingredient_2}. This is the subject of the preprints of Yuan--Zhang--Zhang \cite{YZZ10,YZZ12,YZZ23}. Note that the results of these articles do not cover the required level of generality (as they require that the levels of all factors are square-free) to establish \cite[Conjecture~5.1]{BCPV}. However, it is expected that a generalisation of the results in these preprints by Liu--Yuan--Zhang--Zhang (as announced by S. Zhang) will settle\footnote{Note that we need this result only in a degenerate case (i.e. when $h=g^c$). In this scenario, one can alternatively attempt to establish \eqref{item_ingredient_2} using the results of~\cite{Xue_Hang_2019} (see especially the discussion in the final paragraph of \S1.1 in op. cit.), but we have not pursued this. We remark that \cite{Xue_Hang_2019} there are no restrictions on the levels since one works in op. cit. over the generic fibre.} \eqref{item_ingredient_2}, and together with the circle of ideas sketched in this Appendix (as detailed in our companion article \cite{BCPV}), it will complete the proof of \eqref{item_211}.

\subsubsection{} 
\label{subsubsec_whyweight2_GK}
In \eqref{item_211bis}, we work only with specialisations of weight $2$. This is because, if $\wt(\kappa)>2$, the analogous comparison to what is predicted by \eqref{item_211bisbis} would take place in the Chow group of homologically trivial codimension-$\frac{\wt(\kappa)}{2}$ cycles on the suitable Kuga--Sato variety (rather than the Jacobian of the appropriately defined modular curve). Even if the higher-weight generalisation of the Gross--Kudla conjecture were available (which currently seems out of reach), one would also need the injectivity of a $p$-adic Abel--Jacobi map (which is possibly a much harder problem than Conjecture~\ref{conj_main_6_plus_2}).

    \subsubsection*{Data Availability \& Conflict of Interest Statement} On behalf of all authors, the corresponding author states that the manuscript has no associated data, and that there is no conflict of interest.

\bibliographystyle{amsalpha}
\bibliography{references}
\end{document}